\documentclass[11pt]{article}

\usepackage[margin=1in]{geometry}
\usepackage{amsmath,amssymb,amsthm,mathtools,mathrsfs}
\usepackage{booktabs}
\usepackage{tabularx}
\usepackage{longtable}
\usepackage{enumitem}
\usepackage{graphicx}
\usepackage{float}
\usepackage{algorithm}
\usepackage[noend]{algpseudocode}
\algrenewcommand\algorithmicrequire{\textbf{Input:}}
\algrenewcommand\algorithmicensure{\textbf{Output:}}
\usepackage[hypcap=false]{caption}
\usepackage{microtype}
\usepackage{placeins}
\usepackage{needspace}
\usepackage{environ}
\usepackage{xcolor}
\usepackage{tikz}
\usetikzlibrary{arrows.meta,calc,decorations.pathmorphing,positioning,shapes.geometric}
\usepackage[colorlinks=true,linkcolor=blue!55!black,citecolor=blue!55!black,urlcolor=blue!55!black]{hyperref}
\hypersetup{
  pdftitle={Barycentric Weak Inner-Product Gromov--Wasserstein},
  pdfauthor={Youssef Mroueh}
}

\newcommand{\R}{\mathbb R}
\newcommand{\E}{\mathbb E}
\newcommand{\cP}{\mathcal P}
\newcommand{\cU}{\mathcal U}
\newcommand{\cM}{\mathcal M}
\newcommand{\ip}[2]{\left\langle #1,#2\right\rangle}
\newcommand{\norm}[1]{\left\lVert #1\right\rVert}
\newcommand{\cx}{\preceq_{\mathrm{cx}}}
\newcommand{\wIGW}{\mathrm{wIGW}_{\mathrm{bar}}}
\newcommand{\wIGWeps}{\mathrm{wIGW}_{\mathrm{bar},\varepsilon}}
\newcommand{\IGW}{\mathrm{IGW}}

\newcommand{\WGW}{\mathrm{WGW}}
\DeclareMathOperator{\Tr}{Tr}
\DeclareMathOperator{\KL}{KL}
\DeclareMathOperator*{\argmin}{arg\,min}
\DeclareMathOperator*{\argmax}{arg\,max}

\theoremstyle{plain}
\newtheorem{theorem}{Theorem}[section]
\newtheorem{proposition}[theorem]{Proposition}
\newtheorem{lemma}[theorem]{Lemma}
\newtheorem{corollary}[theorem]{Corollary}
\theoremstyle{definition}
\newtheorem{definition}[theorem]{Definition}

\theoremstyle{remark}
\newtheorem{remark}[theorem]{Remark}

\let\appendixproof\proof
\let\endappendixproof\endproof
\RenewEnviron{proof}[1][\proofname]{%
  \par\smallskip\noindent\emph{Proof.}
  The complete argument is given in
  \hyperref[proof:\thetheorem]{Appendix~\ref*{proof:\thetheorem}}.
  \par\smallskip
}

\title{Barycentric Weak Inner-Product Gromov--Wasserstein}
\author{Youssef Mroueh\\[0.3em]
\normalsize IBM Research\\
\normalsize\texttt{mroueh@us.ibm.com}}
\date{}

\begin{document}
\maketitle

\begin{abstract}
Gromov--Wasserstein (GW) compares distributions through relations within each space.
This pointwise comparison can be too sensitive in one-to-many settings, where several
target outcomes refine one source state and their mean carries the geometry of
interest. We introduce a weak GW framework that compares source
relations with relations between the target conditional laws induced by a coupling.
For inner-product relations, we retain the conditional means
\(m_\pi(x)=\E_\pi[Y\mid X=x]\). The resulting barycentric weak inner-product GW
(wIGW) satisfies
\[
  \wIGW^2(\mu,\nu)=\inf_{\eta\cx\nu}\IGW^2(\mu,\eta).
\]
Here \(\eta\cx\nu\) means that \(\nu\) is a mean-preserving spread of \(\eta\).
Thus wIGW searches for an intermediate target geometry that can be refined into the
prescribed target law without changing conditional means. Under finite second
moments, minimizers exist and martingale gluing recovers an optimal coupling. With
ridge regularization, moment duality gives an \(A\)--\(B\) min--max problem whose
inner step is weak optimal transport with a quadratic cost parameterized by \(A\) and
\(B\); the outer problem optimizes these matrices. For finitely supported measures,
we give an iterative algorithm. Under
a quantitative ridge condition, the reduced problem is
convex--concave, and the projected outer iteration satisfies an explicit
contraction bound for inexact inner solves. Point cloud and graph feature refinement
experiments illustrate how mean-preserving target refinements can have zero cost.
A paired peripheral blood mononuclear cell (PBMC) multiome study evaluates
atlas based cell type transfer through RNA/ATAC alignment in cell to cell and
prototype to cell settings, with the prototype to cell setting representing the
one-to-many case.
\end{abstract}

\section{Introduction}\label{sec:introduction}

Gromov--Wasserstein (GW) compares probability spaces through their internal pairwise
relations \cite{Memoli,PeyreCuturiSolomon}. Inner-product GW (IGW) uses
\[
  c_{\mathcal X}(x,x')=\ip{x}{x'},
  \qquad c_{\mathcal Y}(y,y')=\ip{y}{y'},
\]
as its relation functions. This permits comparisons of Euclidean embeddings and
graph features whose coordinate systems need not agree. For $p\ge1$, let
$\cP_p(\R^d)$ denote the Borel probability measures on $\R^d$ with finite
$p$th moment, and let $\Pi(\mu,\nu)$ denote the set of couplings of $\mu$ and
$\nu$. For $\mu\in\cP_2(\R^{d_x})$ and $\nu\in\cP_2(\R^{d_y})$, the ordinary
IGW objective is
\begin{equation}\label{eq:intro-ordinary-igw}
 \IGW^2(\mu,\nu)
 :=\inf_{\pi\in\Pi(\mu,\nu)}
 \iint\bigl(\ip{x}{x'}-\ip{y}{y'}\bigr)^2
 d\pi(x,y)d\pi(x',y').
\end{equation}
The OT envelope of the squared inner-product loss motivates the computational
construction below. Write
\[
 S_\mu:=\int xx^\top d\mu(x),\qquad
 S_\nu:=\int yy^\top d\nu(y).
\]
Then ordinary IGW has the optimal transport envelope
\begin{equation}\label{eq:intro-ordinary-igw-envelope}
 \IGW^2(\mu,\nu)
 =\norm{S_\mu}_F^2+\norm{S_\nu}_F^2
 +\min_{A\in\R^{d_x\times d_y}}
 \left\{2\norm A_F^2+\mathsf{OT}_{c_A}(\mu,\nu)\right\},
 \qquad c_A(x,y):=-4y^\top A^\top x,
\end{equation}
where
$\mathsf{OT}_{c_A}(\mu,\nu):=\inf_{\pi\in\Pi(\mu,\nu)}\int c_A\,d\pi$.
Thus ordinary IGW reduces to an outer optimization over one matrix $A$; each
evaluation solves ordinary OT with a bilinear cost
\cite{ZhangGoldfeldMrouehSriperumbudur}.

\paragraph{From IGW to conditional laws.}
The objective in \eqref{eq:intro-ordinary-igw} evaluates the loss on individual
target realizations. In a one-to-many correspondence, a coarse source state may
represent several refined outcomes; for example, a coarse graph vertex may split
into several vertices in a refined graph. In multimodal data, an atlas prototype
may likewise correspond to a heterogeneous population of target cells. Quantum
measurements provide another coarse to fine example: classical post-processing can
merge several outcomes of a positive operator-valued measure (POVM) into one
detector readout, while a quantum instrument also attaches a state update branch to
each outcome \cite{OreshkovCalsamiglia2009,LeppajarviSedlak2021}. Across these
examples, the conditional mean can carry
the geometry of the coarse state, while the conditional law describes its refined
outcomes.

For a coupling $\pi\in\Pi(\mu,\nu)$, write
$\pi(dx,dy)=\mu(dx)\pi_x(dy)$. Weak optimal transport assigns a general cost to
the entire conditional target law, whereas its barycentric specialization retains
only the conditional mean \cite{GRST,BBP}:
\[
\begin{aligned}
 &\inf_{\pi\in\Pi(\mu,\nu)}\int C(x,\pi_x)\,d\mu(x),
 &&\text{(weak OT)},\\
 &m_\pi(x):=\int y\,\pi_x(dy),\qquad
 \inf_{\pi\in\Pi(\mu,\nu)}\int c\bigl(x,m_\pi(x)\bigr)\,d\mu(x),
 &&\text{(barycentric weak OT)}.
\end{aligned}
\]
We lift this principle from pointwise costs to pairwise relations.

Given a source relation $c_{\mathcal X}$, a measurable relation
$D:\mathcal P(\mathcal Y)^2\to\R$ between two conditional laws, and a loss
$\mathcal L$, we define the following weak Gromov--Wasserstein discrepancy based
on conditional laws:
\begin{equation}\label{eq:intro-aggregated-wgw}
 \inf_{\pi\in\Pi(\mu,\nu)}
 \iint \mathcal L\!\left(c_{\mathcal X}(x,x'),
 D(\pi_x,\pi_{x'})\right)d\mu(x)d\mu(x').
\end{equation}
Figure~\ref{fig:pointwise-conditional-spaces} contrasts this construction with
ordinary GW. Section~\ref{sec:weak-gw-framework} gives the general conditional
cost framework, which contains \eqref{eq:intro-aggregated-wgw}; a separate
averaging specialization recovers ordinary GW exactly.

\newcommand{\pointmartingalefigure}{%
Every feasible conditional mean map can be realized by a coupling in two stages:
\[
  X\xrightarrow{\ m\ }Z=m(X)
  \xrightarrow{\ \kappa\ }Y,
  \qquad \E[Y\mid Z]=Z.
\]
The wIGW loss compares the relations between $X$ and $Z$, while the martingale
kernel $\kappa$ supplies the conditional variation needed to reproduce the target
marginal $\nu$. For an optimal $m$, this construction yields an optimal coupling.
Figure~\ref{fig:point-martingale} illustrates the realization.

\begin{figure}[ht!]
\centering
\resizebox{\textwidth}{!}{%
\begin{tikzpicture}[
  font=\small,
  >=Latex,
  panel/.style={rounded corners=3pt,draw=black!22,fill=black!1},
  source/.style={circle,fill=blue!72!black,draw=white,line width=.35pt,inner sep=2.8pt},
  bary/.style={circle,fill=green!55!black,draw=white,line width=.35pt,inner sep=2.8pt},
  child/.style={circle,fill=red!70!black,draw=white,line width=.3pt,inner sep=2.5pt},
  midpoint/.style={circle,fill=white,draw=green!55!black,line width=1pt,inner sep=2.15pt},
  skeleton/.style={line width=1pt},
  spread/.style={dashed,line width=.75pt,red!65!black},
  stage/.style={->,line width=1.05pt,black!65}
]
  \draw[panel] (0,0) rectangle (4.25,4.25);
  \draw[panel] (5.35,0) rectangle (9.60,4.25);
  \draw[panel] (10.70,0) rectangle (14.95,4.25);

  \node[font=\bfseries,anchor=north] at (2.125,4.05) {source points};
  \node[font=\bfseries,anchor=north] at (7.475,4.05) {wIGW pushforward};
  \node[font=\bfseries,anchor=north] at (12.825,4.05) {target refinement};

  \coordinate (x1) at (0.92,2.95);
  \coordinate (x2) at (2.70,3.20);
  \coordinate (x3) at (3.02,1.33);
  \coordinate (x4) at (1.28,1.02);
  \draw[skeleton,blue!34] (x1)--(x2)--(x3)--(x4)--cycle;
  \node[source,label=left:$x_1$] at (x1) {};
  \node[source,label=right:$x_2$] at (x2) {};
  \node[source,label=right:$x_3$] at (x3) {};
  \node[source,label=below:$x_4$] at (x4) {};
  \node[blue!65!black] at (2.125,.38) {$x_i\sim\mu$};

  \coordinate (z1) at (6.13,2.72);
  \coordinate (z2) at (7.82,3.02);
  \coordinate (z3) at (8.02,1.24);
  \coordinate (z4) at (6.38,1.02);
  \draw[skeleton,green!42!black] (z1)--(z2)--(z3)--(z4)--cycle;
  \node[bary,label=left:$z_1$] at (z1) {};
  \node[bary,label=above:$z_2$] at (z2) {};
  \node[bary,label=right:$z_3$] at (z3) {};
  \node[bary,label=below:$z_4$] at (z4) {};
  \node[green!45!black] at (7.475,.38) {$z_i=m(x_i)$};

  \coordinate (t1) at (11.48,2.72);
  \coordinate (t2) at (13.17,3.02);
  \coordinate (t3) at (13.37,1.24);
  \coordinate (t4) at (11.73,1.02);
  \draw[skeleton,densely dashed,green!34!black] (t1)--(t2)--(t3)--(t4)--cycle;
  \foreach \z/\ya/\yb in {
    t1/{11.18,3.04}/{11.78,2.40},
    t2/{12.82,2.82}/{13.52,3.22},
    t3/{13.03,1.42}/{13.71,1.06},
    t4/{11.44,.73}/{12.02,1.31}}
  {
    \draw[spread] (\ya)--(\yb);
    \node[child] at (\ya) {};
    \node[child] at (\yb) {};
    \node[midpoint] at (\z) {};
  }
  \node[red!65!black,anchor=west] at (11.80,2.34) {$y_{1,-}$};
  \node[red!65!black,anchor=west] at (11.12,3.18) {$y_{1,+}$};
  \node[green!45!black] at (12.825,.38)
    {$\E[Y\mid X=x_i]=z_i$};

  \draw[stage] (4.38,2.12)--(5.22,2.12)
    node[midway,above] {$m$};
  \draw[stage] (9.73,2.12)--(10.57,2.12)
    node[midway,above] {$\kappa(z_i,dy)$};
\end{tikzpicture}}
\caption{Point space view of wIGW. The mean map sends each source point $x_i$ to
$z_i$, and the wIGW loss compares pairwise inner products among the source points
with those among their images. A martingale kernel then distributes the mass at
each $z_i$ over target points while preserving its mean; the resulting mixture has
target marginal $\nu$.}
\label{fig:point-martingale}
\end{figure}
\FloatBarrier
}

\newcommand{\catcloudsfigure}{%
Figure~\ref{fig:cat-clouds} illustrates an exact zero wIGW certificate and its
numerical recovery. The rotation $z_i=R_{35^\circ}x_i$ preserves the source Gram
matrix, while the symmetric children $y_{i,\pm}=z_i\pm s_i v_i$ average to $z_i$.
Sending each pair to its coarse parent therefore gives zero wIGW, although ordinary
IGW remains positive because it compares the individual children. The third panel
shows the numerically recovered pushforward;
Section~\ref{sec:experiments} describes the solver and its diagnostics, and
Corollary~\ref{cor:noisy-isometry} formalizes the construction.

\begin{figure}[ht!]
\centering
\includegraphics[width=.86\linewidth]{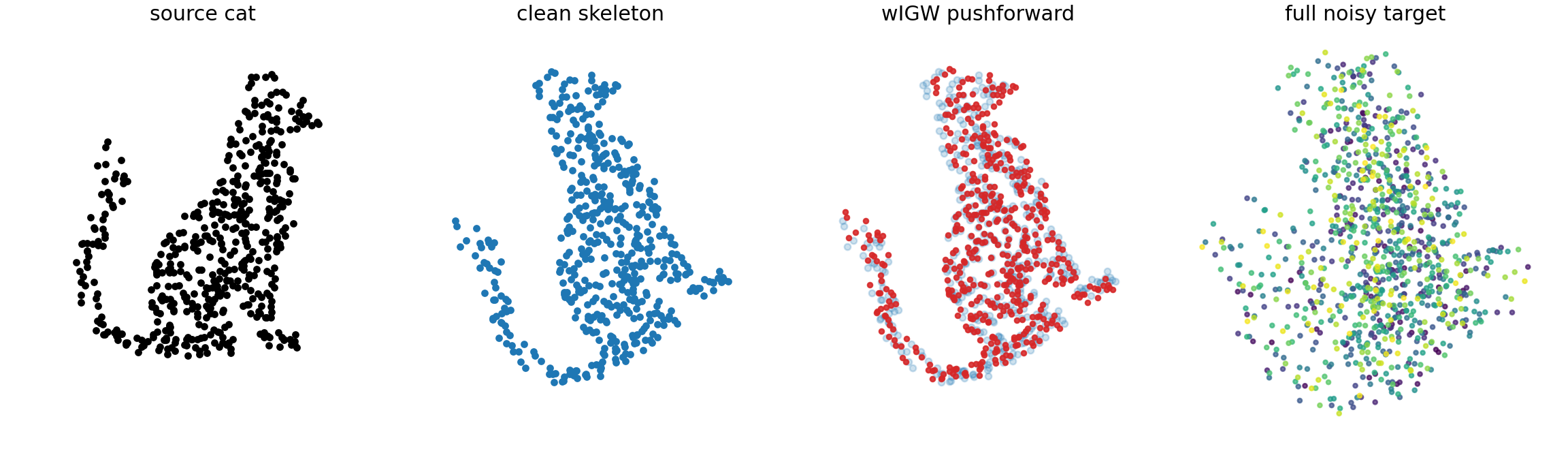}
\caption{A martingale refinement of the cat silhouette in
\cite{IBMUSD}. The panels show the centered source, its Gram preserving rotation,
the solved wIGW pushforward $(m_{\widehat\pi})_\#\mu$ overlaid on that reference,
and the full target with two children per parent. The returned coupling has the full
noisy target as its second marginal.}
\label{fig:cat-clouds}
\end{figure}
}

\newcommand{\conditionalcomparisonfigure}{%
\begin{figure}[H]
\centering
\resizebox{.82\textwidth}{!}{%
\begin{tikzpicture}[
  font=\small,
  >=Latex,
  panel/.style={rounded corners=3pt,draw=black!22,fill=black!1},
  xspace/.style={fill=blue!7,draw=blue!58!black,line width=.9pt},
  yspace/.style={fill=red!6,draw=red!58!black,line width=.9pt},
  xpoint/.style={circle,fill=blue!72!black,draw=white,line width=.35pt,inner sep=2.5pt},
  ypoint/.style={circle,fill=red!68!black,draw=white,line width=.35pt,inner sep=2.5pt},
  ambientx/.style={circle,fill=blue!42,inner sep=1.15pt},
  ambienty/.style={circle,fill=red!40,inner sep=1.15pt},
  row/.style={ellipse,draw=red!62!black,fill=red!9,line width=.9pt,
    minimum width=1.62cm,minimum height=.82cm},
  atom/.style={circle,fill=red!65!black,inner sep=1.15pt},
  relation/.style={<->,line width=.9pt},
  coupling/.style={->,line width=.95pt,orange!82!black},
  rowmap/.style={->,line width=1pt,blue!66!black}
]
\begin{scope}
  \draw[panel] (0,0) rectangle (7.35,4.95);
  \node[font=\bfseries,anchor=north] at (3.675,4.82) {(a) Ordinary GW};

  \path[xspace] plot[smooth cycle,tension=.82] coordinates{
    (.38,.92) (.38,2.07) (.82,3.56) (1.75,4.00)
    (2.86,3.52) (3.12,2.16) (2.72,.84) (1.44,.55)};
  \path[yspace] plot[smooth cycle,tension=.82] coordinates{
    (4.22,.88) (4.22,2.16) (4.72,3.54) (5.63,4.00)
    (6.78,3.47) (7.02,2.12) (6.62,.78) (5.32,.54)};
  \node[blue!65!black,font=\bfseries,anchor=east] at (.82,3.78) {$\mathcal X$};
  \node[red!65!black,font=\bfseries,anchor=east] at (4.86,3.78) {$\mathcal Y$};

  \foreach \p in {(.78,1.02),(1.12,2.18),(1.82,.86),(2.52,2.66),(2.57,1.20)}
    \node[ambientx] at \p {};
  \foreach \p in {(4.66,1.02),(4.95,2.21),(5.58,.84),(6.34,2.66),(6.44,1.18)}
    \node[ambienty] at \p {};

  \node[xpoint,label=left:$x$] (ox) at (1.08,3.03) {};
  \node[xpoint,label=left:$x'$] (oxp) at (2.33,1.72) {};
  \node[ypoint,label=right:$y$] (oy) at (4.91,3.03) {};
  \node[ypoint,label=right:$y'$] (oyp) at (6.17,1.72) {};
  \draw[relation,blue!70!black] (ox)--(oxp)
    node[midway,above=4pt] {$c_{\mathcal X}$};
  \draw[relation,red!70!black] (oy)--(oyp)
    node[midway,right=3pt] {$c_{\mathcal Y}$};
  \draw[coupling] (ox) to[bend left=8] (oy);
  \draw[coupling] (oxp) to[bend right=8] (oyp);
  \node[orange!82!black,font=\bfseries] at (3.68,3.48) {$\pi$};
  \node at (3.675,.22)
    {$\mathcal L\!\left(c_{\mathcal X}(x,x'),c_{\mathcal Y}(y,y')\right)$};
\end{scope}

\begin{scope}[xshift=7.75cm]
  \draw[panel] (0,0) rectangle (7.35,4.95);
  \node[font=\bfseries,anchor=north] at (3.675,4.82)
    {(b) Aggregated weak GW based on conditional laws};

  \path[xspace] plot[smooth cycle,tension=.82] coordinates{
    (.38,.92) (.38,2.07) (.82,3.56) (1.75,4.00)
    (2.86,3.52) (3.12,2.16) (2.72,.84) (1.44,.55)};
  \path[yspace] plot[smooth cycle,tension=.82] coordinates{
    (4.02,.70) (3.96,2.13) (4.55,3.65) (5.62,4.06)
    (6.86,3.58) (7.08,2.09) (6.58,.66) (5.16,.43)};
  \node[blue!65!black,font=\bfseries,anchor=east] at (.82,3.78) {$\mathcal X$};
  \node[red!65!black,font=\bfseries,anchor=east] at (4.38,3.82) {$\mathcal Y$};

  \foreach \p in {(.78,1.02),(1.12,2.18),(1.82,.86),(2.52,2.66),(2.57,1.20)}
    \node[ambientx] at \p {};
  \node[xpoint,label=left:$x$] (wx) at (1.08,3.03) {};
  \node[xpoint,label=left:$x'$] (wxp) at (2.33,1.72) {};
  \draw[relation,blue!70!black] (wx)--(wxp)
    node[midway,above=4pt] {$c_{\mathcal X}$};

  \node[row] (rowx) at (5.10,3.03) {};
  \node[row] (rowxp) at (5.94,1.55) {};
  \node[red!68!black,anchor=south] at (rowx.north) {$\pi_x$};
  \node[red!68!black,anchor=north] at (rowxp.south) {$\pi_{x'}$};
  \foreach \dx/\dy in {-.48/.08,-.19/-.14,.12/.12,.43/-.05}
    \node[atom] at ([xshift=\dx cm,yshift=\dy cm]rowx.center) {};
  \foreach \dx/\dy in {-.48/.05,-.18/-.13,.14/.12,.44/-.04}
    \node[atom] at ([xshift=\dx cm,yshift=\dy cm]rowxp.center) {};

  \draw[rowmap] (wx) to[bend left=7]
    node[midway,above=2pt] {$x\mapsto\pi_x$} (rowx.west);
  \draw[rowmap] (wxp) to[bend right=7]
    node[midway,below=2pt] {$x'\mapsto\pi_{x'}$} (rowxp.west);
  \draw[relation,red!70!black] (rowx.south east) to[bend left=18]
    (rowxp.north east);
  \node[red!70!black] at (6.34,2.45) {$D$};
  \node at (3.675,.22)
    {$\mathcal L\!\left(c_{\mathcal X}(x,x'),D(\pi_x,\pi_{x'})\right)$};
\end{scope}
\end{tikzpicture}}
\caption{Ordinary GW and aggregated weak GW based on conditional laws. Ordinary GW compares
$c_{\mathcal X}(x,x')$ with $c_{\mathcal Y}(y,y')$ for target points paired by the
coupling. The aggregated construction sends $x$ and $x'$ to the conditional laws $\pi_x$
and $\pi_{x'}$, then compares the source relation with $D(\pi_x,\pi_{x'})$. Here
$\mathcal L$ is the relation loss, and the choice of $D$ determines which
information from the conditional laws enters the comparison.}
\label{fig:pointwise-conditional-spaces}
\end{figure}
\FloatBarrier
}

\conditionalcomparisonfigure

\paragraph{Barycentric wIGW and convex order.}
The paper focuses on the barycentric relation
\[
 D_{\rm bar}(\rho,\rho')
 :=\ip{b(\rho)}{b(\rho')},
 \qquad b(\rho):=\int y\,d\rho(y),
\]
for conditional laws with finite first moments. For a coupling $\pi$, let
\[
 m_\pi(x):=b(\pi_x)=\E_\pi[Y\mid X=x].
\]
Because $\nu$ has a finite second moment, this mean is defined for
$\mu$ almost every $x$.
With $c_{\mathcal X}(x,x')=\ip{x}{x'}$ and
$\mathcal L(a,b)=(a-b)^2$, we define the barycentric weak inner-product GW
discrepancy (wIGW) by
\begin{equation}\label{eq:intro-wigw}
  \wIGW^2(\mu,\nu)
  :=\inf_{\pi\in\Pi(\mu,\nu)}
  \iint\left(\ip{x}{x'}-\ip{m_\pi(x)}{m_\pi(x')}\right)^2
  d\mu(x)d\mu(x').
\end{equation}
We write \(\wIGW(\mu,\nu)\) for the nonnegative square root of the value in
\eqref{eq:intro-wigw}.
The wIGW objective therefore depends on each conditional law $\pi_x$ through its
mean $m_\pi(x)$. Appendix~\ref{app:general-wgw-geometry} discusses relations that
retain distributional spread or shape.

Convex order determines which mean maps can be induced by a coupling with target
marginal $\nu$. For
$\eta,\nu\in\cP_1(\R^{d_y})$, write $\eta\cx\nu$ when
$\int u\,d\eta\le\int u\,d\nu$ for every finite convex function $u$ of at most
linear growth. Conditional Jensen and Strassen's theorem \cite{Strassen} show,
up to $\mu$ almost everywhere equality, that the attainable conditional means are
precisely the maps $m\in L^2(\mu;\R^{d_y})$ satisfying $m_\#\mu\cx\nu$.

In equal dimensions, Gozlan and Juillet showed that quadratic barycentric weak OT
admits the Wasserstein projection formula
\begin{equation}\label{eq:intro-quadratic-wot}
 \mathsf T_2(\nu\mid\mu)
 :=\inf_{\pi\in\Pi(\mu,\nu)}
   \int\norm{x-m_\pi(x)}^2d\mu(x)
 =\inf_{\substack{\eta\in\cP_2(\R^d)\\\eta\cx\nu}}
   W_2^2(\mu,\eta),
 \qquad \mu,\nu\in\cP_2(\R^d),
\end{equation}
onto the measures below $\nu$ in convex order \cite{GozlanJuillet}. We apply this
characterization to \eqref{eq:intro-wigw} and use bilinearity of the inner product
and conditional Jensen to derive the map and IGW projection formulas
\begin{equation}\label{eq:intro-projection}
  \begin{aligned}
  \wIGW^2(\mu,\nu)
  &=\inf_{\substack{m\in L^2(\mu;\R^{d_y})\\m_\#\mu\cx\nu}}
    \iint\!\left(\ip{x}{x'}-\ip{m(x)}{m(x')}\right)^2d\mu(x)d\mu(x')\\
  &=\inf_{\substack{\eta\in\cP_2(\R^{d_y})\\\eta\cx\nu}}
    \IGW^2(\mu,\eta).
  \end{aligned}
\end{equation}
Quadratic weak OT minimizes $W_2$ over $\eta\cx\nu$, whereas the wIGW projection
uses IGW on the same feasible set. Because convex order is directional, wIGW is a
discrepancy. The
Variational Dominance Criterion provides another use of a directional stochastic
order constraint \cite{DomingoEnrichSchiffMroueh2023}.

For measures with finite second moments, the coupling, map, and projection
formulations of wIGW all admit minimizers.

\pointmartingalefigure

\paragraph{Variational form, duality, and computation.}
The projection identity \eqref{eq:intro-projection} describes the geometry of wIGW,
and its map formulation leads to a dual representation. For a feasible mean map
$m$, set
\[
 S_m:=\int m(x)m(x)^\top d\mu(x),
 \qquad M_m:=\int x m(x)^\top d\mu(x).
\]
Proposition~\ref{prop:moment} reduces the map objective to
$\norm{S_\mu}_F^2+\norm{S_m}_F^2-2\norm{M_m}_F^2$. Duality introduces matrices
$A$ and $B\succeq0$ for these two moments and a convex potential for the convex
order constraint. Under compact support, this gives the unregularized dual in
Theorem~\ref{thm:compact-dual}.

For general measures with finite second moments, we add the ridge penalty
$\varepsilon\int\norm{m_\pi(x)}^2d\mu(x)$ and denote the resulting value by
$\wIGWeps^2(\mu,\nu)$. The ridge makes the resulting cost coercive and
strongly convex in $z$. Proposition~\ref{prop:wot} uses this structure to establish
uniqueness of the minimizing barycentric map and equality between the weak OT
primal and its dual over convex potentials. For fixed $A$ and $B$, the weak OT cost is
\[
 c_{A,B}^{\varepsilon}(x,z)
 :=z^\top(2B+\varepsilon I_{d_y})z-4z^\top A^\top x.
\]
Write $M_2(\rho):=\int\norm z^2d\rho(z)$. Let $\mathcal A_2$ be the Frobenius
ball in $\R^{d_x\times d_y}$ with radius $\sqrt{M_2(\mu)M_2(\nu)}$, and let
$\mathcal B_2$ be the Frobenius ball of positive semidefinite $d_y\times d_y$
matrices with radius $M_2(\nu)$. Combining the moment dualities with weak OT duality, for every
$\varepsilon>0$ we obtain the following dual variational form:
\begin{equation}\label{eq:intro-ridge-ab-envelope}
\begin{aligned}
 \wIGWeps^2(\mu,\nu)
 &=\norm{S_\mu}_F^2+
   \inf_{A\in\mathcal A_2}\sup_{B\in\mathcal B_2}
   \left\{2\norm A_F^2-\norm B_F^2+
   \mathsf W_{A,B}^{\varepsilon}(\mu,\nu)\right\},\\
 \mathsf W_{A,B}^{\varepsilon}(\mu,\nu)
 &:=\inf_{\pi\in\Pi(\mu,\nu)}
   \int c_{A,B}^{\varepsilon}\bigl(x,m_\pi(x)\bigr)d\mu(x).
\end{aligned}
\end{equation}
We refer to the outer matrix optimization in
\eqref{eq:intro-ridge-ab-envelope} as the \(A\)--\(B\) min--max problem.
Equation~\eqref{eq:intro-ridge-ab-envelope} gives the computational form: the
outer problem optimizes $A$ and $B$, while the inner block evaluates a convex
barycentric weak OT problem. This parallels the ordinary IGW envelope
\eqref{eq:intro-ordinary-igw-envelope}, which uses one matrix and ordinary OT.
Theorem~\ref{thm:swap} gives the outer exchange when
$\varepsilon\ge2\lambda_{\max}(S_\mu)$ and $\varepsilon>0$; the convergence
analysis later uses the strict inequality
$\varepsilon>2\lambda_{\max}(S_\mu)$.
Sections~\ref{sec:ridge-duality}--\ref{sec:algorithm} give the full duality,
reconstruction, and algorithmic results.

\catcloudsfigure
\FloatBarrier

\paragraph{Relation to prior work.}
Weak optimal transport was introduced in \cite{GRST}; existence and duality were
developed in \cite{BBP}, and \cite{GozlanJuillet} studied the quadratic barycentric
projection. Following \cite{PatyChoneKramarz2022}, our finite weak OT oracle uses KL
mirror descent with Sinkhorn KL projections, but without entropy regularization of the
weak OT objective. The outer representation builds on \cite{ZhangGoldfeldMrouehSriperumbudur},
and the inexact oracle analysis follows the strategy of \cite{RiouxGoldfeldKato2024}.
SCOT uses GW for single-cell multiome alignment \cite{DemetciEtAl2022}.
Table~\ref{tab:prior-comparison} compares the conditional or distributional objects in
related constructions. In wIGW, the optimized coupling induces the laws $\pi_x$; the
coupling objective is generally nonconvex, and the discrepancy is directional.

\begin{table}[H]
\centering
\footnotesize
\renewcommand{\arraystretch}{1.05}
\begin{tabularx}{\textwidth}{@{}>{\raggedright\arraybackslash}p{.21\textwidth}
  >{\raggedright\arraybackslash}p{.30\textwidth}
  >{\raggedright\arraybackslash}X@{}}
\toprule
Construction & Conditional or distributional object & Relation to wIGW\\
\midrule
SCOT \cite{DemetciEtAl2022} & Cell graph geometries in different
single-cell modalities & Ordinary GW baseline for multimodal alignment.\\
CDOT \cite{ChungSongKimPark} & Conditional expectation operators induced by the
optimized coupling & Convex pseudometric based on intertwining distance operators.\\
MIRROR/SI-GW \cite{WangWangDing2026} & A cross-attention coupling and a learned
target geometry & Geometry learning with a fixed coupling.\\
WL meets GW \cite{ChenLimMemoliWanWang} & Fixed transition kernels of measure
Markov chains & Exogenous conditional laws supplied as transition kernels.\\
$Z$-GW \cite{BauerMemoliNeedhamNishino} & Fixed kernels valued in $Z$, including
distributional structure & Exogenous relational data valued in $Z$.\\
Semi-relaxed and linear GW \cite{VincentCuazEtAl,BeierBeinertSteidl} & A relaxed
marginal, or barycentric projections into a reference space & Different relaxation and
linearization mechanisms.\\
\bottomrule
\end{tabularx}
\caption{Selected related GW constructions and the origin of their conditional or
distributional objects.}
\label{tab:prior-comparison}
\end{table}

The distinction from these constructions is structural. Semi-relaxed GW changes a
marginal constraint, whereas wIGW retains the prescribed target marginal $\nu$ and
relaxes only the geometry carried by conditional means; equivalently, it projects IGW
onto laws $\eta\cx\nu$ and realizes the result by a martingale refinement. Linear GW
uses barycentric projections to represent previously computed transport plans, whereas
the barycentric map here is optimized as part of the discrepancy. The conditional laws
in WL meets GW and $Z$-GW are exogenous relational data, while the laws $\pi_x$ in
wIGW are induced by the coupling being optimized. CDOT instead builds a convex
pseudometric from conditional expectation and distance operators.

\FloatBarrier
\paragraph{Contributions.}
\begin{itemize}[leftmargin=2em]
\item We formulate weak GW using conditional laws induced by a coupling and recover
ordinary GW as a special case.
\item For barycentric inner-product relations, we derive the map and moment forms,
characterize zero discrepancy, prove existence of minimizers, and establish the
convex order projection identity.
\item We obtain compact and ridge dual formulas. The ridge formulation is an
$A$--$B$ min--max problem with a convex barycentric weak OT block and an explicit
reconstruction principle.
\item We give a finite sample algorithm and, under a quantitative ridge condition,
derive a contraction bound with explicit control of oracle error and give numerical
illustrations on shape and graph feature refinements.
\item We evaluate atlas based cell type transfer through RNA/ATAC alignment in
cell to cell and prototype to cell settings on the paired 10x Genomics PBMC
multiome dataset \cite{TenXPBMC}, with target labels excluded from the transport
optimization, and assess sensitivity to the ridge parameter.
\end{itemize}

\paragraph{Organization.}
Section~\ref{sec:variational-foundations} recalls the ordinary IGW envelope and the
weak OT theorem used later. Section~\ref{sec:weak-gw-framework} defines weak GW
based on conditional laws, and Sections~\ref{sec:wigw-definition}--\ref{sec:projection}
develop wIGW, its moment structure, and the convex order projection theorem.
Sections~\ref{sec:compact-duality}--\ref{sec:reconstruction} give compact and ridge
duality, equivalent formulations, and reconstruction. Section~\ref{sec:algorithm}
gives the finite algorithm and convergence analysis; Section~\ref{sec:experiments}
presents synthetic shape and graph feature refinement studies and the PBMC
atlas based transfer benchmark; and
Section~\ref{sec:conclusion} concludes. Appendix~\ref{app:variational-proofs}
contains the analytic foundations; Appendices~\ref{app:barycentric-proofs}--
\ref{app:algorithm-proofs} contain the geometric, duality, reconstruction, and
algorithmic proofs. Appendix~\ref{app:general-wgw-geometry} discusses alternative
relations between conditional laws, and Appendix~\ref{app:reference-theorems}
states the Strassen, Sion, and Danskin results used in the paper.
Appendix~\ref{app:experimental-configurations} lists the numerical
configurations, defines the PBMC evaluation metrics, and reports the secondary
paired cell retrieval analysis.

\section{Variational foundations for IGW and weak optimal transport}
\label{sec:variational-foundations}

\subsection{Ordinary IGW: OT envelope and computational precedent}
\label{subsec:ordinary-igw}

We begin with the ordinary IGW representation that motivates the weak theory.
For a Euclidean probability law $\rho$ with a finite moment of order $p$, write
\begin{equation}\label{eq:moment-notation}
 M_p(\rho):=\int\norm z^p d\rho(z).
\end{equation}
For
$\mu\in\cP_2(\R^{d_x})$, $\nu\in\cP_2(\R^{d_y})$, and
$\pi\in\Pi(\mu,\nu)$, set
\begin{equation}\label{eq:ordinary-igw-moment-notation}
 S_\mu=\int xx^\top d\mu(x),\qquad
 S_\nu=\int yy^\top d\nu(y),\qquad
 M_\pi=\int xy^\top d\pi(x,y),
\end{equation}
and define
\begin{equation}\label{eq:ordinary-igw-definition}
 \IGW^2(\mu,\nu)
 :=\inf_{\pi\in\Pi(\mu,\nu)}
 \iint\big(\ip{x}{x'}-\ip{y}{y'}\big)^2
 d\pi(x,y)d\pi(x',y').
\end{equation}
The unsquared symbol \(\IGW(\mu,\nu)\) denotes the nonnegative square root of
the value in \eqref{eq:ordinary-igw-definition}.
For $A\in\R^{d_x\times d_y}$, introduce the bilinear cost
\begin{equation}\label{eq:ordinary-igw-bilinear-cost}
 c_A(x,y):=-4y^\top A^\top x,
 \qquad
 \mathsf{OT}_{c_A}(\mu,\nu)
 :=\inf_{\pi\in\Pi(\mu,\nu)}\int c_A\,d\pi.
\end{equation}

\begin{theorem}[Ordinary IGW OT envelope and duality
\cite{ZhangGoldfeldMrouehSriperumbudur,RiouxGoldfeldKato2024}]
\label{thm:ordinary-igw-envelope}
Let $\mu\in\cP_2(\R^{d_x})$ and $\nu\in\cP_2(\R^{d_y})$. With the
moment notation, moment matrices, IGW objective, and bilinear OT cost defined in
\eqref{eq:moment-notation}--\eqref{eq:ordinary-igw-bilinear-cost},
\begin{equation}\label{eq:ordinary-igw-envelope}
 \IGW^2(\mu,\nu)
 =\norm{S_\mu}_F^2+\norm{S_\nu}_F^2
 +\min_{A\in\R^{d_x\times d_y}}
 \left\{2\norm A_F^2+\mathsf{OT}_{c_A}(\mu,\nu)\right\}.
\end{equation}
The minimum is unchanged if $A$ is restricted to
\[
 \mathcal A_{\mu,\nu}:=
 \left\{A\in\R^{d_x\times d_y}:\norm A_F\le\sqrt{M_2(\mu)M_2(\nu)}\right\}.
\]
Moreover, for every $A\in\R^{d_x\times d_y}$, set
\[
  \mathcal K_A:=\left\{(\varphi,\psi)\in L^1(\mu)\times L^1(\nu):
  \begin{array}{l}\varphi,\psi\text{ are selected Borel versions, and}\\[-2pt]
  \varphi(x)+\psi(y)\le c_A(x,y)\text{ pointwise}\end{array}\right\}.
\]
Kantorovich duality then gives
\begin{equation}\label{eq:ordinary-igw-kantorovich-dual}
 \mathsf{OT}_{c_A}(\mu,\nu)
 =\sup_{(\varphi,\psi)\in\mathcal K_A}
 \left\{\int\varphi\,d\mu+\int\psi\,d\nu\right\}.
\end{equation}
Equivalently, with
\[
 Q_Au(x):=\inf_y\{u(y)-4y^\top A^\top x\},
\]
the dual formula in \eqref{eq:ordinary-igw-kantorovich-dual} also has the
one-potential form
\begin{equation}\label{eq:ordinary-igw-Q-dual}
 \mathsf{OT}_{c_A}(\mu,\nu)
 =\sup_u\left\{\int Q_Au\,d\mu-\int u\,d\nu\right\},
\end{equation}
where the supremum is over Borel $u$ for which $u\in L^1(\nu)$ and the
universally measurable transform $Q_Au$ belongs to $L^1(\mu)$. Measures are
understood on their completions when such transforms are integrated.
The infimum over $\pi$ in \eqref{eq:ordinary-igw-definition} also admits a
minimizer.
\end{theorem}

\begin{proof}
For independent copies $(X,Y),(X',Y')$ with law $\pi$, expansion and independence
give
\[
 \E\big(\ip{X}{X'}-\ip{Y}{Y'}\big)^2
 =\norm{S_\mu}_F^2+\norm{S_\nu}_F^2-2\norm{M_\pi}_F^2.
\]
The Fenchel identity
\[
 -2\norm M_F^2
 =\min_A\{2\norm A_F^2-4\ip{A}{M}_F\},
 \qquad A^\star=M,
\]
therefore permits the two infima to be grouped in either order and yields
\eqref{eq:ordinary-igw-envelope}. At every joint minimizer,
$A^\star=M_{\pi^\star}$, and Cauchy--Schwarz gives
\[
 \norm{A^\star}_F
 \le\int\norm x\norm y\,d\pi^\star
 \le\sqrt{M_2(\mu)M_2(\nu)}.
\]
This proves the exact ball restriction. It also proves existence of a minimizer: the coupling set
is weakly compact by Prokhorov's theorem \cite[Theorem~5.1]{Billingsley}, its fixed
marginals make the family uniformly square-integrable,
and the objective is continuous along weakly convergent couplings; the ball
$\mathcal A_{\mu,\nu}$ is compact.

Finally,
\[
 |c_A(x,y)|\le
 2\norm A_F(\norm x^2+\norm y^2),
\]
so $c_A$ is continuous and bounded in absolute value by a sum of integrable
marginal functions. Standard Kantorovich duality for such a cost
\cite[Theorem~5.10]{VillaniOT} gives
\eqref{eq:ordinary-igw-kantorovich-dual} with pointwise Borel representatives of
the potentials. Set $u=-\psi$. For fixed $u$, the maximal pointwise admissible
first potential is
\[
 Q_Au(x)=\inf_y\{u(y)-4y^\top A^\top x\}.
\]
This infimum is universally measurable in general by
\cite[Proposition~7.47]{BertsekasShreve}, hence measurable for the completion of
$\mu$. Replacing $\varphi$ by $Q_Au$ proves
\eqref{eq:ordinary-igw-Q-dual}; standard truncation under the marginal envelope
above gives the stated integrable class. The theorem imposes the two-potential
constraint pointwise; an almost everywhere constraint would instead lead to a
$\nu$-essential infimum.
\end{proof}

Theorem~\ref{thm:ordinary-igw-envelope} is the inner-product specialization, in
our normalization, of the variational mechanism developed by Zhang, Goldfeld,
Mroueh, and Sriperumbudur for quadratic Euclidean GW
\cite{ZhangGoldfeldMrouehSriperumbudur}. Their squared-distance theorem assumes
fourth moments because its expanded cost contains $\norm x^2\norm y^2$; the
bilinear IGW formula above needs only second moments. The essential structural point
is the same: a quadratic coupling functional becomes an outer finite-dimensional
optimization plus an ordinary OT problem with a parametrized cost. In ordinary
IGW, the target marginal fixes $S_\nu$, so the envelope needs only the variable
$A$.

Rioux, Goldfeld, and Kato use the entropic version of this envelope as an
algorithmic principle \cite{RiouxGoldfeldKato2024}. In their entropic setting,
regularization yields a differentiable envelope amenable to optimization with
inexact Sinkhorn solves. We use this only as the
computational precedent for the variational architecture developed below: wIGW
replaces the ordinary OT oracle by barycentric weak OT and introduces a second outer
variable $B$ to linearize the optimized barycentric second moment.

\subsection{Weak optimal transport}\label{subsec:weak-ot}

Weak OT replaces a pointwise cost between spaces by a cost of a source
point and its entire conditional target distribution. Let $\mathcal X,\mathcal Y$ be standard Borel
spaces, let $\mu\in\cP(\mathcal X)$ and $\nu\in\cP(\mathcal Y)$. By the
disintegration theorem \cite[Theorem~3.4]{Kallenberg}, every
$\pi\in\Pi(\mu,\nu)$ can be written as
$\pi(dx,dy)=\mu(dx)\pi_x(dy)$. For a measurable conditional cost
$C:\mathcal X\times\cP(\mathcal Y)\to(-\infty,+\infty]$, the weak OT problem is
\begin{equation}\label{eq:generic-wot}
 \mathsf{WT}_{C}(\mu,\nu)
 :=\inf_{\pi\in\Pi(\mu,\nu)}\int C(x,\pi_x)d\mu(x).
\end{equation}
The defining feature is that $C(x,\cdot)$ may be nonlinear in the conditional law
\cite{GRST,BBP}.

For barycentric costs, convex order and Strassen's theorem characterize the
possible conditional means.

\begin{definition}[Convex order]\label{def:convex-order}
For $\eta,\nu\in\cP_1(\R^d)$, write $\eta\cx\nu$ when
$\int u\,d\eta\le\int u\,d\nu$ for every finite convex function of at most linear
growth, meaning that for some finite constant $C$,
\[
 |u(z)|\le C\bigl(1+\norm z\bigr)
 \qquad\text{for every }z\in\R^d.
\]
Equivalently, the inequality holds for every convex test for which both
integrals are well defined; see \cite[Chapter~7]{ShakedShanthikumar2007} for
standard terminology for convex order.
\end{definition}

The order is equivalent to the existence of a martingale coupling whose
conditional mean recovers the less dispersed law.

\begin{theorem}[Strassen martingale characterization]\label{thm:strassen}
For $\eta,\nu\in\cP_1(\R^d)$, $\eta\cx\nu$ if and only if there is a coupling
$\kappa\in\Pi(\eta,\nu)$ with the following martingale property. Let $(Z,Y)$ be
the coordinate random variables under $\kappa$. Then $Z\sim\eta$, $Y\sim\nu$, and
$\E_\kappa[Y\mid Z]=Z$. Equivalently, there is a
Borel probability kernel $\kappa(dy\mid z)$ satisfying
\[
  \int\kappa(\mathord\cdot\mid z)d\eta(z)=\nu,
  \qquad \int y\,\kappa(dy\mid z)=z\quad\text{for $\eta$-a.e. }z.
\]
This is Strassen's theorem \cite{Strassen}.
Appendix~\ref{app:reference-theorems} gives the complete statement, hypotheses,
and proof.
\end{theorem}

Strassen's theorem converts the marginal constraint on a coupling into an order
constraint on its conditional mean map. To state the resulting primal and dual
forms, let $\mu\in\cP_2(\R^{d_x})$, $\nu\in\cP_2(\R^{d_y})$, let
$c:\R^{d_x}\times\R^{d_y}\to\R$ be measurable, and use the disintegration
$\pi(dx,dy)=\mu(dx)\pi_x(dy)$ for $\pi\in\Pi(\mu,\nu)$. For
$\rho\in\cP_1(\R^{d_y})$, define
\begin{equation}\label{eq:barycentric-wot-notation}
\begin{aligned}
 b(\rho)&:=\int y\,d\rho(y),
 &m_\pi(x)&:=b(\pi_x),\\
 \mathsf{WT}_{c}^{\rm bar}(\mu,\nu)
 &:=\inf_{\pi\in\Pi(\mu,\nu)}\int c(x,m_\pi(x))d\mu(x),
 &Q_cu(x)&:=\inf_{z\in\R^{d_y}}\{c(x,z)+u(z)\}.
\end{aligned}
\end{equation}
Let
\begin{equation}\label{eq:quadratic-convex-potentials}
 \mathcal U_2^{\rm cvx}:=
 \left\{u:\R^{d_y}\to\R:
 \begin{array}{l}
 u\text{ is finite, continuous, and convex, and for some }C_u<\infty,\\[-2pt]
 |u(z)|\le C_u(1+\norm z^2)\text{ for every }z\in\R^{d_y}
 \end{array}\right\}.
\end{equation}

The noncompact existence and probability-valued duality inputs below are
Theorems~2.9 and~3.1 of Backhoff-Veraguas, Beiglb\"ock, and Pammer
\cite{BBP}. We then reduce their dual to convex functions of the barycenter.

\begin{theorem}[Barycentric weak OT representations \cite{GRST,BBP}]
\label{thm:wot-blueprint}
Let $\mu\in\cP_2(\R^{d_x})$ and $\nu\in\cP_2(\R^{d_y})$. Assume that
$c:\R^{d_x}\times\R^{d_y}\to\R$ is continuous, $c(x,\cdot)$ is convex, and for
some $\alpha>0$ and finite $C$,
\[
 \alpha\norm z^2-C(1+\norm x^2)
 \le c(x,z)\le C(1+\norm x^2+\norm z^2).
\]
With $\mathsf{WT}_{c}^{\rm bar}$, $Q_c$, and $\mathcal U_2^{\rm cvx}$ defined in
\eqref{eq:barycentric-wot-notation}--\eqref{eq:quadratic-convex-potentials}, the
coupling infimum defining $\mathsf{WT}_{c}^{\rm bar}(\mu,\nu)$ admits a minimizer,
and
\begin{equation}\label{eq:wot-blueprint}
\begin{aligned}
 \mathsf{WT}_{c}^{\rm bar}(\mu,\nu)
 &=\min_{m:\,m_\#\mu\cx\nu}\int c(x,m(x))d\mu(x)\\
 &=\sup_{u\in\mathcal U_2^{\rm cvx}}
   \left\{\int Q_cu\,d\mu-\int u\,d\nu\right\}.
\end{aligned}
\end{equation}
Here the minimum is over Borel maps $m\in L^2(\mu;\R^{d_y})$ satisfying
$m_\#\mu\cx\nu$ in the sense of Definition~\ref{def:convex-order}.
If $c(x,\cdot)$ is uniformly strongly convex, the minimizing barycentric map is
unique in $L^2(\mu)$.
\end{theorem}

\begin{proof}
For a coupling $\pi$, conditional Jensen gives
$(m_\pi)_\#\mu\cx\nu$. Conversely, if $m_\#\mu\cx\nu$,
Theorem~\ref{thm:strassen} supplies a martingale kernel from $m_\#\mu$ to $\nu$;
gluing it after $m$ produces a coupling with conditional mean $m$. This proves the
first equality.

We next verify the noncompact dual. On
$\R^{d_x}\times\cP_2(\R^{d_y})$, set
\[
 C(x,p):=c(x,b(p)),\qquad b(p)=\int y\,dp(y).
\]
The barycenter map is continuous for $W_2$, so $C$ is jointly continuous. It is
convex in $p$ because $b$ is affine and $c(x,\cdot)$ is convex. If $C_0$ is the
constant in the lower growth bound, then adding
$h(x)=C_0(1+\norm x^2)$ makes $C+h$ globally bounded below.
Theorems~2.9 and~3.1 of Backhoff-Veraguas, Beiglb\"ock, and Pammer
\cite{BBP}, applied to $C+h$ and followed by subtraction of $\int h\,d\mu$,
therefore gives existence of a primal minimizer and strong duality for $C$.

It remains to identify its probability-valued dual with the barycentric dual in
\eqref{eq:wot-blueprint}. For a target potential $\psi$ in the BBP dual, define its barycentric
convexification
\[
 u(z):=\inf\left\{\int\psi(y)\,dp(y):
 p\in\cP_2(\R^{d_y}),\ b(p)=z\right\}.
\]
The choice $p=\delta_z$ gives $u\le\psi$. The growth bounds in the BBP dual make
$u$ finite and lower bounded, while mixtures of admissible measures prove convexity.
A finite convex function on Euclidean space is continuous, and the quadratic upper
growth follows by testing with $\delta_z$. Moreover, partitioning the infimum by its
barycenter gives
\[
 \inf_p\left\{c(x,b(p))+\int\psi\,dp\right\}
 =\inf_z\{c(x,z)+u(z)\}=Q_cu(x).
\]
Convexity gives $\int u\,dp\ge u(b(p))$ by Jensen, while $p=\delta_z$ gives
equality at every barycenter $z$. Hence replacing $\psi$ by $u$ leaves this
transform unchanged. It can only increase the dual objective because $u\le\psi$.
Thus the dual in \eqref{eq:wot-blueprint} has the primal value on this convex subclass of functions
bounded below.

We may enlarge it to all of $\mathcal U_2^{\rm cvx}$. Indeed, a finite convex
$u$ has an affine supporting lower bound. Combining that bound with the coercive
lower bound on $c$, and using the quadratic upper growth of $c$ and $u$ at the
test point $z=0$, gives a finite constant $C'$ such that
\[
 |Q_cu(x)|\le C'(1+\norm x^2).
\]
Hence every integral above is well defined. For any Borel
$m\in L^2(\mu;\R^{d_y})$ satisfying $m_\#\mu\cx\nu$,
\[
 Q_cu(x)\le c(x,m(x))+u(m(x)),\qquad
 \int u(m)d\mu\le\int u\,d\nu,
\]
so weak duality holds throughout the enlarged class. Since the original subclass
already has the primal value, the enlargement preserves equality. Finally, uniform
strong convexity makes $m\mapsto\int c(x,m(x))d\mu$ strongly convex on the
convex set of Borel maps $m\in L^2(\mu;\R^{d_y})$ satisfying
$m_\#\mu\cx\nu$, proving uniqueness of the minimizing barycentric map.
\end{proof}

For $d_x=d_y=d$ and $c(x,z)=\norm{x-z}^2$, this theorem specializes to quadratic
barycentric weak OT. Its geometric form is the known convex order projection
identity \cite{GozlanJuillet}
\begin{equation}\label{eq:wot-projection-blueprint}
 \mathsf T_2(\nu\mid\mu)
 =\inf_{\substack{\eta\in\cP_2(\R^d)\\\eta\cx\nu}}W_2^2(\mu,\eta).
\end{equation}
The relational construction uses the same conditional mechanism for a source pair:
the conditional laws $(\pi_x,\pi_{x'})$ replace the single law $\pi_x$, and their aggregate is
compared with the source relation. This gives the weak GW definition below. The
Jensen--Strassen correspondence will then yield its convex order projection and,
after ridge regularization, the weak OT block used by the numerical method.

\section{A weak Gromov--Wasserstein framework based on conditional laws}
\label{sec:weak-gw-framework}

We formulate weak GW directly in terms of the conditional laws induced by a
coupling.

\begin{definition}[Weak GW based on conditional laws]\label{def:generic-wgw}
Let $(\mathcal X,\mathscr X)$ and $(\mathcal Y,\mathscr Y)$ be standard Borel
spaces, let $\mu\in\cP(\mathcal X)$ and $\nu\in\cP(\mathcal Y)$, let
$c_{\mathcal X}:\mathcal X^2\to\R$ be measurable, and let
$\mathfrak C:\R\times\cP(\mathcal Y)^2\to[0,+\infty]$ be measurable for the
evaluation $\sigma$-field on $\cP(\mathcal Y)$. For each
$\pi\in\Pi(\mu,\nu)$, choose a disintegration
$\pi(dx,dy)=\mu(dx)\pi_x(dy)$. Define, whenever the integral is well defined,
\begin{equation}\label{eq:generic-wgw-definition}
  \WGW_{\mathfrak C}(\mu,\nu)
  :=\inf_{\pi\in\Pi(\mu,\nu)}
  \iint \mathfrak C\big(c_{\mathcal X}(x,x'),\pi_x,\pi_{x'}\big)
  d\mu(x)d\mu(x').
\end{equation}
Changing the disintegration on a $\mu$-null set does not change the value.
\end{definition}

For the two principal lifts, let
$c_{\mathcal Y}:\mathcal Y^2\to\R$ and
$\mathcal L:\R^2\to[0,+\infty]$ be measurable. The first lift averages the
original pointwise loss;
the second aggregates each pair of conditional laws before applying the loss:
\begin{equation}\label{eq:principal-lifted-costs}
\begin{aligned}
 \mathfrak C_{\rm GW}(a,\rho,\rho')
   &:=\iint\mathcal L\big(a,c_{\mathcal Y}(y,y')\big)d\rho(y)d\rho'(y'),\\
 \mathfrak C_D(a,\rho,\rho')
   &:=\mathcal L\big(a,D(\rho,\rho')\big).
\end{aligned}
\end{equation}
Here $D:\cP(\mathcal Y)^2\to\R$ is assumed measurable. The product measure
$\rho\otimes\rho'$ depends measurably on the pair $(\rho,\rho')$. Consequently,
integrating the measurable pointwise loss against this product defines a measurable
lift $\mathfrak C_{\rm GW}$. The aggregated lift $\mathfrak C_D$ is measurable
because it is the composition of the measurable maps $D$ and $\mathcal L$.

The averaging lift in \eqref{eq:principal-lifted-costs} provides a consistency
check. Disintegrating both copies of a coupling recovers the original GW
objective before conditional aggregation, as the following proposition shows.

\begin{proposition}[Ordinary GW is an exact specialization]\label{prop:gw-specialization}
Let $(\mathcal X,\mathscr X)$ and $(\mathcal Y,\mathscr Y)$ be standard Borel
spaces, let $\mu\in\cP(\mathcal X)$ and $\nu\in\cP(\mathcal Y)$, and let
$c_{\mathcal X}$ and $c_{\mathcal Y}$ be measurable relations with values in $\R$. For
every measurable $\mathcal L:\R^2\to[0,+\infty]$, define
$\mathfrak C_{\rm GW}$ by \eqref{eq:principal-lifted-costs} and
$\WGW_{\mathfrak C_{\rm GW}}$ by \eqref{eq:generic-wgw-definition}. Then
\[
  \WGW_{\mathfrak C_{\rm GW}}(\mu,\nu)
  =\inf_{\pi\in\Pi(\mu,\nu)}
  \iint \mathcal L\big(c_{\mathcal X}(x,x'),c_{\mathcal Y}(y,y')\big)
  d\pi(x,y)d\pi(x',y').
\]
In particular, for $\mathcal X=\R^{d_x}$, $\mathcal Y=\R^{d_y}$,
$\mu\in\cP_2(\R^{d_x})$, $\nu\in\cP_2(\R^{d_y})$,
$c_{\mathcal X}(x,x')=\ip{x}{x'}$,
$c_{\mathcal Y}(y,y')=\ip{y}{y'}$, and
$\mathcal L(a,b)=(a-b)^2$, then
\[
 \WGW_{\mathfrak C_{\rm GW}}(\mu,\nu)=\IGW^2(\mu,\nu),
\]
where the ordinary IGW objective is defined in
\eqref{eq:ordinary-igw-definition}.
\end{proposition}

\begin{proof}
For a fixed coupling, disintegrate both factors of $\pi\otimes\pi$ and apply
Tonelli's theorem (or Fubini's theorem in the integrable case):
\[
\begin{aligned}
 &\iint\mathcal L\big(c_{\mathcal X}(x,x'),c_{\mathcal Y}(y,y')\big)
 d\pi(x,y)d\pi(x',y')\\
 &\quad=\iint\left[\iint\mathcal L\big(c_{\mathcal X}(x,x'),
 c_{\mathcal Y}(y,y')\big)d\pi_x(y)d\pi_{x'}(y')\right]d\mu(x)d\mu(x').
\end{aligned}
\]
The bracket is $\mathfrak C_{\rm GW}(c_{\mathcal X}(x,x'),\pi_x,\pi_{x'})$.
Infimizing over $\pi$ proves the claim.
\end{proof}

\section{Barycentric wIGW}
\label{sec:wigw-definition}

Let \(\mu\in\cP_2(\R^{d_x})\) and \(\nu\in\cP_2(\R^{d_y})\). For
\(\pi\in\Pi(\mu,\nu)\), let \(\pi_x\) be a disintegration and define
\begin{equation}\label{eq:conditional-mean-map}
  m_\pi(x):=\int y\,\pi_x(dy)=\E_\pi[Y\mid X=x].
\end{equation}
Conditional Jensen gives $m_\pi\in L^2(\mu;\R^{d_y})$ and places every such map
below $\nu$ in convex order. The next lemma establishes the convexity and weak
compactness needed to optimize directly over all maps with this property.
Appendix~\ref{app:weak-compactness-direct-method} defines the weak
topology and weak compactness used here and states the direct method theorem
applied later.

\begin{lemma}[Geometry of the admissible map set]\label{lem:feasible-set}
Let $\mu\in\cP_2(\R^{d_x})$ and $\nu\in\cP_2(\R^{d_y})$, and define
\begin{equation}\label{eq:admissible-map-set}
  \mathcal C_\nu:=
  \{m\in L^2(\mu;\R^{d_y}):m_\#\mu\cx\nu\},
\end{equation}
where $\cx$ is the convex order of Definition~\ref{def:convex-order}. Then
$\mathcal C_\nu$ is nonempty, convex, and weakly compact in
$L^2(\mu;\R^{d_y})$. Every $m\in\mathcal C_\nu$ satisfies
\[
  \norm{m}_{L^2(\mu)}^2\le M_2(\nu):=\int\norm{y}^2d\nu(y).
\]
Convex order may be tested using finite continuous convex functions of at most
linear growth.
\end{lemma}

\begin{proof}
Let $\bar y=\int y\,d\nu(y)$. Jensen's inequality gives
$\delta_{\bar y}\cx\nu$, so the constant map $m\equiv\bar y$ proves nonemptiness.
If $m_0,m_1\in\mathcal C_\nu$ and $t\in[0,1]$, then for every convex $u$,
\[
 \int u((1-t)m_0+tm_1)d\mu
 \le(1-t)\int u(m_0)d\mu+t\int u(m_1)d\mu
 \le\int u\,d\nu,
\]
which proves convexity. Using the equivalent convex-order formulation for convex
tests with finite integrals, take $u(z)=\norm z^2$ to obtain the stated bound.

To prove weak closedness, suppose $m_n\rightharpoonup m$ in $L^2$. For every
finite continuous convex $u$ with at most linear growth, the integral functional
$m\mapsto\int u(m)d\mu$ is convex and lower semicontinuous in norm, hence weakly
lower semicontinuous. Therefore
\[
  \int u(m)d\mu\le\liminf_n\int u(m_n)d\mu\le\int u\,d\nu.
\]
Such functions determine convex order on $\cP_1(\R^{d_y})$: an arbitrary finite
convex function is obtained increasingly from maxima of finitely many supporting
affine functions, and truncating the slopes gives linearly growing convex tests.
Thus $m\in\mathcal C_\nu$. The set is weakly closed and bounded in the reflexive
Hilbert space $L^2$, hence weakly compact.
\end{proof}

We specialize this framework by evaluating the wIGW relational loss at the
conditional mean map. For $m\in L^2(\mu;\R^{d_y})$, set
\begin{equation}\label{eq:wigw-map-objective}
 \mathcal J_\mu(m):=
 \iint\left(\ip{x}{x'}-\ip{m(x)}{m(x')}\right)^2d\mu(x)d\mu(x').
\end{equation}

\begin{definition}[Barycentric weak inner-product GW]\label{def:wigw}
For $\mu\in\cP_2(\R^{d_x})$, $\nu\in\cP_2(\R^{d_y})$, and the conditional mean
map $m_\pi$ in \eqref{eq:conditional-mean-map}, define
\begin{equation}\label{eq:wigw-coupling-primal}
  \wIGW^2(\mu,\nu)
  :=\inf_{\pi\in\Pi(\mu,\nu)}\mathcal J_\mu(m_\pi).
\end{equation}
Here $\mathcal J_\mu$ is defined in \eqref{eq:wigw-map-objective}.
The unsquared symbol \(\wIGW(\mu,\nu)\) denotes the nonnegative square root
of the value in \eqref{eq:wigw-coupling-primal}.
\end{definition}

\Needspace{0.18\textheight}
This definition is the aggregated specialization of the second lift in
\eqref{eq:principal-lifted-costs}, obtained by setting
\[
 D_{\rm bar}(\rho,\rho')
 :=\ip{\int y\,d\rho(y)}{\int y'\,d\rho'(y')},
 \qquad \mathcal L(a,b)=(a-b)^2.
\]
Regard $D_{\rm bar}$ as a function on $\cP(\R^{d_y})^2$ by setting
$D_{\rm bar}(\rho,\rho')=0$ whenever either measure lies outside
$\cP_1(\R^{d_y})$. The set $\cP_1(\R^{d_y})$ is Borel, and the barycenter map is
Borel on this set, so the resulting extension is measurable. Under
$\nu\in\cP_2$, the identity
$\int\!\int\norm y^2d\pi_x(y)d\mu(x)=M_2(\nu)$ shows that
$\pi_x\in\cP_2$ for $\mu$-almost every $x$. Hence every admissible coupling uses
the defining formula for $D_{\rm bar}$ for $\mu$-almost every conditional law.
With $c_{\mathcal X}(x,x')=\ip{x}{x'}$, one has the exact identification
\[
 \wIGW^2(\mu,\nu)=\WGW_{\mathfrak C_{D_{\rm bar}}}(\mu,\nu).
\]
For every fixed coupling $\pi$, if $(X,Y)$ and $(X',Y')$ are independent with
law $\pi$, conditional Jensen gives
\[
\begin{aligned}
 &\left(\ip{x}{x'}-D_{\rm bar}(\pi_x,\pi_{x'})\right)^2\\
 &\quad=\left(\ip{x}{x'}-
 \E[\ip{Y}{Y'}\mid X=x,X'=x']\right)^2\\
 &\quad\le \E\left[(\ip{x}{x'}-\ip{Y}{Y'})^2
 \mid X=x,X'=x'\right].
\end{aligned}
\]
The inequality quantifies the reduction obtained by replacing the pointwise target
relation with its conditional expectation before applying the squared loss.

\begin{remark}[Barycentric and distributional relations]
\label{rem:conditional-law-alternatives}
Equation~\eqref{eq:principal-lifted-costs} defines the aggregated lift
\[
  \mathfrak C_D(a,\rho,\rho')
  =\mathcal L\bigl(a,D(\rho,\rho')\bigr),
\]
whose weak GW value $\WGW_{\mathfrak C_D}$ is defined in
Definition~\ref{def:generic-wgw}. The present construction takes
$D=D_{\rm bar}$, where
\[
  D_{\rm bar}(\rho,\rho')
  =\ip{b(\rho)}{b(\rho')},
  \qquad b(\rho)=\int y\,d\rho(y).
\]
Thus each conditional law enters only through its barycenter, and conditional laws
with the same mean are identified. Other choices of $D$ can retain more
distributional information. Appendix~\ref{app:general-wgw-geometry},
Remark~\ref{rem:conditional-law-relations}, defines the maximal covariance relation
$D_{\rm MCov}$, the Wasserstein relation $D_{W_2}$, the kernel inner-product relation
$D_k^{\rm ip}$, and the MMD relation $D_k^{\rm MMD}$, together with their compatible
source relations. The projection, moment, and finite $A$--$B$ results below are
established for $D_{\rm bar}$. Corresponding results for the other relations require
separate analysis and are left for future work.
\end{remark}

Convex order characterizes exactly the conditional mean maps generated by
couplings with target marginal $\nu$. This characterization removes the coupling
from the primal.

\begin{proposition}[Coupling/map equivalence]\label{prop:map}
Let $\mu\in\cP_2(\R^{d_x})$ and $\nu\in\cP_2(\R^{d_y})$. With
$\mathcal C_\nu$ and $\mathcal J_\mu$ defined in
\eqref{eq:admissible-map-set} and \eqref{eq:wigw-map-objective}, respectively,
\begin{equation}\label{eq:wigw-map-primal}
  \wIGW^2(\mu,\nu)
  =\inf_{m\in\mathcal C_\nu}\mathcal J_\mu(m)
  =\inf_{m\in\mathcal C_\nu}
  \iint\left(\ip{x}{x'}-\ip{m(x)}{m(x')}\right)^2d\mu(x)d\mu(x').
\end{equation}
\end{proposition}

\begin{proof}
If \(m=m_\pi\), conditional Jensen gives, for every convex \(u\),
\[
  \int u\,d(m_\#\mu)=\E u(\E[Y\mid X])\le\E u(Y)=\int u\,d\nu.
\]
Thus \(m_\#\mu\cx\nu\). Conversely, if \(m_\#\mu\cx\nu\),
Theorem~\ref{thm:strassen} gives a martingale kernel from \(m_\#\mu\) to \(\nu\).
Then \(\pi(dx,dy)=\mu(dx)\kappa(dy\mid m(x))\) has conditional mean \(m(x)\).
The kernel composition is measurable because the spaces are standard Borel. The
objective depends on \(\pi\) only through this mean. Finally, all terms in
\eqref{eq:wigw-map-primal}
are finite: conditional Jensen gives $m\in L^2$, and for independent copies
$\E\ip{m(X)}{m(X')}^2=\norm{S_m}_F^2<\infty$.
\end{proof}

\begin{remark}[Directionality]
The condition \(m_\#\mu\cx\nu\) is directional, and the values
\(\wIGW(\mu,\nu)\) and \(\wIGW(\nu,\mu)\) can differ. We use ``discrepancy'' and
reserve ``distance'' for symmetric settings.
\end{remark}

\section{Moment representation and zero structure}
\label{sec:moment-representation}

We reduce the pairwise loss to three finite-dimensional second moments. For
$\mu\in\cP_2(\R^{d_x})$ and
$m\in L^2(\mu;\R^{d_y})$, define
\begin{equation}\label{eq:wigw-moment-matrices}
  S_\mu=\int xx^\top d\mu(x),\qquad
  S_m=\int m(x)m(x)^\top d\mu(x),\qquad
  M_m=\int x\,m(x)^\top d\mu(x).
\end{equation}

\begin{proposition}[Moment reduction]\label{prop:moment}
Let $\mu\in\cP_2(\R^{d_x})$, $\nu\in\cP_2(\R^{d_y})$, and
$m\in\mathcal C_\nu$, where $\mathcal C_\nu$ is defined in
\eqref{eq:admissible-map-set}. For the moment matrices in
\eqref{eq:wigw-moment-matrices},
\[
\begin{aligned}
  &\iint\left(\ip{x}{x'}-\ip{m(x)}{m(x')}\right)^2d\mu(x)d\mu(x')\\
  &\hspace{4em}=\norm{S_\mu}_F^2+\norm{S_m}_F^2-2\norm{M_m}_F^2.
\end{aligned}
\]
Consequently,
\[
  \wIGW^2(\mu,\nu)=\norm{S_\mu}_F^2+
  \inf_{m\in\mathcal C_\nu}\{\norm{S_m}_F^2-2\norm{M_m}_F^2\}.
\]
\end{proposition}

\begin{proof}
Let \((X,m(X))\) and \((X',m(X'))\) be independent copies. Independence gives
\[
  \E\ip{X}{X'}^2=\norm{S_\mu}_F^2,
  \quad \E\ip{m(X)}{m(X')}^2=\norm{S_m}_F^2,
\]
and
\[
  \E[\ip{X}{X'}\ip{m(X)}{m(X')}]
  =\sum_{i,k}(\E[X_i m_k(X)])^2=\norm{M_m}_F^2.
\]
Expand the square.
\end{proof}

The squared loss gives a direct description of the zero set.
Lemma~\ref{lem:feasible-set} makes the feasible map set weakly compact, and the
moment representation above makes the objective weakly lower semicontinuous.
The map problem therefore has a minimizer. Nonnegativity of the squared objective
then yields the following characterization.

\begin{corollary}[Zero set]\label{cor:zero}
Let $\mu\in\cP_2(\R^{d_x})$ and $\nu\in\cP_2(\R^{d_y})$. Then
\(\wIGW(\mu,\nu)=0\) if and only if there is a Borel map \(m\) such that
\[
  m\in L^2(\mu;\R^{d_y}),
  \qquad
  m_\#\mu\cx\nu,
  \qquad \ip{x}{x'}=\ip{m(x)}{m(x')}
  \quad \mu\otimes\mu\text{-a.e.}
\]
Here $\cx$ denotes the convex order of Definition~\ref{def:convex-order}.
\end{corollary}

\begin{proof}
Let $m^\star$ be a minimizer, whose existence follows from the preceding
compactness and lower semicontinuity argument. Since the integrand is
nonnegative, $m^\star$ has zero objective if and only if that integrand vanishes
$\mu\otimes\mu$-almost everywhere. Conversely, any feasible map with this
property has zero objective.
\end{proof}

The zero set criterion is especially transparent when the conditional mean is an
isometric embedding. It then allows arbitrary target noise whose conditional mean
vanishes.

\begin{corollary}[Isometries with martingale noise]\label{cor:noisy-isometry}
Let $\mu\in\cP_2(\R^{d_x})$ and $\nu\in\cP_2(\R^{d_y})$.
Let \(T:\R^{d_x}\to\R^{d_y}\) be a linear isometric embedding,
\(T^\top T=I_{d_x}\). If $T_\#\mu\cx\nu$ in the convex order of
Definition~\ref{def:convex-order}, then \(\wIGW(\mu,\nu)=0\). In particular,
suppose $(X,Y)$ has a joint law in $\Pi(\mu,\nu)$ and
\[
  Y=TX+\xi,
  \qquad \E[\xi\mid X]=0.
\]
Then \(\wIGW(\mu,\nu)=0\). The conclusion allows $\xi$ to depend on $X$ and to
be heteroscedastic, non-Gaussian, and anisotropic.
\end{corollary}

\begin{proof}
For the first statement, take \(m(x)=Tx\) in Corollary~\ref{cor:zero}. It is
feasible by assumption, and
\[
  \ip{m(x)}{m(x')}=x^\top T^\top Tx'=\ip{x}{x'}.
\]
For the second statement, conditional centering gives
\[
  \E[Y\mid X]=TX+\E[\xi\mid X]=TX.
\]
Let $\varphi:\R^{d_y}\to\R$ be any finite convex function of at most linear
growth. Conditional Jensen yields
\[
\begin{aligned}
  \int\varphi(z)\,d(T_\#\mu)(z)
  &=\E[\varphi(TX)]
   =\E\!\left[\varphi\!\left(\E[Y\mid X]\right)\right]\\
  &\le \E[\varphi(Y)]
   =\int\varphi(y)\,d\nu(y).
\end{aligned}
\]
Thus $T_\#\mu\cx\nu$ by Definition~\ref{def:convex-order}, and the first
statement applies.
\end{proof}

\begin{remark}[Familiar orthogonal transformation with noise]
When $d_x=d_y$, the familiar case is
\[
  Y=OX+\xi,
  \qquad \E[\xi\mid X]=0,
\]
with $O$ orthogonal. The noise may be heteroscedastic, anisotropic, dependent on
$X$, and non-Gaussian. The certificate requires a vanishing conditional mean.
Global centering is weaker: $\E[\xi]=0$ may hold while
$\E[\xi\mid X]\ne0$.
\end{remark}

\begin{remark}[Orthogonal invariance]
For $U\in O(d_x)$ and $V\in O(d_y)$,
\(\wIGW(U_\#\mu,V_\#\nu)=\wIGW(\mu,\nu)\). Raw inner products are sensitive to
translations, so centering is part of our experimental protocol.
\end{remark}

The zero set results identify exact relational matches. The next section extends
this geometry to arbitrary values of the discrepancy through projection in convex
order.

\section{Projection onto the convex order cone}\label{sec:projection}

For \(\mu\in\cP_2(\R^{d_x})\) and \(\eta\in\cP_2(\R^{d_y})\), define
\begin{equation}\label{eq:projection-igw-objective}
  \IGW^2(\mu,\eta):=\inf_{\gamma\in\Pi(\mu,\eta)}
  \iint\left(\ip{x}{x'}-\ip{z}{z'}\right)^2
  d\gamma(x,z)d\gamma(x',z').
\end{equation}
We show that wIGW equals the minimum ordinary IGW discrepancy over laws below the
target in convex order.

\begin{theorem}[Exact projection in convex order]\label{thm:projection}
For every \(\mu\in\cP_2(\R^{d_x})\) and \(\nu\in\cP_2(\R^{d_y})\),
\[
  \wIGW^2(\mu,\nu)=
  \inf_{\substack{\eta\in\cP_2(\R^{d_y})\\ \eta\cx\nu}}
  \IGW^2(\mu,\eta).
\]
Here $\cx$ is the convex order of Definition~\ref{def:convex-order} and
$\wIGW^2(\mu,\nu)$ and $\IGW^2(\mu,\eta)$ are defined in
\eqref{eq:wigw-coupling-primal} and \eqref{eq:projection-igw-objective},
respectively.
\end{theorem}

\begin{proof}
Take a feasible barycentric map \(m\) and set \(\eta=m_\#\mu\cx\nu\). The
deterministic coupling \((\mathrm{id},m)_\#\mu\) is feasible for
\(\IGW(\mu,\eta)\), so the projection value is no larger than the wIGW value.

Conversely, fix \(\eta\cx\nu\) and \(\gamma\in\Pi(\mu,\eta)\). Let
\(m_\gamma(x)=\E_\gamma[Z\mid X=x]\). Conditional Jensen gives
\((m_\gamma)_\#\mu\cx\eta\cx\nu\). Let \((X,Z)\) and \((X',Z')\) be
independent with law \(\gamma\). A regular conditional law of $(Z,Z')$ given
$(X,X')=(x,x')$ is $\gamma_x\otimes\gamma_{x'}$; hence conditional independence
and bilinearity give
\[
\begin{aligned}
  \E[\ip{Z}{Z'}\mid X=x,X'=x']
  &=\iint\ip{z}{z'}d\gamma_x(z)d\gamma_{x'}(z')\\
  &=\ip{m_\gamma(x)}{m_\gamma(x')}.
\end{aligned}
\]
Conditional Jensen, applied to the square as a convex function of the scalar
\(\ip{Z}{Z'}\), yields
\[
\begin{aligned}
  &\left(\ip{X}{X'}-\ip{m_\gamma(X)}{m_\gamma(X')}\right)^2\\
  &\qquad\le
  \E\left[\left(\ip{X}{X'}-\ip{Z}{Z'}\right)^2\mid X,X'\right].
\end{aligned}
\]
This application is legitimate under only second moments: independence gives
$\E\ip{Z}{Z'}^2=\norm{S_\eta}_F^2<\infty$ and likewise for $X$ and
$m_\gamma(X)$, the latter by conditional Jensen. Integrating the conditional
inequality shows that the wIGW objective at $m_\gamma$ is at most the IGW
objective at $\gamma$. Convex order is transitive because its defining inequalities
are transitive. The argument holds for every admissible $\eta$ and $\gamma$;
infimize first over $\gamma$ and then over $\eta$.
\end{proof}

The feasible projection set expands when the target is replaced by a martingale
refinement. The projection identity therefore gives the following directional
monotonicity.

\begin{corollary}[Monotonicity under target martingale refinement]
\label{cor:target-refinement}
Let \(\mu\in\cP_2(\R^{d_x})\) and
\(\nu,\widetilde\nu\in\cP_2(\R^{d_y})\). If
\(\nu\cx\widetilde\nu\) in the convex order of
Definition~\ref{def:convex-order}, then
\[
  \wIGW(\mu,\widetilde\nu)\le \wIGW(\mu,\nu).
\]
Thus a mean-preserving spread of the target cannot increase directional wIGW.
\end{corollary}

\begin{proof}
Transitivity of convex order gives
\[
  \{\eta\in\cP_2(\R^{d_y}):\eta\cx\nu\}\subseteq
  \{\eta\in\cP_2(\R^{d_y}):\eta\cx\widetilde\nu\}.
\]
The claim follows by comparing the two infima in
Theorem~\ref{thm:projection}.
\end{proof}

Choosing the target law itself as the intermediate measure compares the weak and
ordinary discrepancies directly.

\begin{corollary}[Comparison with ordinary IGW]\label{cor:igw-comparison}
For every $\mu\in\cP_2(\R^{d_x})$ and $\nu\in\cP_2(\R^{d_y})$, the
quantities defined in \eqref{eq:wigw-coupling-primal} and
\eqref{eq:projection-igw-objective} satisfy
\(\wIGW(\mu,\nu)\le\IGW(\mu,\nu)\).
\end{corollary}

\begin{proof}
In Theorem~\ref{thm:projection}, the choice $\eta=\nu$ is admissible because
$\nu\cx\nu$.
\end{proof}

The preceding results identify optimal values. The direct method also provides
optimizing maps, projection laws, and couplings, while Strassen's kernel realizes
each optimal map with the prescribed target marginal.

\begin{theorem}[Existence of minimizers and coupling realization in $\cP_2$]
\label{thm:existence}
Let $\mu\in\cP_2(\R^{d_x})$ and $\nu\in\cP_2(\R^{d_y})$, and let
$\mathcal C_\nu$, $\mathcal J_\mu$, and $m_\pi$ be defined in
\eqref{eq:admissible-map-set}, \eqref{eq:wigw-map-objective}, and
\eqref{eq:conditional-mean-map}. The following three problems each admit a
minimizer:
\[
\begin{aligned}
 \text{map:}\quad
 &\inf_{m\in\mathcal C_\nu}\mathcal J_\mu(m),\\
 \text{coupling:}\quad
 &\inf_{\pi\in\Pi(\mu,\nu)}\mathcal J_\mu(m_\pi),\\
 \text{projection:}\quad
 &\inf_{\substack{\eta\in\cP_2(\R^{d_y})\\ \eta\cx\nu}}
   \IGW^2(\mu,\eta).
\end{aligned}
\]
Here $\cx$ is the convex order of
Definition~\ref{def:convex-order} and the last objective is defined in
\eqref{eq:projection-igw-objective}. If $m^\star$ minimizes the map problem, then
$\eta^\star=(m^\star)_\#\mu$ minimizes the projection problem and
$(\mathrm{id},m^\star)_\#\mu$ minimizes the IGW coupling problem between $\mu$
and $\eta^\star$. By Theorem~\ref{thm:strassen}, there exists a Borel probability
kernel \(\kappa^\star(dy\mid z)\) satisfying
\[
 \int\kappa^\star(\mathord\cdot\mid z)d((m^\star)_\#\mu)(z)=\nu,
 \qquad
 \int y\,\kappa^\star(dy\mid z)=z
 \quad\text{for $(m^\star)_\#\mu$-a.e. }z.
\]
Every such kernel defines
\[
  \pi^\star(dx,dy)=\mu(dx)\kappa^\star(dy\mid m^\star(x)).
\]
The measure $\pi^\star$ belongs to $\Pi(\mu,\nu)$ and minimizes the coupling
problem.
Conversely, if \(\eta^\star\in\cP_2(\R^{d_y})\) minimizes the projection problem and
\(\gamma^\star\in\Pi(\mu,\eta^\star)\) minimizes the coupling problem defining
$\IGW^2(\mu,\eta^\star)$ in \eqref{eq:projection-igw-objective}, then, for any
disintegration $\gamma^\star(dx,dz)=\mu(dx)\gamma_x^\star(dz)$, the map
\[
  \widehat m(x):=\int z\,\gamma_x^\star(dz)
  =\E_{\gamma^\star}[Z\mid X=x]
\]
minimizes the map problem. A probability kernel satisfying the same target marginal
and martingale identities with $m^\star$ replaced by $\widehat m$ exists; the same
construction then gives a coupling in $\Pi(\mu,\nu)$ that minimizes the coupling
problem.
In particular, every optimal map is realizable by an optimal wIGW coupling.
\end{theorem}

\begin{proof}
By Lemma~\ref{lem:feasible-set}, $\mathcal C_\nu$ is weakly compact in $L^2$.
For $m\in\mathcal C_\nu$, write $Tm=M_m$. The finite-dimensional bounded linear
map $T:L^2(\mu;\R^{d_y})\to\R^{d_x\times d_y}$ sends weakly convergent sequences
to sequences converging in Frobenius norm. Moreover,
\[
 \norm{S_m}_F^2
 =\sup_{B\succeq0}\left\{2\int m(x)^\top Bm(x)d\mu(x)-\norm B_F^2\right\}.
\]
For every fixed $B\succeq0$, the integral in braces is convex and
lower semicontinuous in norm as a function of $m$, hence weakly lower semicontinuous.
Its supremum is therefore
weakly lower semicontinuous. Since $m\mapsto-2\norm{Tm}_F^2$ is weakly continuous,
Proposition~\ref{prop:moment}'s map objective is weakly lower semicontinuous on
$\mathcal C_\nu$ and therefore has a minimizer there.

Let $m^\star$ be a minimizer and set $\eta^\star=(m^\star)_\#\mu$. The first half
of the proof of Theorem~\ref{thm:projection} shows that the deterministic coupling
$(\mathrm{id},m^\star)_\#\mu$ has IGW cost equal to the weak optimum. Since the
projection and wIGW values are equal, both $\eta^\star$ and that coupling are
optimal. Theorem~\ref{thm:strassen} supplies a martingale kernel from
$\eta^\star$ to $\nu$; gluing it after $m^\star$ produces an optimal wIGW coupling,
so the coupling primal also admits a minimizer.

Conversely, ordinary IGW has a minimizer under second moments by
Theorem~\ref{thm:ordinary-igw-envelope}. For any minimizer $\eta^\star$ of the
projection problem and any minimizer $\gamma^\star$ of the inner coupling problem,
the second half of the proof
of Theorem~\ref{thm:projection} sends $\gamma^\star$ to the feasible conditional
mean $\widehat m$ with no larger cost. Equality of the two optimal values makes it a
wIGW minimizer, and the stated martingale gluing follows again from
Theorem~\ref{thm:strassen}.
\end{proof}

The two directions of Theorem~\ref{thm:existence} produce optimal mean maps in
different ways. The map $m^\star$ is chosen directly as an optimizer of the map
problem. By contrast, $\widehat m(x)=\E_{\gamma^\star}[Z\mid X=x]$ is the
conditional mean of an optimal IGW coupling $\gamma^\star$ for an optimal
projection law. If $X\sim\mu$, define the random variable
$\widehat M:=\widehat m(X)$ and its law
$\widehat\eta:=\widehat m_\#\mu$. This law is an optimal projection law, although
it need not equal the original optimizer $\eta^\star$ when minimizers are
nonunique. A Strassen martingale kernel
$\widehat\kappa(dy\mid z)$ from $\widehat\eta$ to $\nu$ then gives the chain
\[
 X\sim\mu
 \xrightarrow{\ \widehat m\ }
 \widehat M=\widehat m(X)\sim\widehat\eta
 \xrightarrow{\ \widehat\kappa\ }
 Y\sim\nu,
 \qquad
 \widehat\pi(dx,dy)
 =\mu(dx)\widehat\kappa(dy\mid\widehat m(x)).
\]
Both $m^\star$ and $\widehat m$ are optimal feasible mean maps, but they need not
coincide when minimizers are nonunique. Strassen's theorem glues either map to the
prescribed target law $\nu$ and thereby produces an optimal wIGW coupling.

These optimizing objects support the dual representations. Under compact support,
convex potentials and finite-dimensional moment variables can be handled directly.

\section{Duality under compact support}\label{sec:compact-duality}

The projection formula gives a geometric characterization of wIGW. Under compact
support, we derive a dual form in which convex potentials enforce convex
order and two finite-dimensional matrices linearize the moment terms.

We write \(\mathbb S^d\) for the space of real symmetric \(d\times d\)
matrices and \(\mathbb S_+^d\) for its positive semidefinite cone.

Assume that \(\mu\in\cP_2(\R^{d_x})\) and \(\nu\in\cP_2(\R^{d_y})\) satisfy
\(\operatorname{supp}(\mu)\subset K_X\) and
\(\operatorname{supp}(\nu)\subset K_Y\) for compact sets \(K_X,K_Y\). Set
\begin{equation}\label{eq:compact-dual-domains}
\begin{aligned}
  K&=\operatorname{conv}(K_Y),\\
  R_X&=\sup_{x\in K_X}\norm{x},\qquad
  R_Y=\sup_{z\in K}\norm{z},\\
  \mathcal A&=\{A\in\R^{d_x\times d_y}:\norm{A}_F\le R_X R_Y\},\\
  \mathcal B&=\{B\in\mathbb S_+^{d_y}:0\preceq B\preceq R_Y^2I\},\\
  \cM&=\{m\in L^2(\mu;\R^{d_y}):m(x)\in K\ \mu\text{-a.e.}\},\\
  \cU_K&=\{u\in C(K):u\text{ is convex on }K\}.
\end{aligned}
\end{equation}
The set \(\cM\) is weakly compact convex in \(L^2\), and \(\cU_K\) is a convex
cone in \(C(K)\) with the uniform topology. On measures supported in \(K\),
\(\cU_K\) tests convex order in both directions. If \(m_\#\mu\cx\nu\), a
Strassen martingale coupling is supported on \(K\times K\), so conditional Jensen
gives \(\int u(m)d\mu\le\int u\,d\nu\) for every continuous convex function
\(u\) defined on \(K\). Conversely, if convex order fails, a globally defined
convex function separates \(m_\#\mu\) and \(\nu\); its restriction to \(K\)
belongs to \(\cU_K\) and preserves the strict separating inequality.
For
\((A,B,u)\in\mathcal A\times\mathcal B\times\cU_K\), define
\begin{equation}\label{eq:compact-transform-functional}
\begin{aligned}
  Q_{A,B}^{K}u(x)
  &=\min_{z\in K}\{u(z)+2z^\top Bz-4z^\top A^\top x\},\\
  \mathcal D_K(A,B,u)&=2\norm{A}_F^2-\norm{B}_F^2
  -\int u\,d\nu+\int Q_{A,B}^{K}u\,d\mu.
\end{aligned}
\end{equation}

The Fenchel identities are
\begin{equation}\label{eq:moment-fenchel-identities}
\begin{aligned}
  \norm{S_m}_F^2
  &=\sup_{B\succeq0}\{2\ip{B}{S_m}_F-\norm{B}_F^2\},
  &B^\star&=S_m,\\
  -2\norm{M_m}_F^2
  &=\inf_{A\in\R^{d_x\times d_y}}
    \{2\norm{A}_F^2-4\ip{A}{M_m}_F\},
  &A^\star&=M_m.
\end{aligned}
\end{equation}
The compact restrictions contain these optimizers.

We use Sion's minimax theorem in the precise form stated in
Appendix~\ref{app:reference-theorems}. In every application below, the compact
variable, ambient topologies, and sectionwise semicontinuity and convexity are
identified explicitly.

\begin{theorem}[Duality under compact support]\label{thm:compact-dual}
Let $K_X\subset\R^{d_x}$ and $K_Y\subset\R^{d_y}$ be compact, and let
\(\mu\in\cP_2(\R^{d_x})\) and \(\nu\in\cP_2(\R^{d_y})\) have supports
contained in \(K_X\) and \(K_Y\), respectively. With the domains in
\eqref{eq:compact-dual-domains}, the transform and functional in
\eqref{eq:compact-transform-functional}, and
\(S_\mu=\int xx^\top d\mu(x)\),
\[
  \wIGW^2(\mu,\nu)=\norm{S_\mu}_F^2+
  \inf_{A\in\mathcal A}\sup_{\substack{B\in\mathcal B\\u\in\cU_K}}
  \mathcal D_K(A,B,u).
\]
\end{theorem}

\begin{proof}
Every map feasible for the convex-order constraint belongs to $\cM$. Indeed,
$K=\operatorname{conv}(K_Y)$ is compact and convex, so
$z\mapsto\operatorname{dist}(z,K)$ is convex and
\[
 \int\operatorname{dist}(m(x),K)d\mu(x)
 \le \int\operatorname{dist}(y,K)d\nu(y)=0.
\]
Thus $m(x)\in K$ for $\mu$-almost every $x$.
Convex order on \(K\) is encoded by
\[
  \sup_{u\in\cU_K}\left\{\int u(m)d\mu-\int u\,d\nu\right\}
  =\begin{cases}0,&m_\#\mu\cx\nu,\\+\infty,&\text{otherwise}.
  \end{cases}
\]
For the zero case, use Strassen's martingale coupling and conditional Jensen on
$K$; for the infinite case, scale the restriction to $K$ of a global separating
convex test.
Together with the Fenchel identities, this gives
\[
  \wIGW^2-\norm{S_\mu}_F^2
  =\inf_{A\in\mathcal A}\inf_{m\in\cM}
  \sup_{(B,u)\in\mathcal B\times\cU_K}\mathcal L(m,A,B,u),
\]
where
\[
\begin{aligned}
  \mathcal L(m,A,B,u)
  =&\ 2\norm{A}_F^2-\norm{B}_F^2-\int u\,d\nu\\
  &+\int[u(m)+2m^\top Bm-4m^\top A^\top x]d\mu.
\end{aligned}
\]
Fix \(A\). For fixed \((B,u)\), the functional is convex and weakly lower
semicontinuous in \(m\): the integral of $u(m)$ is weakly lower semicontinuous
because $u$ is continuous convex, and the positive quadratic integral has the same
property. For fixed \(m\), the functional is continuous concave in \(B\) and
continuous linear in \(u\) for the Frobenius and uniform topologies. The set
$\cM$ is convex, norm bounded, and weakly closed (use the integral of the squared
distance to $K$), hence weakly compact. Thus every Sion hypothesis \cite{Sion} has now been
verified and $\inf_m$ may be exchanged with $\sup_{B,u}$.

For fixed $(A,B,u)$, the integrand is a Carath\'eodory function of $(x,z)$ on
$K_X\times K$. Its argmin correspondence is nonempty, compact-valued, and
measurable by the measurable minimum theorem
\cite[Theorem~18.19]{AliprantisBorder}. A measurable selector $m_{A,B,u}$ therefore
exists, belongs to $\cM$,
and gives
\[
 \inf_{m\in\cM}\int[\cdots]d\mu
 =\int\min_{z\in K}[\cdots]d\mu
 =\int Q_{A,B}^{K}u\,d\mu.
\]
\end{proof}

The formula separates the three roles in the compact problem: \(u\) enforces
convex order, \(A\) linearizes the cross moment between the source and the map, and
\(B\) linearizes the map second moment. To make the weak OT block explicit, for
$(A,B)\in\mathcal A\times\mathcal B$ define
\[
 c^0_{A,B}(x,z):=2z^\top Bz-4z^\top A^\top x
\]
and
\begin{equation}\label{eq:compact-wot-block}
\begin{aligned}
 \mathsf W^K_{A,B}(\mu,\nu)
 &:=\min_{\pi\in\Pi(\mu,\nu)}
   \int c^0_{A,B}\bigl(x,m_\pi(x)\bigr)\,d\mu(x)\\
 &=\min_{\substack{m\in\cM\\m_\#\mu\cx\nu}}
   \int c^0_{A,B}\bigl(x,m(x)\bigr)\,d\mu(x)\\
 &=\sup_{u\in\cU_K}
   \left\{-\int u\,d\nu+\int Q^K_{A,B}u\,d\mu\right\}.
\end{aligned}
\end{equation}
The coupling and map problems agree by the Jensen--Strassen characterization.
For fixed $(A,B)$, Sion's theorem exchanges the compact map minimization with the
potential supremum: the objective is weakly lower semicontinuous and convex in
$m$, and continuous and affine in $u$. This proves equality with the compact
potential dual in \eqref{eq:compact-wot-block}. The outer formulation is therefore
\begin{equation}\label{eq:compact-wot-envelope}
 \wIGW^2(\mu,\nu)
 =\norm{S_\mu}_F^2+
 \inf_{A\in\mathcal A}\sup_{B\in\mathcal B}
 \left\{2\norm A_F^2-\norm B_F^2+
 \mathsf W^K_{A,B}(\mu,\nu)\right\}.
\end{equation}
This regrouping preserves the order $\inf_A\sup_B$; it does not exchange the two
outer matrix optimizations. The inner barycentric map may be nonunique when $B$ is
singular. Section~\ref{sec:ridge-duality} adds a positive ridge to obtain the
corresponding coercive weak OT block without compact support.

\section{Ridge duality and the weak OT envelope}\label{sec:ridge-duality}

In Section~\ref{sec:compact-duality}, we derived an unregularized potential dual
by restricting the conditional means to a compact set. We now assume only finite
second moments. We first isolate the two matrix variables in the unregularized
moment formula. The resulting fixed-matrix cost may fail the coercivity required
by Theorem~\ref{thm:wot-blueprint}, so we add a ridge to obtain a coercive weak
OT block and its potential dual. A stronger ridge condition is needed only for
the final exchange of the two outer matrix optimizations.

\subsection{Unregularized matrix representation}

Let \(\mu\in\cP_2(\R^{d_x})\), \(\nu\in\cP_2(\R^{d_y})\), and write
\[
 M_2(\mu)=\int\norm{x}^2d\mu(x),
 \qquad
 M_2(\nu)=\int\norm{y}^2d\nu(y).
\]
Cauchy--Schwarz and Lemma~\ref{lem:feasible-set} place every Fenchel optimizer
\(A=M_m\), \(B=S_m\) in the compact matrix sets
\begin{equation}\label{eq:ridge-outer-domains}
\begin{aligned}
 \mathcal A_2
 &=\{A\in\R^{d_x\times d_y}:\norm{A}_F
     \le\sqrt{M_2(\mu)M_2(\nu)}\},\\
 \mathcal B_2
 &=\{B\in\mathbb S_+^{d_y}:\norm{B}_F\le M_2(\nu)\}.
\end{aligned}
\end{equation}
\paragraph{Notation for Sections~\ref{sec:ridge-duality}--\ref{sec:reconstruction}.}
The admissible mean maps form
\(\mathcal C_\nu=\{m\in L^2(\mu;\R^{d_y}):m_\#\mu\cx\nu\}\).
The outer variables range over \(\mathcal A_2\) and \(\mathcal B_2\) in
\eqref{eq:ridge-outer-domains}; all matrix norms and inner products are Frobenius,
and every \(B\in\mathcal B_2\) is symmetric positive semidefinite. For a mean map
\(m\), the matrices \(M_m\) and \(S_m\) are the cross moment and second moment
defined in \eqref{eq:wigw-moment-matrices}.
For \(m\in\mathcal C_\nu\), \(A\in\mathcal A_2\), and
\(B\in\mathcal B_2\), define
\begin{equation}\label{eq:unregularized-map-matrix-data}
\begin{aligned}
 c_{A,B}^0(x,z)
 &:=2z^\top Bz-4z^\top A^\top x,\\
 \mathcal H_0(m,A,B)
 &:=2\norm A_F^2-\norm B_F^2
   +\int c_{A,B}^0(x,m(x))d\mu(x).
\end{aligned}
\end{equation}

\begin{lemma}[Unregularized matrix representation]
\label{lem:unregularized-map-matrix}
Let \(\mu\in\cP_2(\R^{d_x})\) and
\(\nu\in\cP_2(\R^{d_y})\). Let \(\mathcal C_\nu\) be the admissible map set in
\eqref{eq:admissible-map-set}, and let \(\mathcal A_2,\mathcal B_2\) and
\(\mathcal H_0\) be defined by
\eqref{eq:ridge-outer-domains}--\eqref{eq:unregularized-map-matrix-data}.
Then, with \(S_\mu=\int xx^\top d\mu(x)\),
\begin{equation}\label{eq:unregularized-map-matrix-form}
\begin{aligned}
 \wIGW^2(\mu,\nu)-\norm{S_\mu}_F^2
 &=\inf_{m\in\mathcal C_\nu}\inf_{A\in\mathcal A_2}
   \sup_{B\in\mathcal B_2}\mathcal H_0(m,A,B)\\
 &=\inf_{A\in\mathcal A_2}\inf_{m\in\mathcal C_\nu}
   \sup_{B\in\mathcal B_2}\mathcal H_0(m,A,B).
\end{aligned}
\end{equation}
For each fixed feasible \(m\), the Fenchel optimizers are
\(A=M_m\) and \(B=S_m\), with \(M_m,S_m\) defined in
\eqref{eq:wigw-moment-matrices}.
\end{lemma}

\begin{proof}
Proposition~\ref{prop:moment} gives
\[
 \wIGW^2(\mu,\nu)-\norm{S_\mu}_F^2
 =\inf_{m\in\mathcal C_\nu}
   \{\norm{S_m}_F^2-2\norm{M_m}_F^2\}.
\]
Cauchy--Schwarz and Lemma~\ref{lem:feasible-set} give
\[
 \norm{M_m}_F\le\sqrt{M_2(\mu)M_2(\nu)},\qquad
 \norm{S_m}_F\le\Tr(S_m)=\norm m_{L^2(\mu)}^2\le M_2(\nu).
\]
Thus the Fenchel optimizers lie in the domains
\eqref{eq:ridge-outer-domains}. For each fixed \(m\in\mathcal C_\nu\),
\begin{align*}
 -2\norm{M_m}_F^2
 &=\inf_{A\in\mathcal A_2}
   \{2\norm A_F^2-4\ip{A}{M_m}_F\},\\
 \norm{S_m}_F^2
 &=\sup_{B\in\mathcal B_2}
   \{2\ip{B}{S_m}_F-\norm B_F^2\}.
\end{align*}
Combining these identities and using
\[
 -4\ip{A}{M_m}_F+2\ip{B}{S_m}_F
 =\int c_{A,B}^0(x,m(x))d\mu(x)
\]
gives the first equality in
\eqref{eq:unregularized-map-matrix-form}. The second follows because the two
minimizing infima commute.
\end{proof}

Lemma~\ref{lem:unregularized-map-matrix} is an exact finite-dimensional
reduction, but it does not yet provide a noncompact weak OT potential dual.
Theorem~\ref{thm:wot-blueprint} requires a lower bound
\begin{equation}\label{eq:section8-required-coercivity}
 c(x,z)\ge \alpha\norm z^2-C(1+\norm x^2)
 \qquad(\alpha>0,\ C<\infty).
\end{equation}
If \(B\) is singular, choose \(0\ne v\in\ker B\). Then
\(c_{A,B}^0(0,tv)=0\) for every \(t\), which contradicts
\eqref{eq:section8-required-coercivity} as \(|t|\to\infty\).
Since \(\mathcal B_2\) contains singular matrices, including \(B=0\),
Theorem~\ref{thm:wot-blueprint} cannot be applied uniformly to the
unregularized fixed-matrix costs.

\subsection{Ridge regularization and the weak OT block}

For \(\varepsilon>0\), define ridge wIGW by
\begin{equation}\label{eq:ridge-coupling-primal}
\begin{aligned}
  \wIGWeps^2(\mu,\nu)
  :=\inf_{\pi\in\Pi(\mu,\nu)}\Bigg\{
  &\iint\left(\ip{x}{x'}-
  \ip{m_\pi(x)}{m_\pi(x')}\right)^2d\mu(x)d\mu(x')\\
  &\quad+\varepsilon\int\norm{m_\pi(x)}^2d\mu(x)\Bigg\}.
\end{aligned}
\end{equation}
The added term is finite by conditional Jensen. Proposition~\ref{prop:map} and
Proposition~\ref{prop:moment} give, respectively,
\begin{equation}\label{eq:ridge-map-primal}
  \wIGWeps^2(\mu,\nu)
  =\inf_{m_\#\mu\cx\nu}\left\{
  \iint\left(\ip{x}{x'}-\ip{m(x)}{m(x')}\right)^2d\mu(x)d\mu(x')
  +\varepsilon\norm{m}_{L^2(\mu)}^2\right\},
\end{equation}
and
\begin{equation}\label{eq:ridge-moment-primal}
  \wIGWeps^2(\mu,\nu)=\norm{S_\mu}_F^2+
  \inf_{m\in\mathcal C_\nu}\left\{\norm{S_m}_F^2-2\norm{M_m}_F^2
  +\varepsilon\norm{m}_{L^2(\mu)}^2\right\}.
\end{equation}
The ridge changes the value by at most its strength times the target second
moment:
\[
 0\le \wIGWeps^2(\mu,\nu)-\wIGW^2(\mu,\nu)
 \le\varepsilon M_2(\nu).
\]
The lower bound is immediate. For the upper bound, evaluate the ridge objective
at an unregularized minimizing map from Theorem~\ref{thm:existence} and use
Lemma~\ref{lem:feasible-set}.

For \(A\in\R^{d_x\times d_y}\) and \(B\in\mathbb S_+^{d_y}\), set
\begin{equation}\label{eq:ridge-cost}
\begin{aligned}
 C_{B,\varepsilon}&:=2B+\varepsilon I_{d_y},\\
 c_{A,B}^\varepsilon(x,z)&:=z^\top C_{B,\varepsilon}z-4z^\top A^\top x.
\end{aligned}
\end{equation}
Thus \(c_{A,B}^\varepsilon=c_{A,B}^0+\varepsilon\norm z^2\), and Young's
inequality gives
\begin{equation}\label{eq:section8-ridge-coercivity}
 c_{A,B}^\varepsilon(x,z)
 \ge \frac{\varepsilon}{2}\norm z^2
      -\frac{8\norm A_F^2}{\varepsilon}\norm x^2,
 \qquad
 \nabla_{zz}^2c_{A,B}^\varepsilon
 \succeq2\varepsilon I_{d_y}.
\end{equation}
The first estimate verifies \eqref{eq:section8-required-coercivity}; the second
makes the cost uniformly strongly convex in \(z\). These two facts give
coercivity and uniqueness of the optimal barycentric mean map in the
fixed-matrix weak OT problem; they do not imply uniqueness of its coupling.

For brevity, write
\begin{equation}\label{eq:ridge-potential-class}
 \cU_2:=\mathcal U_2^{\rm cvx},
\end{equation}
where \(\mathcal U_2^{\rm cvx}\) is defined in
\eqref{eq:quadratic-convex-potentials}. Define
\begin{equation}\label{eq:ridge-cost-transform}
\begin{aligned}
 \mathsf W_{A,B}^\varepsilon(\mu,\nu)
 &:=\inf_{\pi\in\Pi(\mu,\nu)}
   \int c_{A,B}^\varepsilon(x,m_\pi(x))d\mu(x),\\
 Q_{A,B}^{(\varepsilon)}u(x)
 &:=\inf_{z\in\R^{d_y}}
  \{u(z)+c_{A,B}^\varepsilon(x,z)\}.
\end{aligned}
\end{equation}

\Needspace{6\baselineskip}
\begin{proposition}[Weak OT duality for the ridge cost]\label{prop:wot}
Let \(\mu\in\cP_2(\R^{d_x})\), \(\nu\in\cP_2(\R^{d_y})\),
\(A\in\R^{d_x\times d_y}\), \(B\in\mathbb S_+^{d_y}\), and
\(\varepsilon>0\). Let \(c_{A,B}^\varepsilon\),
\(Q_{A,B}^{(\varepsilon)}\), \(\mathsf W_{A,B}^\varepsilon\), and
\(\cU_2\) be defined by
\eqref{eq:ridge-cost}--\eqref{eq:ridge-cost-transform}. For a disintegration
\(\pi(dx,dy)=\mu(dx)\pi_x(dy)\), set
\(m_\pi(x)=\int y\,\pi_x(dy)\), and let \(\cx\) denote the convex order in
Definition~\ref{def:convex-order}. Then
\[
\begin{aligned}
 \mathsf W_{A,B}^\varepsilon(\mu,\nu)
 &=\min_{\pi\in\Pi(\mu,\nu)}
   \int c_{A,B}^\varepsilon(x,m_\pi(x))d\mu(x)\\
 &=\min_{\substack{m\in L^2(\mu;\R^{d_y})\\m_\#\mu\cx\nu}}
   \int c_{A,B}^\varepsilon(x,m(x))d\mu(x)\\
 &=\sup_{u\in\cU_2}\left\{\int Q_{A,B}^{(\varepsilon)}u\,d\mu-
   \int u\,d\nu\right\}.
\end{aligned}
\]
All minimizing couplings have the same conditional mean
\(m_{A,B}\), uniquely determined in \(L^2(\mu;\R^{d_y})\). For fixed
\((A,B)\), the coupling objective
\[
 \pi\longmapsto
 \int c_{A,B}^\varepsilon\bigl(x,m_\pi(x)\bigr)d\mu(x)
\]
is convex on \(\Pi(\mu,\nu)\). It is strongly convex in the induced mean map
\(m_\pi\), but it need not be strictly convex in \(\pi\): distinct couplings can
have the same conditional mean. Moreover, for
\(A'\in\R^{d_x\times d_y}\) and \(B'\in\mathbb S_+^{d_y}\), with the same
\(\varepsilon\),
\begin{equation}\label{eq:ridge-oracle-lipschitz}
\begin{aligned}
 \left|\mathsf W_{A,B}^\varepsilon(\mu,\nu)
       -\mathsf W_{A',B'}^\varepsilon(\mu,\nu)\right|
 \le{}&4\sqrt{M_2(\mu)M_2(\nu)}\,\norm{A-A'}_F\\
 &+2M_2(\nu)\,\norm{B-B'}_F .
\end{aligned}
\end{equation}
\end{proposition}

\begin{proof}
The bounds in \eqref{eq:section8-ridge-coercivity}, together with the quadratic
upper bound obtained from Young's inequality, verify all hypotheses of
Theorem~\ref{thm:wot-blueprint}. That theorem gives existence, coupling/map
equality, and the potential dual. Uniform strong convexity gives uniqueness of
the minimizing mean map. If \(\pi_t=t\pi_1+(1-t)\pi_2\), affinity of
disintegration with a common first marginal gives
\(m_{\pi_t}=tm_{\pi_1}+(1-t)m_{\pi_2}\). Convexity of
\(z\mapsto c_{A,B}^\varepsilon(x,z)\) therefore proves convexity in the coupling.
Strong convexity acts only after passage to the conditional mean, so it does not
imply strict convexity in \(\pi\). For the last estimate, evaluate the two fixed-matrix
costs at a common feasible map and use
\(\norm m_{L^2}^2\le M_2(\nu)\) and
\(\int\norm{x}\norm{m(x)}d\mu(x)
\le\sqrt{M_2(\mu)M_2(\nu)}\).
\end{proof}

\subsection{The ridge envelope and outer minimax equality}

For \(m\in\mathcal C_\nu\), \(A\in\mathcal A_2\), and
\(B\in\mathcal B_2\), define
\begin{equation}\label{eq:ridge-map-matrix-functional}
\begin{aligned}
 \mathcal H_\varepsilon(m,A,B)
 &:=2\norm A_F^2-\norm B_F^2
    +\int c_{A,B}^\varepsilon(x,m(x))d\mu(x)\\
 &=2\norm A_F^2-\norm B_F^2
   -4\ip{A}{M_m}_F+2\ip{B}{S_m}_F
   +\varepsilon\norm m_{L^2(\mu)}^2,
\end{aligned}
\end{equation}
and
\begin{equation}\label{eq:ridge-envelope-functionals}
\begin{aligned}
 F_\varepsilon(A,B)
 &:=2\norm A_F^2-\norm B_F^2+\mathsf W_{A,B}^\varepsilon(\mu,\nu),\\
 \mathcal D_\varepsilon(A,B,u)
 &:=2\norm A_F^2-\norm B_F^2-
   \int u\,d\nu+\int Q_{A,B}^{(\varepsilon)}u\,d\mu.
\end{aligned}
\end{equation}

\paragraph{Envelope notation.}
Hereafter, \(c_{A,B}^\varepsilon\) is the ridge cost in
\eqref{eq:ridge-cost}, \(\mathsf W_{A,B}^\varepsilon\) is its barycentric weak
OT value in \eqref{eq:ridge-cost-transform}, and
\(\mathcal H_\varepsilon,F_\varepsilon,\mathcal D_\varepsilon\) are the
functional in the map and matrix variables, the outer functional, and the
potential functional in
\eqref{eq:ridge-map-matrix-functional}--\eqref{eq:ridge-envelope-functionals}.
The potential class is \(\cU_2\) from \eqref{eq:ridge-potential-class}.
For every \(B\in\mathcal B_2\),
\(2B+\varepsilon I_{d_y}\succeq\varepsilon I_{d_y}\), and this strict
positivity persists on an open neighborhood of
\(\mathcal A_2\times\mathcal B_2\). We apply
Theorem~\ref{thm:danskin-envelope} on that neighborhood and restrict the
resulting gradients to the admissible matrix domains.

\begin{theorem}[Weak OT envelope for ridge wIGW]\label{thm:ridge-dual}
Let \(\mu\in\cP_2(\R^{d_x})\), \(\nu\in\cP_2(\R^{d_y})\), and
\(\varepsilon>0\). Use the notation above, and let
\(\wIGWeps^2(\mu,\nu)\) be the ridge coupling primal in
\eqref{eq:ridge-coupling-primal}.
Then, with \(S_\mu=\int xx^\top d\mu(x)\),
\begin{equation}\label{eq:ridge-map-matrix-form}
\begin{aligned}
 \wIGWeps^2(\mu,\nu)-\norm{S_\mu}_F^2
 &=\inf_{A\in\mathcal A_2}\inf_{m\in\mathcal C_\nu}
   \sup_{B\in\mathcal B_2}\mathcal H_\varepsilon(m,A,B)\\
 &=\inf_{A\in\mathcal A_2}\sup_{B\in\mathcal B_2}
   \inf_{m\in\mathcal C_\nu}\mathcal H_\varepsilon(m,A,B)\\
 &=\inf_{A\in\mathcal A_2}\sup_{B\in\mathcal B_2}
   F_\varepsilon(A,B)\\
 &=\inf_{A\in\mathcal A_2}
   \sup_{\substack{B\in\mathcal B_2\\u\in\cU_2}}
   \mathcal D_\varepsilon(A,B,u).
\end{aligned}
\end{equation}
\end{theorem}

\begin{proof}
For each feasible $m$, the two Fenchel identities in the proof of
Lemma~\ref{lem:unregularized-map-matrix} have optimizers $A=M_m$ and
$B=S_m$. Adding the ridge term, which is independent of both matrices, therefore
gives
\[
 \norm{S_m}_F^2-2\norm{M_m}_F^2+\varepsilon\norm m_{L^2}^2
 =\inf_{A\in\mathcal A_2}\sup_{B\in\mathcal B_2}
   \mathcal H_\varepsilon(m,A,B).
\]
Taking the infimum over $m$ and commuting the two minimizing infima proves the
first line of \eqref{eq:ridge-map-matrix-form}. With $A$ fixed,
Lemma~\ref{lem:feasible-set} and Sion's theorem \cite{Sion} exchange only $m$
and $B$, giving the second line. Proposition~\ref{prop:wot} identifies the
inner map minimum with $\mathsf W_{A,B}^\varepsilon$, and then with its dual over
convex potentials. These two identifications give the last two lines.
\end{proof}

\paragraph{Minimax exchange for fixed \(A\).}
Fix \(A\in\mathcal A_2\), equip \(\mathcal C_\nu\) with the weak \(L^2\)
topology, and equip \(\mathcal B_2\) with the Frobenius topology.
Lemma~\ref{lem:feasible-set} shows that \(\mathcal C_\nu\) is weakly compact and
convex; \(\mathcal B_2\) is compact and convex in finite dimension. For
\(B\in\mathcal B_2\), positivity of \(B\) gives
\(2B+\varepsilon I_{d_y}\succeq\varepsilon I_{d_y}\), so
\(m\mapsto\mathcal H_\varepsilon(m,A,B)\) is convex and weakly lower
semicontinuous. For fixed \(m\), the terms that depend on \(B\) are
\(2\ip{B}{S_m}_F-\norm B_F^2\); hence
\(B\mapsto\mathcal H_\varepsilon(m,A,B)\) is continuous and concave. Sion's
theorem therefore gives
\[
 \inf_{m\in\mathcal C_\nu}\sup_{B\in\mathcal B_2}
 \mathcal H_\varepsilon(m,A,B)
 =
 \sup_{B\in\mathcal B_2}\inf_{m\in\mathcal C_\nu}
 \mathcal H_\varepsilon(m,A,B).
\]
This is the exchange in the second line of
\eqref{eq:ridge-map-matrix-form}: it exchanges \(m\) and \(B\) while leaving
\(A\) fixed. Proposition~\ref{prop:wot} then identifies the inner minimum with
the ridge weak OT value \(\mathsf W_{A,B}^{\varepsilon}(\mu,\nu)\), giving the
third line.

For every positive ridge, this argument uses only convexity in the minimizing
mean variable and concavity in \(B\), with \(A\) fixed. Although the objective is
convex in \(A\) for fixed \((m,B)\), these separate convexity statements do not
give joint convexity in \((A,m)\): the term
\(-4\ip{A}{M_m}_F\) couples the two variables. The spectral ridge condition in
Theorem~\ref{thm:swap} supplies precisely the additional joint convexity needed
for the outer saddle formulation.

\paragraph{From the IGW OT envelope to the wIGW weak OT envelope.}
The ordinary IGW envelope in Theorem~\ref{thm:ordinary-igw-envelope} and the
ridge wIGW envelope in Theorem~\ref{thm:ridge-dual} have parallel variational
forms:
\[
\begin{aligned}
 \IGW^2(\mu,\nu)-\norm{S_\mu}_F^2-\norm{S_\nu}_F^2
 &=\min_{A\in\R^{d_x\times d_y}}
   \left\{2\norm A_F^2+\mathsf{OT}_{c_A}(\mu,\nu)\right\},\\
 \wIGWeps^2(\mu,\nu)-\norm{S_\mu}_F^2
 &=\inf_{A\in\mathcal A_2}\sup_{B\in\mathcal B_2}
   \left\{2\norm A_F^2-\norm B_F^2+
   \mathsf W_{A,B}^{\varepsilon}(\mu,\nu)\right\}.
\end{aligned}
\]
The two inner transport problems are
\[
\begin{aligned}
 \mathsf{OT}_{c_A}(\mu,\nu)
 &=\inf_{\pi\in\Pi(\mu,\nu)}\int c_A(x,y)d\pi(x,y),\\
 \mathsf W_{A,B}^{\varepsilon}(\mu,\nu)
 &=\inf_{\pi\in\Pi(\mu,\nu)}
   \int c_{A,B}^{\varepsilon}\bigl(x,m_\pi(x)\bigr)d\mu(x).
\end{aligned}
\]
Ordinary IGW evaluates the bilinear cost \(c_A\) directly at paired points
\((x,y)\), so its inner problem is ordinary OT. Ridge wIGW evaluates
\(c_{A,B}^{\varepsilon}\) at the conditional mean
\(m_\pi(x)=\E_\pi[Y\mid X=x]\), so its inner problem is barycentric weak OT.
The target second moment \(S_\nu\) is fixed in ordinary IGW, and the matrix
\(A\) represents the coupling dependent cross moment. In wIGW, both
\(M_{m_\pi}\) and \(S_{m_\pi}\) vary with the coupling: \(A\) represents the
cross moment, and \(B\) represents the second moment of the conditional mean.
The ridge adds \(\varepsilon I_{d_y}\) to the quadratic weak OT cost, giving
the coercivity and strong convexity used in Proposition~\ref{prop:wot}. At an
outer optimizer, Proposition~\ref{prop:reconstruction} proves
\(A^\star=M_{m^\star}\) and \(B^\star=S_{m^\star}\).

\Needspace{14\baselineskip}
\paragraph{Outer \(A\)--\(B\) minimax equality.}
The preceding exchange concerns \(m\) and \(B\), with \(A\) fixed. For every
\(\varepsilon>0\), Theorem~\ref{thm:ridge-dual} therefore gives the ordered
envelope \(\inf_{A\in\mathcal A_2}\sup_{B\in\mathcal B_2}
F_\varepsilon(A,B)\). To justify reversing these two outer optimizations by
Sion's theorem, we establish joint convexity in \((A,m)\). Set
\(\lambda_X:=\lambda_{\max}(S_\mu)\) and impose
\[
 \varepsilon>0
 \quad\text{and}\quad
 \varepsilon\ge 2\lambda_X.
\]
Under this condition, \(\mathcal H_\varepsilon(\cdot,\cdot,B)\) is jointly
convex in \((m,A)\) for every fixed \(B\). Because
\(\pi\mapsto m_\pi\) is affine on \(\Pi(\mu,\nu)\), the equivalent coupling
formulation is jointly convex in \((\pi,A)\). For fixed \((A,m)\), or fixed
\((A,\pi)\), the objective is concave in \(B\). Thus the formulation is a
convex--concave saddle problem with minimizing block \((A,m)\), equivalently
\((A,\pi)\), and maximizing variable \(B\). After partial minimization,
\(F_\varepsilon\) is convex in \(A\) and concave in \(B\), so a second
application of Sion's theorem gives the outer minimax equality stated next.
Equality at the spectral threshold is sufficient here;
Section~\ref{sec:algorithm} uses the strict condition
\(\varepsilon>2\lambda_X\) to obtain a positive convergence modulus.

\Needspace{12\baselineskip}
\begin{theorem}[Outer minimax equality]\label{thm:swap}
Let \(\mu\in\cP_2(\R^{d_x})\), \(\nu\in\cP_2(\R^{d_y})\), and set
\(S_\mu=\int xx^\top d\mu(x)\) and
\(\lambda_X=\lambda_{\max}(S_\mu)\). Let \(\mathcal A_2,\mathcal B_2\) and
\(F_\varepsilon\) be defined by
\eqref{eq:ridge-outer-domains} and \eqref{eq:ridge-envelope-functionals}.
If \(\varepsilon>0\) and \(\varepsilon\ge2\lambda_X\), then
\(A\mapsto F_\varepsilon(A,B)\) is convex and continuous on
\(\mathcal A_2\) for every \(B\in\mathcal B_2\), while
\(B\mapsto F_\varepsilon(A,B)\) is concave and continuous on
\(\mathcal B_2\) for every \(A\in\mathcal A_2\). Consequently,
\[
 g_\varepsilon(A):=\max_{B\in\mathcal B_2}F_\varepsilon(A,B)
\]
is convex and continuous on \(\mathcal A_2\), and the outer problem is the
convex minimization \(\min_{A\in\mathcal A_2}g_\varepsilon(A)\). Moreover,
\[
 \inf_{A\in\mathcal A_2}\sup_{B\in\mathcal B_2}F_\varepsilon(A,B)
 =\sup_{B\in\mathcal B_2}\inf_{A\in\mathcal A_2}F_\varepsilon(A,B).
\]
\end{theorem}

\begin{proof}
Write \(Tm=M_m\). Since \(T^*T\preceq\lambda_XI\),
\[
 2\norm A_F^2-4\ip{A}{Tm}_F+\varepsilon\norm m_{L^2}^2
 =2\norm{A-Tm}_F^2+
  \ip{m}{(\varepsilon I-2T^*T)m}_{L^2}
\]
is jointly convex in \((A,m)\) when
\(\varepsilon\ge2\lambda_X\). The \(B\)-quadratic preserves convexity in \(m\),
so partial minimization over \(\mathcal C_\nu\) makes
\(F_\varepsilon(\cdot,B)\) convex. For fixed \(A\), the inner value is an
infimum of affine functions of \(B\), and
\(-\norm B_F^2\) makes \(F_\varepsilon(A,\cdot)\) concave.
Continuity follows from \eqref{eq:ridge-oracle-lipschitz} and the outer
quadratic terms. The pointwise maximum over the compact set \(\mathcal B_2\)
therefore makes \(g_\varepsilon\) convex and continuous. Sion's theorem
\cite{Sion} on the compact convex sets
\(\mathcal A_2,\mathcal B_2\) gives the asserted minimax equality.
\end{proof}

Thus every positive ridge yields the weak OT envelope and potential dual. The
spectral condition in Theorem~\ref{thm:swap} yields the outer minimax equality;
Section~\ref{sec:algorithm} uses the strict inequality
\(\varepsilon>2\lambda_X\) to obtain a positive convergence modulus. The next
section compares the resulting formulations, and
Section~\ref{sec:reconstruction} reconstructs an optimal coupling from an exact
weak OT block.

\section{Equivalent formulations and computational roles}
\label{sec:equivalent-formulations}

The projection identity, compact dual, and ridge envelope describe the construction
at different levels. We collect their roles in
Table~\ref{tab:equivalent-forms} before introducing the algorithm. Its first four rows equal the
unregularized \(\wIGW^2(\mu,\nu)\); its final two equal
\(\wIGWeps^2(\mu,\nu)\) and display the weak OT oracle together with its dual
over convex potentials.

\begin{table}[H]
\centering
\footnotesize
\setlength{\tabcolsep}{4pt}
\renewcommand{\arraystretch}{1.05}
\captionsetup{skip=4pt}
\begin{tabularx}{\textwidth}{@{}>{\raggedright\arraybackslash}p{.20\textwidth}X
  >{\raggedright\arraybackslash}p{.25\textwidth}@{}}
\toprule
View & Expression & Conditions and role\\
\midrule
Conditional coupling &
$\displaystyle \inf_{\pi\in\Pi(\mu,\nu)}
 \iint(\ip{x}{x'}-\ip{m_\pi(x)}{m_\pi(x')})^2d\mu d\mu$ &
$\mu,\nu\in\cP_2$. Definition~\ref{def:wigw}. This formulation keeps the full
coupling and its conditional laws visible.\\

Map and moments &
$\displaystyle \norm{S_\mu}_F^2+
 \inf_{m_\#\mu\cx\nu}\{\norm{S_m}_F^2-2\norm{M_m}_F^2\}$ &
$\mu,\nu\in\cP_2$. Propositions~\ref{prop:map}--\ref{prop:moment}. This form
exposes the convex order constraint and finite moments.\\

Projection in convex order &
$\displaystyle \inf_{\eta\cx\nu}\IGW^2(\mu,\eta)$ &
$\mu,\nu\in\cP_2$. Theorem~\ref{thm:projection}. This identity gives the geometric
interpretation and comparison with ordinary IGW.\\

Compact $u$-dual &
$\displaystyle \norm{S_\mu}_F^2+
 \inf_{A\in\mathcal A}\sup_{B\in\mathcal B,\,u\in\cU_K}
 \mathcal D_K(A,B,u)$ &
Compact supports. Theorem~\ref{thm:compact-dual}. Here $u$ enforces convex order,
while $A,B$ linearize the moment terms.\\
\midrule
Ridge weak OT envelope &
$\displaystyle \norm{S_\mu}_F^2+
 \inf_{A\in\mathcal A_2}\sup_{B\in\mathcal B_2}
 \{2\norm A_F^2-\norm B_F^2+\mathsf W_{A,B}^{\varepsilon}(\mu,\nu)\}$ &
$\mu,\nu\in\cP_2$, $\varepsilon>0$. Theorem~\ref{thm:ridge-dual}. This is the
principal computational form because the inner problem is weak OT.\\

Ridge $u$-dual &
$\displaystyle \norm{S_\mu}_F^2+
 \inf_{A\in\mathcal A_2}\sup_{B\in\mathcal B_2,\,u\in\cU_2}
 \mathcal D_\varepsilon(A,B,u)$ &
$\mu,\nu\in\cP_2$, $\varepsilon>0$. Proposition~\ref{prop:wot} and
Theorem~\ref{thm:ridge-dual}. This form supports dual certificates and analysis.\\
\bottomrule
\end{tabularx}
\caption{Equivalent formulations of wIGW and its ridge version.}
\label{tab:equivalent-forms}
\end{table}

\FloatBarrier
\section{Exact reconstruction from the weak OT block}\label{sec:reconstruction}

The weak OT envelope returns a coupling for each pair of outer matrices, whereas
the ridge primal is expressed through the moments of its conditional mean. Exact
reconstruction requires these two descriptions to agree at an outer optimizer.
Within each fixed-matrix weak OT block, the ridge makes the conditional mean map
unique. A martingale kernel then supplies the conditional variation needed to
realize the target marginal.
We derive compatibility from the outer optimality conditions and recover a coupling
by martingale gluing.

\begin{proposition}[Compatibility, reconstruction, and gluing]
\label{prop:reconstruction}
Let \(\mu\in\cP_2(\R^{d_x})\), \(\nu\in\cP_2(\R^{d_y})\), and
\(\varepsilon>0\). Let \(\mathcal A_2,\mathcal B_2\) and
\(F_\varepsilon\) be defined by \eqref{eq:ridge-outer-domains} and
\eqref{eq:ridge-envelope-functionals}. The nested outer problem has an optimizer
\[
 A^\star\in\argmin_{A\in\mathcal A_2}
 \max_{B\in\mathcal B_2}F_\varepsilon(A,B),
 \qquad
 B^\star\in\argmax_{B\in\mathcal B_2}F_\varepsilon(A^\star,B),
\]
where \(B^\star\) is the unique maximizer. Let
\(\pi^\star\in\Pi(\mu,\nu)\) minimize
\(\mathsf W_{A^\star,B^\star}^\varepsilon(\mu,\nu)\) from
\eqref{eq:ridge-cost-transform}, set
\(m^\star(x)=\E_{\pi^\star}[Y\mid X=x]\), and define
\[
 M_{m^\star}=\int x\,m^\star(x)^\top d\mu(x),\qquad
 S_{m^\star}=\int m^\star(x)m^\star(x)^\top d\mu(x).
\]
Then compatibility holds:
\[
  A^\star=M_{m^\star},\qquad B^\star=S_{m^\star}.
\]
Consequently, \(m^\star\) solves the ridge map primal
\eqref{eq:ridge-map-primal}. Moreover,
\[
  \kappa^\star(dy\mid z)=\mathcal L_{\pi^\star}(Y\in dy\mid m^\star(X)=z)
\]
is a martingale kernel from \((m^\star)_\#\mu\) to \(\nu\), and
\[
  \widehat\pi^\star(dx,dy)=\mu(dx)\kappa^\star(dy\mid m^\star(x))
\]
belongs to \(\Pi(\mu,\nu)\) and solves the ridge coupling primal
\eqref{eq:ridge-coupling-primal}.
\end{proposition}

\begin{proof}
For $\varepsilon>0$, Proposition~\ref{prop:wot} gives a unique inner barycentric
map $m_{A,B}$ and Danskin's theorem
(Theorem~\ref{thm:danskin-envelope}) gives
\[
 \nabla_AF_\varepsilon(A,B)=4(A-M_{m_{A,B}}),\qquad
 \nabla_BF_\varepsilon(A,B)=2(S_{m_{A,B}}-B).
\]
The estimate \eqref{eq:ridge-oracle-lipschitz}, together with the outer
quadratic terms, proves that $F_\varepsilon$ is continuous. Its outer sets are
compact, so the nested outer optimizer exists.

For each fixed $A$, $F_\varepsilon(A,\cdot)$ is $2$-strongly concave and has a
unique maximizer. At $B^\star$, the constrained first-order inequality, tested with
$S_{m^\star}\in\mathcal B_2$, gives
\[
 \ip{2(S_{m^\star}-B^\star)}{S_{m^\star}-B^\star}_F\le0,
\]
so $B^\star=S_{m^\star}$, including when the ball constraint is active. Now let
$g(A)=\max_{B\in\mathcal B_2}F_\varepsilon(A,B)$. Uniqueness of the maximizer
and Theorem~\ref{thm:danskin-envelope} make $g$ differentiable. At its constrained minimizer
$A^\star$, testing the first-order inequality with
$M_{m^\star}\in\mathcal A_2$ gives
\[
 \ip{4(A^\star-M_{m^\star})}{M_{m^\star}-A^\star}_F\ge0,
\]
and therefore $A^\star=M_{m^\star}$.

Adding the outer regularizers to the weak OT value and using this compatibility gives
\[
  F_\varepsilon(A^\star,B^\star)=\norm{S_{m^\star}}_F^2
  -2\norm{M_{m^\star}}_F^2+\varepsilon\norm{m^\star}_{L^2}^2.
\]
The nested min--max identity of Theorem~\ref{thm:ridge-dual} and
Proposition~\ref{prop:wot} identify this with the nonconstant part
\(\wIGWeps^2(\mu,\nu)-\norm{S_\mu}_F^2\) of the optimal ridge primal value. Let
\(M=m^\star(X)\). The tower property gives
\[
  \E[Y\mid M]=\E[\E[Y\mid X]\mid M]=\E[m^\star(X)\mid M]=M.
\]
Thus the disintegration through \(M\) is martingale and has target marginal \(\nu\).
Gluing it after \(m^\star\) preserves the conditional mean and the optimal value.
\end{proof}

\section{Finite algorithm and convergence}\label{sec:algorithm}

Following the primal weak OT algorithm of Paty, Chon\'e, and Kramarz
\cite{PatyChoneKramarz2022}, we discretize the fixed-matrix weak OT block, retain
the convexity established at the measure level in Proposition~\ref{prop:wot}, and
optimize the resulting unregularized convex problem approximately by normalized
KL mirror descent. Sinkhorn scaling computes the ideal KL mirror projection. Entropy
serves only as the mirror geometry: no entropy penalty is added to the weak OT
objective. Thus we apply KL mirror descent to unregularized weak OT; we do not solve
an entropic OT problem. A genuinely entropy-regularized weak OT problem would have
a different value and an additional
regularization bias; it is not the inner problem studied here. The resulting
approximate weak OT solution supplies an inexact oracle to the projected outer
\(A\)--\(B\) iteration analyzed below.

\subsection{Finite numerical formulation and projected algorithm}
\label{subsec:finite-projected-algorithm}

The reconstruction result identifies the moments that the outer iterations must
match. For finite measures, the conditional means are linear functions of the
coupling matrix, so the weak OT block becomes a convex optimization over a
transport polytope.

Let \(\mu=\sum_{i=1}^n a_i\delta_{x_i}\in\cP_2(\R^{d_x})\) and
\(\nu=\sum_{j=1}^p b_j\delta_{y_j}\in\cP_2(\R^{d_y})\), where all masses are
positive and each mass vector sums to one. Define
\begin{equation}\label{eq:finite-coupling-moments}
\begin{aligned}
 \Pi(a,b)
 &:=\{P\in\R_+^{n\times p}:P\mathbf 1_p=a,\ P^\top\mathbf 1_n=b\},\\
 m_i(P)&:=\frac{1}{a_i}\sum_{j=1}^pP_{ij}y_j,\\
 M_P&:=\sum_{i=1}^n a_i x_i m_i(P)^\top,\\
 S_P&:=\sum_{i=1}^n a_i m_i(P)m_i(P)^\top.
\end{aligned}
\end{equation}

\paragraph{Discrete reconstruction.}
At an exact outer solution, represent the inner coupling from
Proposition~\ref{prop:reconstruction} by \(P^\star\in\Pi(a,b)\) and write
\[
  m_i^\star=m_i(P^\star)=\frac{1}{a_i}\sum_j P_{ij}^\star y_j.
\]
If these barycenters are distinct, the reconstructed kernel is
\(K_{ij}=P_{ij}^\star/a_i\). If several rows share a barycenter \(z_\ell\), let
\(I_\ell=\{i:m_i^\star=z_\ell\}\),
\(\alpha_\ell=\sum_{i\in I_\ell}a_i\), and
\[
  K_{\ell j}=\frac{1}{\alpha_\ell}\sum_{i\in I_\ell}P_{ij}^\star.
\]
Then
\[
  \sum_j K_{\ell j}=1,\qquad
  \sum_\ell\alpha_\ell K_{\ell j}=b_j,\qquad
  \sum_j K_{\ell j}y_j=z_\ell.
\]

\paragraph{Numerical diagnostic.}
Proposition~\ref{prop:reconstruction} guarantees compatibility for an exact optimizer
of the full nested problem. At arbitrary $(A,B)$, compatibility is assessed through
\(\norm{A-M_m}_F\), \(\norm{B-S_m}_F\), and the marginal and martingale residuals.

\paragraph{Mirror descent for the weak OT block at fixed \(A\) and \(B\).}
For \(\varepsilon>0\), \(A\in\R^{d_x\times d_y}\), and
\(B\in\mathbb S_+^{d_y}\), let \(C_{B,\varepsilon}=2B+\varepsilon I\). The
fixed-\((A,B)\) oracle is the convex program
\begin{equation}\label{eq:finite-weak-ot-oracle}
\begin{aligned}
  f_{A,B}(P)
  &:=\sum_i a_i[m_i(P)^\top C_{B,\varepsilon}m_i(P)-4x_i^\top A m_i(P)],\\
  \mathsf W_{A,B}^\varepsilon(\mu,\nu)
  &=\min_{P\in\Pi(a,b)} f_{A,B}(P).
\end{aligned}
\end{equation}
For nonnegative matrices \(Q,R\), define the generalized discrete relative
entropy
\[
 \KL(Q\,\|\,R)
 :=\sum_{i,j}\left[
 Q_{ij}\log\frac{Q_{ij}}{R_{ij}}-Q_{ij}+R_{ij}\right],
\]
with \(0\log(0/r)=0\) for \(r\ge0\), and with a summand equal to \(+\infty\)
when \(Q_{ij}>0=R_{ij}\). We adapt the primal KL mirror ascent method of Paty,
Chon\'e, and Kramarz \cite{PatyChoneKramarz2022} to our minimizing oracle as a
normalized KL mirror descent step. Sinkhorn computes the KL projection onto
\(\Pi(a,b)\); the KL term defines the mirror geometry and does not regularize the
weak OT objective. The explicit quadratic gradient, its normalization,
and the surrounding projected \(A\)--\(B\) iteration are specific to the present
fixed-matrix oracle.

\begin{proposition}[Finite convex weak OT oracle and normalized KL mirror step]
\label{prop:finite-inner-mirror}
Fix \(A\in\R^{d_x\times d_y}\), \(B\in\mathbb S_+^{d_y}\), and
\(\varepsilon>0\), and let
\(\alpha_B=\lambda_{\min}(C_{B,\varepsilon})\ge\varepsilon\). The function
\(P\mapsto f_{A,B}(P)\) in \eqref{eq:finite-weak-ot-oracle} is convex and
differentiable on \(\Pi(a,b)\), with
\begin{equation}\label{eq:finite-oracle-gradient}
  \frac{\partial f_{A,B}}{\partial P_{ij}}
  =\ip{2C_{B,\varepsilon}m_i(P)-4A^\top x_i}{y_j}.
\end{equation}
For \(P,Q\in\Pi(a,b)\), it satisfies
\begin{equation}\label{eq:finite-mean-strong-convexity}
\begin{aligned}
 f_{A,B}(Q)\ge{}&f_{A,B}(P)
 +\ip{\nabla f_{A,B}(P)}{Q-P}_F\\
 &+\alpha_B\sum_i a_i\norm{m_i(Q)-m_i(P)}^2.
\end{aligned}
\end{equation}
Thus the objective is \(2\alpha_B\)-strongly convex in its induced barycentric
vector, but need not be strictly or strongly convex in \(P\) when distinct plans
have the same barycenters.

For a strictly positive \(P\in\Pi(a,b)\) and \(\gamma>0\), set
\[
 G(P):=\nabla f_{A,B}(P),
 \qquad c(P):=\max\{\norm{G(P)}_\infty,10^{-12}\}.
\]
The ideal normalized KL mirror step is
\begin{equation}\label{eq:finite-mirror-step}
  \widetilde P_{ij}=P_{ij}\exp\left(-\frac{\gamma}{c(P)}
  \frac{\partial f_{A,B}}{\partial P_{ij}}\right),
  \qquad
  P^+=\argmin_{Q\in\Pi(a,b)}\KL(Q\,\|\,\widetilde P).
\end{equation}
Equivalently, \(P^+\) uniquely minimizes the linearization of \(f_{A,B}\) at
\(P\) plus \(c(P)\KL(\cdot\,\|\,P)/\gamma\). Exact Sinkhorn scaling computes
this KL projection. The floored, finitely scaled, and marginally repaired step in
Algorithm~\ref{alg:implemented-ab} is an inexact realization of
\eqref{eq:finite-mirror-step}.
\end{proposition}

\begin{proof}
The affine formula for \(m_i(P)\) gives the gradient
\eqref{eq:finite-oracle-gradient}. Expanding the two row quadratics gives
\eqref{eq:finite-mean-strong-convexity}. The first-order equation for the KL
proximal problem gives the multiplicative kernel \(\widetilde P\); imposing the
two marginals gives its KL projection onto \(\Pi(a,b)\).
\end{proof}

For an exact minimizer \(P\) of \eqref{eq:finite-weak-ot-oracle},
\[
  \nabla_A F_\varepsilon=4(A-M_P),
  \qquad \nabla_B F_\varepsilon=2(S_P-B).
\]
Although the minimizing plan need not be unique, Proposition~\ref{prop:finite-inner-mirror}
and Proposition~\ref{prop:wot} show that all minimizing plans have the same
barycentric vector. Their parameter derivatives therefore agree, so
Theorem~\ref{thm:danskin-envelope} differentiates the
envelope while holding an exact minimizing plan fixed. Differentiating
$2\norm A_F^2-4\ip A{M_P}_F$ and
$-\norm B_F^2+2\ip B{S_P}_F$ gives these gradient formulas.

Projected outer updates require compact domains determined by the data. At a
compatible outer solution, the moment identities provide such bounds.

\begin{proposition}[Exact data-dependent outer balls]\label{prop:balls}
Let \(\mu=\sum_{i=1}^n a_i\delta_{x_i}\in\cP_2(\R^{d_x})\) and
\(\nu=\sum_{j=1}^p b_j\delta_{y_j}\in\cP_2(\R^{d_y})\) be probability
measures with positive masses, and let \(P\in\Pi(a,b)\). Define \(M_P,S_P\) by
\eqref{eq:finite-coupling-moments} and set
\(M_2(\mu)=\sum_i a_i\norm{x_i}^2\) and
\(M_2(\nu)=\sum_j b_j\norm{y_j}^2\). If outer matrices
\(A\in\R^{d_x\times d_y}\) and \(B\in\mathbb S_+^{d_y}\) are compatible
with \(P\), meaning \(A=M_P\) and \(B=S_P\), then
\[
  \norm{A}_F\le R_A:=\sqrt{M_2(\mu)M_2(\nu)},
  \qquad \norm{B}_F\le R_B:=M_2(\nu),\quad B\succeq0.
\]
In particular, these bounds hold for every compatible outer optimizer and its
associated exact inner solution.
\end{proposition}

\begin{proof}
Conditional Jensen gives \(\sum_i a_i\norm{m_i(P)}^2\le M_2(\nu)\). Hence
\(\norm{M_P}_F^2\le M_2(\mu)M_2(\nu)\). Since \(S_P\succeq0\),
\(\norm{S_P}_F\le\Tr(S_P)\le M_2(\nu)\). At compatibility, \(A=M_P\) and
\(B=S_P\).
\end{proof}

The outer updates can therefore use the compatibility gradients while the inner
coupling is refined independently.

For the retention rule used by the implementation, define the direct ridge
primal at a coupling \(P\in\Pi(a,b)\) by
\begin{equation}\label{eq:finite-ridge-primal}
 J_{\rm w,\varepsilon}(P)
 :=\sum_{i,k}a_i a_k
 \bigl(\ip{x_i}{x_k}-\ip{m_i(P)}{m_k(P)}\bigr)^2
 +\varepsilon\sum_i a_i\norm{m_i(P)}^2.
\end{equation}
Algorithm~\ref{alg:implemented-ab} follows the implementation used for the
reported runs. It requires a strictly positive initial coupling so that the KL
step is well defined. Appendix~\ref{app:experimental-configurations} gives the
initializations and the detailed monitoring conventions used in the experiments.

\begin{algorithm}[t]
\caption{Implemented projected \(A\)--\(B\) solver with an inexact weak OT oracle}
\label{alg:implemented-ab}
\small
\begin{algorithmic}[1]
\Require Data \(X=(x_i)_{i=1}^n,a,Y=(y_j)_{j=1}^p,b\), ridge
  \(\varepsilon>0\), matrix initializations \(A_{\rm init},B_{\rm init}\), and
  a strictly positive \(P_{\rm init}\in\Pi(a,b)\)
\Require Outer and inner budgets \(T_{\rm out},T_{\rm in}\), Sinkhorn budget
  \(N_{\rm sk}\), and step sizes \(\eta_A,\eta_B,\gamma>0\)
\State Compute \(R_A=\sqrt{M_2(\mu)M_2(\nu)}\) and \(R_B=M_2(\nu)\)
\State \(A\gets A_{\rm init}\), \(B\gets\operatorname{Proj}_{\mathbb S_+}(B_{\rm init})\),
  \(P\gets P_{\rm init}\), and \(J_{\rm best}\gets+\infty\)
\For{\(t=0,\ldots,T_{\rm out}-1\)}
  \State \(\bar P\gets P\) and \(\bar f\gets f_{A,B}(P)\)
  \For{\(s=0,\ldots,T_{\rm in}-1\)}
    \State \(P\gets\Call{InexactKLStep}{P,A,B,\gamma,N_{\rm sk}}\)
    \State \(P\gets\Call{RepairMarginals}{P,a,b}\)
    \If{\(f_{A,B}(P)<\bar f\)}
      \State \(\bar P\gets P\) and \(\bar f\gets f_{A,B}(P)\)
    \EndIf
  \EndFor
  \State \(P\gets\bar P\); compute \(m(P),M_P,S_P\) from
    \eqref{eq:finite-coupling-moments}
  \If{\(J_{\rm w,\varepsilon}(P)<J_{\rm best}\)}
    \State \((J_{\rm best},P_{\rm best},A_{\rm best},B_{\rm best})
      \gets(J_{\rm w,\varepsilon}(P),P,A,B)\)
  \EndIf
  \State \(A\gets\operatorname{Proj}_{\norm{\cdot}_F\le R_A}
    [A-4\eta_A(A-M_P)]\)
  \State \(B\gets\operatorname{Proj}_{\mathcal B_2}
    [B+2\eta_B(S_P-B)]\)
\EndFor
\State \(P_{\rm best}\gets\Call{FinalBalance}{P_{\rm best},a,b}\)
\State \(P_{\rm best}\gets\Call{RepairMarginals}{P_{\rm best},a,b}\)
\State Recompute \(J_{\rm best}=J_{\rm w,\varepsilon}(P_{\rm best})\)
\Ensure \(P_{\rm best},A_{\rm best},B_{\rm best}\)
\end{algorithmic}
\end{algorithm}

The call \textsc{InexactKLStep} denotes the finite Sinkhorn approximation of the
normalized KL projection in \eqref{eq:finite-mirror-step}, and
\textsc{RepairMarginals} restores the prescribed marginals up to roundoff. The inner
rule retains the lowest value of \(f_{A,B}\), while the outer return rule retains
the lowest direct ridge primal value \eqref{eq:finite-ridge-primal}; final balancing
and repair act only on that returned state. Theorem~\ref{thm:convergence} analyzes the sequence of projected
outer iterates, not the best so far tuple returned by Algorithm~\ref{alg:implemented-ab}.
Its conclusions require the stated ridge, step-size, and oracle-error hypotheses
and do not by themselves establish an error schedule for the finite KL mirror
routine. Appendix~\ref{app:experimental-configurations} records the implementation
budgets, tolerances, repair details, and state retention conventions.

\paragraph{Implementation of the outer projections.}
Algorithm~\ref{alg:implemented-ab} uses exact Euclidean projections in the
Frobenius geometry. The projection of a matrix \(Z_A\) onto the ball of radius
\(R_A\) is radial:
\[
 \operatorname{Proj}_{\norm{\cdot}_F\le R_A}(Z_A)
 =
 \begin{cases}
  Z_A,&\norm{Z_A}_F\le R_A,\\[2pt]
  \displaystyle\frac{R_A}{\norm{Z_A}_F}Z_A,&\norm{Z_A}_F>R_A.
 \end{cases}
\]
For the \(B\)-update, let \(Z_B=B+2\eta_B(S_P-B)\), symmetrize it as
\(\overline Z_B=(Z_B+Z_B^\top)/2\), and compute
\[
 \overline Z_B=Q\operatorname{diag}(\lambda_1,\ldots,\lambda_{d_y})Q^\top.
\]
Set \(\lambda_k^+=\max\{\lambda_k,0\}\) and
\(r=(\sum_k(\lambda_k^+)^2)^{1/2}\). The implemented projection is
\[
 \operatorname{Proj}_{\mathcal B_2}(Z_B)
 =Q\operatorname{diag}(\widehat\lambda_1,\ldots,\widehat\lambda_{d_y})Q^\top,
 \qquad
 \widehat\lambda_k=
 \begin{cases}
  \lambda_k^+,&r\le R_B,\\[2pt]
  \displaystyle\frac{R_B}{r}\lambda_k^+,&r>R_B.
 \end{cases}
\]
Thus the eigendecomposition first sets every negative eigenvalue to zero. If
the resulting positive semidefinite matrix lies outside the Frobenius ball, all
of its nonnegative eigenvalues are then scaled by the same factor. This is the
Frobenius projection onto
\(\mathcal B_2=\mathbb S_+^{d_y}\cap\{B:\norm B_F\le R_B\}\).

\subsection{Projected contraction and inexact oracle control}
\label{subsec:convergence-inexact-oracle}

Under the strong ridge condition of Theorem~\ref{thm:swap}, the outer problem is
a convex--concave saddle problem. We adapt the inexact-oracle strategy of Rioux,
Goldfeld, and Kato \cite{RiouxGoldfeldKato2024} from ordinary entropic GW
envelopes to convex weak OT inner problems. The exact outer operator uses an exact
solution of each fixed-matrix weak OT block; Algorithm~\ref{alg:implemented-ab}
instead supplies an approximate weak OT oracle. The analysis below treats the
resulting additive operator error and then bounds it by inner objective
suboptimality. The contraction theorem uses the common step size
\(\eta_A=\eta_B=\eta\); unequal block step sizes require a corresponding
preconditioned product norm.

\paragraph{Notation for the outer analysis.}
Throughout this subsection, \(z=(A,B)\),
\(\mathcal Z=\mathcal A_2\times\mathcal B_2\), and the product space carries
the Frobenius inner product. The outer envelope is \(F_\varepsilon\) from
\eqref{eq:ridge-envelope-functionals}. For fixed \((A,B)\), the exact oracle
means the unique optimal barycentric vector \(m_{A,B}\) of
\eqref{eq:finite-weak-ot-oracle}; an optimal coupling itself need not be unique.

\begin{theorem}[Projected contraction with inexact inner oracles]\label{thm:convergence}
Let \(\mu=\sum_{i=1}^n a_i\delta_{x_i}\in\cP_2(\R^{d_x})\) and
\(\nu=\sum_{j=1}^p b_j\delta_{y_j}\in\cP_2(\R^{d_y})\) be probability
measures with positive masses. Set
\[
\begin{aligned}
 S_\mu&=\sum_i a_i x_ix_i^\top,
 &\lambda_X&=\lambda_{\max}(S_\mu),\\
 M_2(\mu)&=\sum_i a_i\norm{x_i}^2,
 &M_2(\nu)&=\sum_j b_j\norm{y_j}^2,
\end{aligned}
\]
and let \(\varepsilon>2\lambda_X\). Let
\(\mathcal Z=\mathcal A_2\times\mathcal B_2\), where the outer sets are defined
in \eqref{eq:ridge-outer-domains}, and let \(F_\varepsilon\) be defined in
\eqref{eq:ridge-envelope-functionals}. For each \((A,B)\in\mathcal Z\), let
\(m_{A,B}=(m_i(A,B))_{i=1}^n\) be the unique barycentric vector of an exact
minimizer of \eqref{eq:finite-weak-ot-oracle}, and set
\[
 M_{A,B}=\sum_i a_i x_i m_i(A,B)^\top,\qquad
 S_{A,B}=\sum_i a_i m_i(A,B)m_i(A,B)^\top.
\]
By uniqueness of the barycentric vector and
Theorem~\ref{thm:danskin-envelope}, \(F_\varepsilon\) is differentiable on the outer sets, and
its saddle operator is
\[
\begin{aligned}
 G(A,B)
 &:=(\nabla_A F_\varepsilon(A,B),-\nabla_B F_\varepsilon(A,B))\\
 &=\bigl(4(A-M_{A,B}),\,2(B-S_{A,B})\bigr).
\end{aligned}
\]
Equip \(\R^{d_x\times d_y}\times\mathbb S^{d_y}\) with the product Frobenius
inner product and norm
\[
 \ip{(A,B)}{(A',B')}
 :=\ip{A}{A'}_F+\ip{B}{B'}_F,
 \qquad
 \norm{(A,B)}^2:=\norm A_F^2+\norm B_F^2.
\]
The operator \(G\) is globally \(L\)-Lipschitz on \(\mathcal Z\), where
\[
 L=\norm{K}_{\rm op},\qquad
 K=\begin{pmatrix}4&0\\0&2\end{pmatrix}
 +\frac8\varepsilon
 \begin{pmatrix}\sqrt{\lambda_X}\\\sqrt{M_2(\nu)}\end{pmatrix}
 \begin{pmatrix}\sqrt{\lambda_X}&\sqrt{M_2(\nu)}\end{pmatrix}.
\]
In particular,
\(L\le4+8(\lambda_X+M_2(\nu))/\varepsilon\). Moreover, \(G\) is strongly
monotone with parameter
at least
\[
  \alpha=\min\left\{4-\frac{8\lambda_X}{\varepsilon},2\right\}.
\]
The variational inequality
\[
 \ip{G(z^\star)}{z-z^\star}\ge0
 \qquad\text{for every }z\in\mathcal Z
\]
has a unique solution \(z^\star\in\mathcal Z\). Given
\(z_k=(A_k,B_k)\in\mathcal Z\) and an approximate oracle
\(\widehat G(z_k)=G(z_k)+e_k\), define the common projected step
\[
  z_{k+1}=\operatorname{Proj}_{\mathcal Z}
  \bigl(z_k-\eta\widehat G(z_k)\bigr).
\]
Here the projection and all norms are taken in the product Frobenius geometry.
If
\(0<\eta<2\alpha/L^2\), then
\[
  \norm{z_{k+1}-z^\star}\le q\norm{z_k-z^\star}+\eta\norm{e_k},
  \quad q=\sqrt{1-2\eta\alpha+\eta^2L^2}<1.
\]
Consequently, writing \(d_k=\norm{z_k-z^\star}\):
\begin{enumerate}[label=(\roman*)]
\item if \(e_k=0\) for every \(k\), then \(d_k\le q^k d_0\);
\item if \(\norm{e_k}\to0\), then \(z_k\to z^\star\);
\item if \(\sup_k\norm{e_k}\le\bar e\), then
\[
 d_k\le q^k d_0+\eta\frac{1-q^k}{1-q}\,\bar e,
 \qquad
 \limsup_{k\to\infty}d_k\le\frac{\eta\bar e}{1-q}.
\]
\end{enumerate}
\end{theorem}

\begin{proof}
The strong ridge minimax theorem and compactness of $\mathcal Z$ give a
solution of the variational inequality; the strong monotonicity proved below makes it
unique.
For fixed \(B\), the inner objective is \(2\alpha_B\)-strongly convex in its
barycentric vector, where
\(\alpha_B=\lambda_{\min}(C_{B,\varepsilon})\ge\varepsilon\). The operator
\(Tm=M_m\) satisfies \(\norm{T}^2\le\lambda_X\). Add the two variational
inequalities for the minimizers at $(A_1,B)$ and $(A_2,B)$. Strong monotonicity
of the inner derivative and Cauchy--Schwarz give
\[
 2\alpha_B\norm{m_{A_1,B}-m_{A_2,B}}_{L^2}^2
 \le4\ip{A_1-A_2}{T(m_{A_1,B}-m_{A_2,B})}_F.
\]
Consequently,
\[
  \norm{m_{A_1,B}-m_{A_2,B}}_{L^2}
  \le\frac{2\sqrt{\lambda_X}}{\alpha_B}\norm{A_1-A_2}_F.
\]
Since $\nabla_AF_\varepsilon=4(A-Tm_{A,B})$, the same estimate yields directly
\[
 \ip{\nabla_AF_\varepsilon(A_1,B)-\nabla_AF_\varepsilon(A_2,B)}
 {A_1-A_2}_F
 \ge\left(4-\frac{8\lambda_X}{\alpha_B}\right)\norm{A_1-A_2}_F^2.
\]
The \(B\)-block is \(2\)-strongly
concave because it is an infimum of affine functions plus $-\norm{B}_F^2$.
Thus $F_\varepsilon(\cdot,B)$ is
$(4-8\lambda_X/\varepsilon)$-strongly convex and
$F_\varepsilon(A,\cdot)$ is $2$-strongly concave. Adding the four defining
inequalities expressing strong convexity and concavity at $(A_1,B_1)$ and
$(A_2,B_2)$ cancels
the mixed terms and yields
\[
 \ip{G(z_1)-G(z_2)}{z_1-z_2}
 \ge\alpha\norm{z_1-z_2}^2.
\]
Hence $G$ is $\alpha$-strongly monotone.

It remains to establish Lipschitzness. Adding the variational
inequalities for the two minimizers at arbitrary $(A_1,B_1)$ and $(A_2,B_2)$ gives
\[
 \norm{m_{A_1,B_1}-m_{A_2,B_2}}_{L^2}
 \le\frac2\varepsilon\left(
 \sqrt{\lambda_X}\norm{A_1-A_2}_F+
 \sqrt{M_2(\nu)}\norm{B_1-B_2}_F\right).
\]
Here we used $\norm{m_{A,B}}_{L^2}\le\sqrt{M_2(\nu)}$ from convex order.
Also,
\[
 \norm{S_{m_1}-S_{m_2}}_F
 \le2\sqrt{M_2(\nu)}\norm{m_1-m_2}_{L^2}.
\]
Applying these inequalities to
$\nabla_AF_\varepsilon=4(A-Tm)$ and
$-\nabla_BF_\varepsilon=2(B-S_m)$ bounds the vector of the two block gradient
norms by $K$ times the vector
$(\norm{A_1-A_2}_F,\norm{B_1-B_2}_F)$. Thus $L=\norm K_{\rm op}$ works, and
the simpler bound in Theorem~\ref{thm:convergence} follows from the diagonal
and rank-one terms.

If $\mathcal Z$ is the product of the
outer balls, the saddle point satisfies
$z^\star=\operatorname{Proj}_{\mathcal Z}(z^\star-\eta G(z^\star))$. Therefore,
by nonexpansiveness of projection,
\[
\begin{aligned}
 \norm{z_{k+1}-z^\star}
 &\le\norm{z_k-z^\star-\eta(G(z_k)-G(z^\star))-\eta e_k}\\
 &\le\sqrt{1-2\eta\alpha+\eta^2L^2}\norm{z_k-z^\star}
   +\eta\norm{e_k},
\end{aligned}
\]
Iterating this recurrence gives
\[
 d_k\le q^k d_0+
 \eta\sum_{j=0}^{k-1}q^{k-1-j}\norm{e_j}.
\]
This is \(q^kd_0\) for exact oracles and gives the geometric-series bound in
Theorem~\ref{thm:convergence} when the errors are uniformly bounded. If
\(\norm{e_j}\to0\), split the
sum at a fixed index: the contribution of its finite initial part vanishes as
\(k\to\infty\), while its tail is bounded by
\(\eta\sup_{j\ge N}\norm{e_j}/(1-q)\). Letting \(N\to\infty\) proves
\(d_k\to0\).
\end{proof}

The recurrence in Theorem~\ref{thm:convergence} is stated in terms of the outer
operator error, while inner accuracy can be expressed by objective suboptimality
relative to the exact fixed-matrix value. Strong convexity of the fixed-matrix
objective in the barycentric mean converts that suboptimality into the required
error bound. The approximate plan may be produced by the KL mirror
routine in Algorithm~\ref{alg:implemented-ab}, provided marginal repair has made
it feasible. The suboptimality below is always measured against the unregularized
fixed-matrix objective \(f_{A,B}\). If a different routine instead minimizes an
entropy-regularized objective, its optimization error and regularization bias must
first be combined into a bound for this unregularized suboptimality.

\begin{proposition}[From inner objective suboptimality to outer oracle error]
\label{prop:inner-gap-oracle-error}
Let \(\mu=\sum_{i=1}^n a_i\delta_{x_i}\in\cP_2(\R^{d_x})\) and
\(\nu=\sum_{j=1}^p b_j\delta_{y_j}\in\cP_2(\R^{d_y})\) be probability
measures with positive masses. Fix \(A\in\R^{d_x\times d_y}\),
\(B\in\mathbb S_+^{d_y}\), and \(\varepsilon>0\), and set
\(C_{B,\varepsilon}=2B+\varepsilon I\). Let \(P^\star\) minimize
\eqref{eq:finite-weak-ot-oracle}, where \(\Pi(a,b)\) and \(m_i(P)\) are defined
in \eqref{eq:finite-coupling-moments}, set
\(m=(m_i(P^\star))_{i=1}^n\), and let \(\widehat P\in\Pi(a,b)\) be an inner
feasible coupling with
\(\widehat m=(m_i(\widehat P))_{i=1}^n\) and objective suboptimality
\[
 \delta:=f_{A,B}(\widehat P)-f_{A,B}(P^\star)\ge0.
\]
Thus \(\delta\) is the true fixed-matrix objective suboptimality, not a KL
proximal objective or an entropy-regularized objective gap.
Set
\[
\begin{aligned}
 \alpha_B&:=\lambda_{\min}(C_{B,\varepsilon})\ge\varepsilon,\\
 \lambda_X&:=\lambda_{\max}\!\left(\sum_i a_i x_ix_i^\top\right),
 &M_2(\nu)&:=\sum_j b_j\norm{y_j}^2.
\end{aligned}
\]
For a barycentric vector \(r=(r_i)_i\), write
\(M_r=\sum_i a_i x_ir_i^\top\) and
\(S_r=\sum_i a_i r_ir_i^\top\). Define the exact and approximate saddle
operators at \((A,B)\) by
\[
\begin{aligned}
 G&=\bigl(4(A-M_m),\,2(B-S_m)\bigr),\\
 \widehat G&=\bigl(4(A-M_{\widehat m}),\,2(B-S_{\widehat m})\bigr),
 &e&=\widehat G-G.
\end{aligned}
\]
Using the weighted norm
\(\norm r_{L^2(\mu)}^2=\sum_i a_i\norm{r_i}^2\) and the product Frobenius
norm
\[
 \norm e^2:=\norm{e_A}_F^2+\norm{e_B}_F^2,
 \qquad e=(e_A,e_B),
\]
one has
\[
  \norm{\widehat m-m}_{L^2(\mu)}\le\sqrt{\delta/\alpha_B},
  \qquad
  \norm{e}\le4\sqrt{\frac{(\lambda_X+M_2(\nu))\delta}{\alpha_B}}.
\]
\end{proposition}

\begin{proof}
The inner objective has strong convexity modulus $2\alpha_B$, so its
suboptimality is at
least $\alpha_B\norm{\widehat m-m}_{L^2}^2$. For the $A$ component,
\[
 \norm{e_A}_F=4\norm{M_{\widehat m}-M_m}_F
 \le4\sqrt{\lambda_X}\norm{\widehat m-m}_{L^2}.
\]
Both barycentric maps are feasible, hence their $L^2$ norms are at most
$\sqrt{M_2(\nu)}$. Therefore
\[
 \norm{S_{\widehat m}-S_m}_F
 \le\int\norm{\widehat m-m}(\norm{\widehat m}+\norm m)d\mu
 \le2\sqrt{M_2(\nu)}\norm{\widehat m-m}_{L^2},
\]
and $\norm{e_B}_F=2\norm{S_{\widehat m}-S_m}_F$. Combining the two components
in product Frobenius norm proves the result.
\end{proof}

Because \(\alpha_B\ge\varepsilon\), a sequence of inner suboptimalities
\(\delta_k\to0\) yields \(\norm{e_k}\to0\), and
Theorem~\ref{thm:convergence}(ii) then gives convergence of the projected outer
states. A uniform bound \(\delta_k\le\bar\delta\) instead gives the explicit
error floor in Theorem~\ref{thm:convergence}(iii) with
\(\bar e=4\sqrt{(\lambda_X+M_2(\nu))\bar\delta/\varepsilon}\).

\section{Experiments}\label{sec:experiments}

We present two groups of experiments. The synthetic studies use point clouds and
graphs endowed with low dimensional node features to isolate the algebraic zero
cost mechanism for martingale refinements of a Gram preserving skeleton. At the
constructed parent to child coupling, the barycentric wIGW cost vanishes, whereas
ordinary IGW at the same coupling is generally positive. The real data study uses
paired RNA and ATAC measurements to evaluate atlas based cell type transfer through
cell to cell and prototype to cell alignment.

\subsection{Synthetic objectives and reported quantities}
\label{subsec:reported-quantities}

For $P\in\Pi(a,b)$, define $m_i(P)=a_i^{-1}\sum_jP_{ij}y_j$. Every reported value
in the synthetic studies is recomputed directly from one of the following formulas:
\[
\begin{aligned}
 J_{\rm w}(P)
 &=\sum_{i,k}a_i a_k
   (\ip{x_i}{x_k}-\ip{m_i(P)}{m_k(P)})^2,\\
 J_{\rm IGW}(P)
 &=\sum_{i,k,j,\ell}
   (\ip{x_i}{x_k}-\ip{y_j}{y_\ell})^2P_{ij}P_{k\ell},\\
 J_{\rm w,\varepsilon}(P)
 &=J_{\rm w}(P)+\varepsilon\sum_i a_i\norm{m_i(P)}^2.
\end{aligned}
\]
In each synthetic construction, the generating parent of every target point is
known. Write $P_{\rm cert}$ for the coupling that places each target mass in
the row of its generating parent. Its name reflects the fact that its
conditional means recover the clean skeleton and therefore certify a zero weak
objective. The weak numerical solver does not receive this coupling.

Let \(z_i\) denote the known clean skeleton, let
\((\widehat A,\widehat B,\widehat P)\) be the state returned by the finite
\(A\)--\(B\) run, with \(\widehat P\) balanced, and set
\(\widehat m_i=m_i(\widehat P)\). The constructed coupling \(P_{\rm cert}\)
is an algebraically feasible zero wIGW witness; its ordinary value is only a
same plan diagnostic. Values at \(\widehat P\), \(P_{\rm POT}\), and
\(P_{\rm env}\) are finite iteration local values, while compatibility and
marginal residuals are numerical checks. Table~\ref{tab:synthetic-reported-quantities}
in Appendix~\ref{app:experimental-configurations} gives the precise definition
and evaluated plan or state for every reported synthetic quantity.

We center and scale each source point cloud once, before applying its rotation
and refinement; the target is then generated without independent centering or rescaling.
Every weak run starts from the product coupling rather than $P_{\rm cert}$ and
uses one deterministic initialization without restarts. The common solver
settings, budgets for each run, data seeds, and ordinary solver parameters are listed in
Tables~\ref{tab:common-ab-config}--\ref{tab:synthetic-ordinary-config} of
Appendix~\ref{app:experimental-configurations}. The choice
$\varepsilon=10^{-4}$ approximates the unregularized objective
and lies outside the strong ridge regime certified by
Theorem~\ref{thm:convergence}. Accordingly, these runs report empirical
behavior at the stated finite budgets.

\subsection{Shape refinements with martingale noise}\label{subsec:shape-refinement}

We sample dark foreground pixels without replacement from the IBM/USD cat image
\cite{IBMUSD}. The source is $X=(x_i)_{i=1}^n$ with uniform masses, and the clean
target skeleton is $z_i=R_{35^\circ}x_i$. Because $R_{35^\circ}$ is orthogonal,
$\ip{x_i}{x_k}=\ip{z_i}{z_k}$ algebraically;
floating point evaluations agree up to roundoff. The main point clouds are
shown in Figure~\ref{fig:cat-clouds} in the introduction. Two refinements are used.

For the symmetric refinement, draw a unit direction $v_i$ and set
\[
y_{i,+}=z_i+s_i v_i,\qquad y_{i,-}=z_i-s_i v_i,\qquad
 b_{i,+}=b_{i,-}=\frac{1}{2n}.
\]
The certificate has $P_{i,(i,+)}=P_{i,(i,-)}=1/(2n)$, hence
$m_i(P_{\rm cert})=z_i$ exactly. The main instance uses $n=500$ and $s_i=0.18$.
The homoscedastic sweep uses seven common radii in $[0,0.28]$. For the
heteroscedastic sweep, let
\[
 s_i(t)=t\left[0.04+0.20
 \frac{z_{i,2}-\min_k z_{k,2}}{\max_kz_{k,2}-\min_kz_{k,2}}\right],
 \qquad t\in\{0,0.2,\ldots,1.2\}.
\]
Thus points higher in the image receive more conditional spread while preserving
every parent mean.

For the Gaussian refinement, $n=250$ and each parent has $q=8$ offsets
$g_{ir}\sim N(0,I_2)$. We replace $g_{ir}$ by
$g_{ir}-q^{-1}\sum_sg_{is}$ and set $y_{ir}=z_i+0.10g_{ir}$ with mass $1/(nq)$.
This empirical recentering enforces zero sample mean for the offsets of every parent
in each realized dataset, so the resulting finite certificate is martingale.

\begin{figure}[t]
\centering
\includegraphics[width=.95\linewidth]{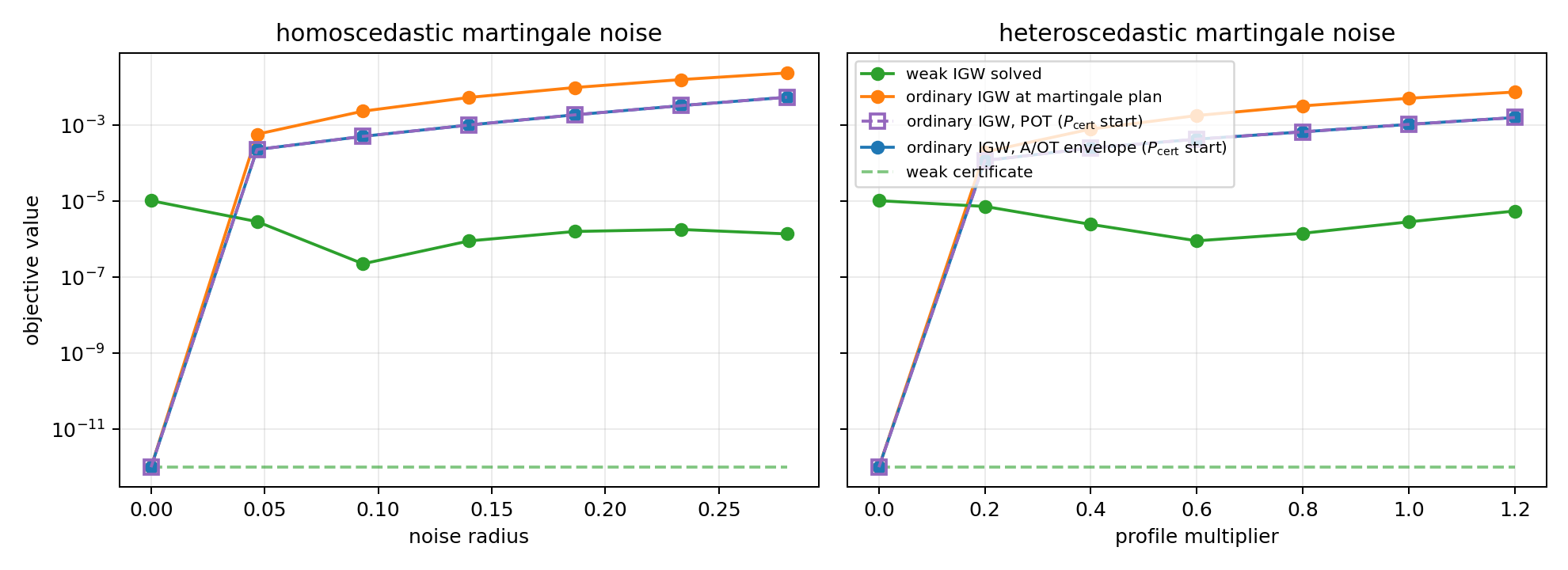}
\caption{Homoscedastic and heteroscedastic martingale noise sweeps. The dashed
certificate value remains at floating point zero. Values returned by the finite
\(A\)--\(B\) solver remain small but nonzero, while ordinary IGW evaluated at the
same refinement plan grows with target spread. The POT and $A$/OT curves are
finite ordinary solves initialized from the same $P_{\rm cert}$.}
\label{fig:cat-sweep}
\end{figure}

\begin{figure}[t]
\centering
\includegraphics[width=.9\linewidth]{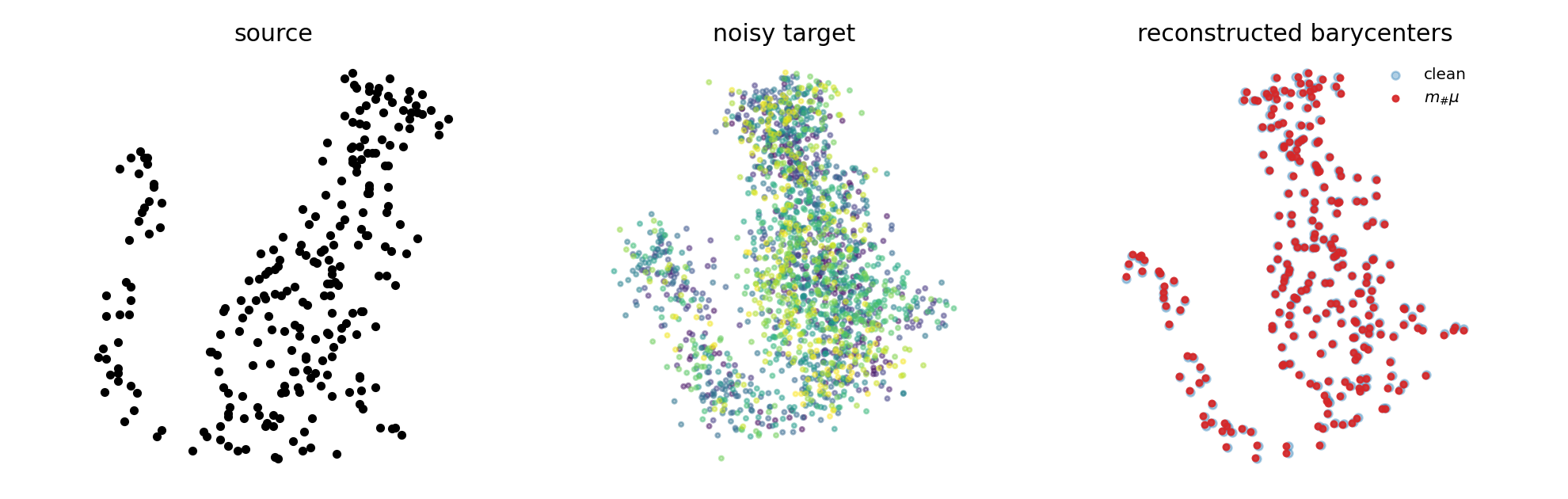}
\caption{Centered Gaussian refinement. Left: source. Middle: full noisy target.
Right: barycenters reconstructed from the numerical wIGW coupling, compared with the
clean rotated skeleton.}
\label{fig:cat-reconstruction}
\end{figure}

\begin{table}[tbp]
\centering
\small
\begin{tabular}{@{}lrr@{}}
\toprule
quantity & symmetric pair, $n=500$ & centered Gaussian, $n=250$\\
\midrule
clean rotation IGW & $0$ & $0$\\
weak certificate & $9.72\cdot10^{-34}$ & $2.78\cdot10^{-33}$\\
weak solved primal & $1.10\cdot10^{-6}$ & $1.93\cdot10^{-7}$\\
ridge solved value & $2.78\cdot10^{-5}$ & $3.91\cdot10^{-5}$\\
ordinary refinement plan & $9.17\cdot10^{-3}$ & $6.97\cdot10^{-3}$\\
ordinary POT & $1.79\cdot10^{-3}$ & $1.57\cdot10^{-3}$\\
ordinary $A$/OT envelope & $1.79\cdot10^{-3}$ & $1.57\cdot10^{-3}$\\
certificate pushforward error & $1.11\cdot10^{-16}$ & $2.22\cdot10^{-16}$\\
solved wIGW pushforward error & $6.07\cdot10^{-2}$ & $9.93\cdot10^{-3}$\\
$A$-compatibility & $9.04\cdot10^{-5}$ & $1.04\cdot10^{-4}$\\
$B$-compatibility & $1.40\cdot10^{-4}$ & $1.69\cdot10^{-4}$\\
maximum marginal residual & $6.94\cdot10^{-18}$ & $2.60\cdot10^{-18}$\\
\bottomrule
\end{tabular}
\caption{Shape results. The weak solved values and compatibility diagnostics use the
same retained $(A,B,P)$ state; both marginals of $P$ satisfy the stated
tolerance. The POT and $A$/OT rows use the common initialization compatible
with $P_{\rm cert}$ defined in Section~\ref{subsec:reported-quantities}.}
\label{tab:shape-results}
\end{table}

Both shape envelope runs stopped after three iterations with zero final
Frank--Wolfe gap. Their maximum marginal residuals were zero and
$2.95\cdot10^{-17}$, respectively, and their direct objectives agreed with POT
within $8.33\cdot10^{-17}$.

Across both sweeps the explicit weak certificate remains at arithmetic zero while
ordinary IGW at the same coupling increases with conditional spread.

\Needspace{5\baselineskip}
Under the fixed local optimization budget, the solved weak values remain small and vary
nonmonotonically with the noise parameter. POT and the $A$/OT envelope, both
initialized from $P_{\rm cert}$, agree to the reported precision. Within each sweep, one seeded direction
field is held fixed and only its radius is scaled across levels. These deterministic
curves test the numerical mechanism for the stated seed and budget.

\subsection{Graph feature refinement}\label{subsec:graph-refinement}

The coarse source has $n=10$ uniform nodes at
\[
 x_i=(\cos(2\pi i/n),\sin(2\pi i/n)),\qquad i=0,\ldots,n-1.
\]
Its undirected edge set contains ring edges $\{i,i+1\}$ and skip edges
$\{i,i+2\}$, with indices modulo $n$. We rotate the node features by $0.55$ radians
to obtain clean centers $z_i=R_{0.55}x_i$. In this instance each center has
$q=10$ children
\[
 y_{ir}=z_i+\rho_i(\epsilon_{ir}-\bar\epsilon_i),\qquad
\epsilon_{ir}\stackrel{\rm iid}{\sim}N(0,I_2),\qquad
 \rho_i=0.22[0.65+0.35\sin^2(1.7i)].
\]
All fine nodes have mass $1/(nq)$ and
$P_{i,(i,r)}^{\rm cert}=1/(nq)$, so $m_i(P_{\rm cert})=z_i$.

For visualization, the $q$ children of each parent are connected in a cycle, and three
deterministic child pairs are connected across the groups associated with each coarse
edge $\{i,j\}$. These edges serve as visualization metadata. The reported IGW
objective is computed from the two-dimensional latent node features and their Gram
matrices, making this an experiment on graph features at coarse and fine scales. The noise sweep uses
$q=8$, common multipliers in $[0,0.42]$, and one seeded graph per level.

\begin{figure}[H]
\centering
\includegraphics[width=.95\linewidth]{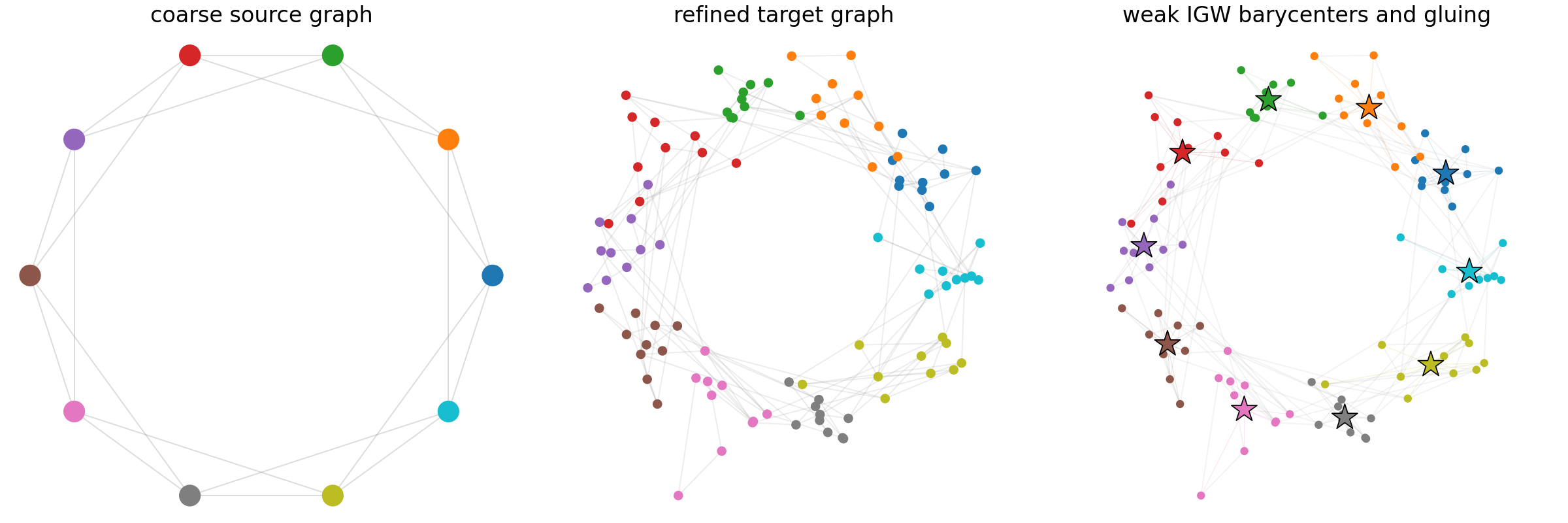}
\caption{Graph feature refinement. Left: the coarse source graph. Center: the refined
target graph. Right: certificate barycenters joined to their children by martingale
gluing rays. Shared marker colors identify children with the same parent coarse node.}
\label{fig:graph-refinement}
\end{figure}

\begin{center}
\begin{minipage}[t]{.47\linewidth}
\vspace{0pt}
\footnotesize
For one instance we obtained:
\[
\begin{array}{@{}lr@{}}
\toprule
\text{quantity}&\text{value}\\
\midrule
\text{weak certificate}&2.51\cdot10^{-32}\\
\text{weak solved primal}&9.97\cdot10^{-6}\\
\text{ridge solved value}&1.10\cdot10^{-4}\\
\text{ordinary refinement plan}&6.22\cdot10^{-2}\\
\text{ordinary POT}&5.48\cdot10^{-2}\\
\text{ordinary }A\text{/OT envelope}&5.48\cdot10^{-2}\\
\text{certificate pushforward error}&2.22\cdot10^{-16}\\
\text{pushforward RMS relative to certificate}&4.21\cdot10^{-1}\\
\text{\(A\)-compatibility}&2.76\cdot10^{-4}\\
\text{\(B\)-compatibility}&8.31\cdot10^{-5}\\
\text{row marginal residual}&1.39\cdot10^{-17}\\
\text{column marginal residual}&1.73\cdot10^{-18}\\
\bottomrule
\end{array}
\]
\end{minipage}\hfill
\begin{minipage}[t]{.49\linewidth}
\vspace{0pt}
\centering
\includegraphics[width=\linewidth]{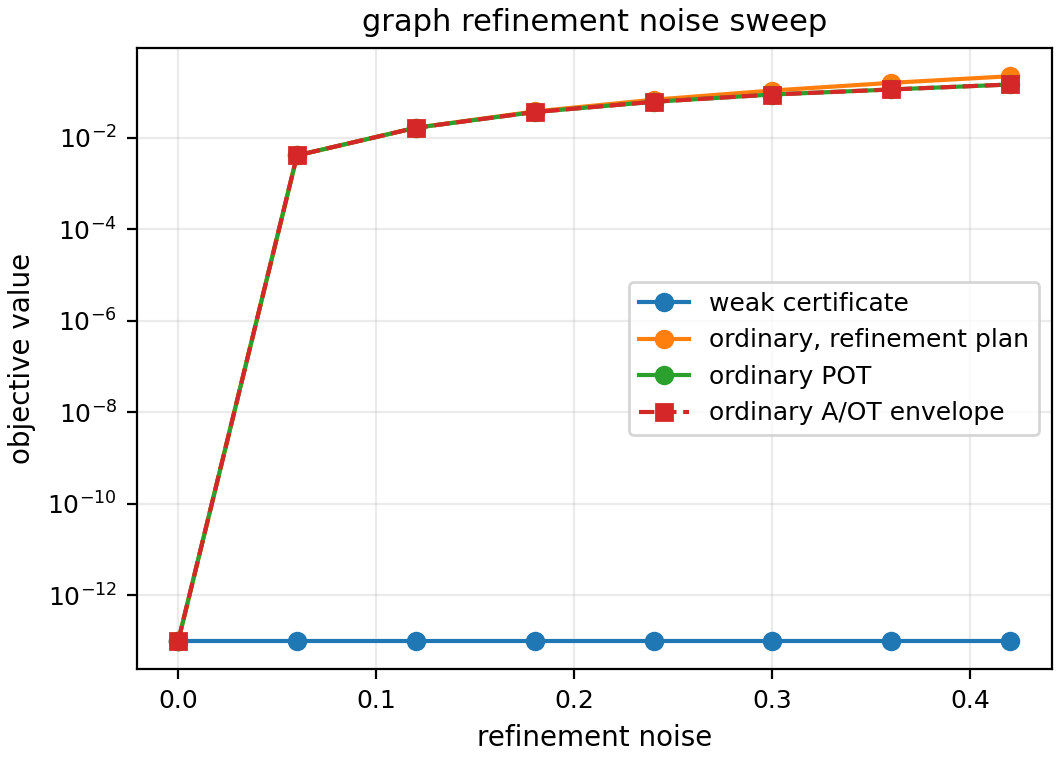}
\captionof{figure}{Graph feature refinement sweep. The explicit certificate value
remains at floating point zero; ordinary IGW evaluated at the same certificate
coupling increases with refinement spread. POT and the $A$/OT envelope start from
that coupling and agree to plotting precision.}
\label{fig:graph-sweep}
\end{minipage}
\end{center}

For this graph instance, the $A$/OT envelope stopped after two
iterations with final Frank--Wolfe gap zero and maximum marginal residual
$1.39\cdot10^{-17}$. Its direct objective agrees with POT to
$1.11\cdot10^{-15}$.

The stars and rays in Figure~\ref{fig:graph-refinement} depict the explicit
certificate; the numerical solver returns a separate coupling. Because inner-product
GW is orthogonally invariant, the finite problem can admit several nearly equivalent
barycentric skeletons. A small wIGW objective can therefore coexist with a large RMS
measured in the certificate's orientation and parent labeling. Both marginal residuals
of the reported graph plan are below the prescribed $10^{-12}$ feasibility tolerance.

\subsection{PBMC multiome atlas based cell type transfer}\label{subsec:pbmc-multiome}

\paragraph{Data and preprocessing.}
The experiments use paired gene expression (RNA) and chromatin accessibility (ATAC)
measurements from the 10x Genomics peripheral blood mononuclear cell (PBMC)
multiome dataset for one healthy donor
\cite{TenXPBMC}. The release contains \(12\,016\) cells. The quality control
criteria retain \(8\,212\) cells after requiring
\(500\)--\(6\,000\) detected genes, at least \(1\,000\) gene
expression counts, a mitochondrial fraction at most \(0.15\), at least \(3\,000\)
ATAC fragments, and a fraction of reads in peaks (FRIP) of at least \(0.15\).
The available barcode metadata do not include transcription start site (TSS)
enrichment, nucleosome signal, or blacklist fraction, so these criteria are not
evaluated. CellTypist provides the RNA
annotations \cite{DominguezCondeEtAl2022}; a fixed mapping gives six classes:
B cells, CD4 T cells, CD8 T cells, NK cells, CD14 monocytes, and FCGR3A
monocytes. Because RNA and ATAC are measured in the same physical cells, this
RNA-derived annotation is attached to the paired ATAC profile for evaluation; it
is not an independent ATAC annotation. The mapping leaves \(7\,760\) eligible
cells.

Each seed in \(\{17,23,31,47,59\}\) defines a balanced atlas of \(900\)
cells, with \(150\) cells per class, and a disjoint evaluation set of \(480\)
cells, with \(80\) cells per class. Every preprocessing transformation is fitted
on the atlas alone. For RNA, library normalization and selection of highly
variable genes precede principal component analysis (PCA); eight components are
retained. For ATAC,
TF--IDF and latent semantic indexing yield nine components; removing the first,
which is dominated by sequencing depth, leaves eight. Within each modality, the
atlas mean is subtracted and the coordinates are divided by the atlas root mean
square radius. The fitted transformations are applied unchanged to the evaluation
cells. The RNA and ATAC
coordinates both lie in \(\R^8\); their axes remain modality specific.

\paragraph{Alignment and transfer tasks.}
We evaluate atlas based cell type transfer under two alignment designs. The
cell to cell task aligns the RNA and ATAC
embeddings of the same \(480\)
evaluation cells:
\[
 X_{\rm cell}\in\R^{480\times8},
 \qquad Y_{\rm ATAC}\in\R^{480\times8}.
\]
The physical RNA/ATAC pairing remains hidden during optimization and is used
only for the secondary retrieval diagnostic reported in
Appendix~\ref{subsec:pbmc-evaluation-metrics}. The prototype to cell task aligns six
RNA class prototypes with the same \(480\) ATAC cells:
\[
 X_{\rm proto}\in\R^{6\times8},
 \qquad Y_{\rm ATAC}\in\R^{480\times8}.
\]
Each source prototype is the mean RNA embedding of one class in the disjoint
atlas. Thus one coarse source point corresponds to a heterogeneous population
of target cells. Barycentric wIGW compares the prototype Gram geometry with the
Gram geometry of the conditional ATAC means while retaining the full ATAC
marginal.

\paragraph{Baselines and protocol.}
We compare wIGW with POT's ordinary IGW solver applied to the Gram matrices
\(XX^\top\) and \(YY^\top\), using squared loss between Gram entries and the
product coupling as the initial plan \cite{POT}. The envelope
\eqref{eq:intro-ordinary-igw-envelope} gives a second implementation of the
same objective, which we call the IGW envelope baseline: each iteration solves
linear OT with cost \(-4XAY^\top\) and then updates \(A=X^\top P Y\). The scaled
rectangular identity avoids initializing \(A\) with the zero cross moment of the
centered product coupling in the prototype task. The table reports the product start
and scaled identity variants separately. The wIGW solver uses the same scaled
identity for \(A\) and fixes
\(\varepsilon=10^{-4}\), \(100\) outer iterations, and \(50\) inner mirror
iterations.

The modality specific coordinate axes make direct Euclidean OT depend on an
arbitrary identification of RNA and ATAC coordinates. The relational baselines
instead compare the within modality geometries. Their objectives are invariant to
independent orthogonal changes of coordinates, although a local solver with a fixed start
need not reach corresponding solutions; we report this distinction below.

Target \(k\)-means, with \(k=6\), uses only the ATAC embedding and serves as a
target only reference for its intrinsic class structure. After fitting, an optimal
one-to-one assignment (Hungarian matching) that maximizes total overlap maps the
six clusters to the six annotated classes for macro-F1; ARI and NMI are invariant
to this relabeling. SCOT provides an external single-cell multiome baseline
built from a graph in each modality \cite{DemetciEtAl2022}. It directly aligns the
\(480\) RNA cells with the \(480\) ATAC cells. In the prototype comparison, SCOT
aligns the \(900\) atlas RNA cells with the \(480\) ATAC cells, after which its
source rows are aggregated by atlas class. The table and figure label the
prototype SCOT result as a contextual comparison.

All transport methods use uniform empirical masses. Atlas labels define the six
prototypes, and source evaluation labels group transported mass by class after
optimization. Target labels determine the balanced evaluation splits and score
class transfer; physical pair identities are used only for the appendix
retrieval diagnostic.
Evaluation labels and pair identities do not enter the transport optimization.
Tables~\ref{tab:pbmc-data-config} and
\ref{tab:pbmc-solver-config} list the principal fixed settings; complete
configurations are stored with the run artifacts.

The main comparison combines four frozen sets of method runs on the same five
preprocessed splits. POT IGW and target \(k\)-means come from the original runs on
those splits; the IGW envelope initialized by a scaled identity and SCOT come from
their respective audits over the same splits; and wIGW uses the fixed
\(\varepsilon=10^{-4}\), \(100\times50\) protocol from the ridge sensitivity runs.
No score derived from target labels is used to select a restart or coupling.

\paragraph{Metrics.}
Following previous RNA/ATAC integration benchmarks, we report macro-F1 for
cell type transfer and ARI and NMI for agreement between the predicted and
annotated target partitions. uniPort reports F1, ARI, and NMI together on
paired scRNA/scATAC PBMC data \cite{CaoGongHongWan2022}, while sciCAN uses
macro-F1 for transfer in both directions between scRNA and scATAC
\cite{XuBegoliMcCord2022}.
For each target cell, we sum the transported mass over every source class and
select the class with largest mass. Since source and target use the same six
classes, this prediction requires no permutation matching. Macro-F1 measures
class transfer, while adjusted Rand index (ARI) and normalized mutual information
(NMI) compare the predicted and target partitions. Formal definitions are given
in Appendix~\ref{subsec:pbmc-evaluation-metrics}. That appendix also reports
top-1 and top-5 retrieval and mean reciprocal rank as secondary diagnostics of
exact cell matching, together with a reference based on random rankings. The main table
reports means and sample standard deviations over the five frozen splits.

\paragraph{Results.}
Table~\ref{tab:pbmc-main-results} and Figure~\ref{fig:pbmc-main-comparison}
compare class transfer across the five frozen splits. In the cell to cell task,
wIGW has mean macro-F1, ARI, and NMI values \(0.030\), \(0.082\), and \(0.052\)
higher, respectively, than the IGW envelope initialized with a scaled identity.
Its ARI and NMI are higher on all five splits.

The prototype to cell task most directly tests the proposed one-to-many construction.
Relative to the IGW envelope, wIGW has higher mean macro-F1, ARI, and NMI by
\(0.041\), \(0.115\), and \(0.060\), respectively; its ARI and NMI are higher on every
split. Target \(k\)-means has a slightly higher mean macro-F1, \(0.590\) versus
\(0.577\), while wIGW is higher by \(0.083\) in ARI and \(0.040\) in NMI. The target
\(k\)-means result shows that the ATAC embedding already contains substantial class
structure. By ARI and NMI, the partition transferred by wIGW agrees more closely
with the six classes.

\begin{table}[H]
\centering
\footnotesize
\renewcommand{\arraystretch}{1.08}
\begin{tabular}{@{}llccc@{}}
\toprule
task & method & macro-F1 & ARI & NMI\\
\midrule
cell to cell
 & POT IGW, product start
 & \(0.426\pm0.302\) & \(0.397\pm0.104\) & \(0.479\pm0.082\)\\
 & IGW envelope, scaled identity
 & \(0.653\pm0.091\) & \(0.502\pm0.029\) & \(0.568\pm0.023\)\\
 & barycentric wIGW, \(\varepsilon=10^{-4}\), \(100\times50\)
 & \(\mathbf{0.684\pm0.150}\) & \(\mathbf{0.584\pm0.046}\) &
   \(\mathbf{0.619\pm0.040}\)\\
 & target \(k\)-means\(^{\ast}\)
 & \(0.590\pm0.016\) & \(0.453\pm0.024\) & \(0.555\pm0.023\)\\
 & SCOT
 & \(0.515\pm0.204\) & \(0.526\pm0.089\) & \(0.576\pm0.071\)\\
\addlinespace[.25em]
prototype to cell
 & POT IGW, product start
 & \(0.248\pm0.079\) & \(0.413\pm0.015\) & \(0.524\pm0.015\)\\
 & IGW envelope, scaled identity
 & \(0.536\pm0.005\) & \(0.421\pm0.011\) & \(0.535\pm0.010\)\\
 & barycentric wIGW, \(\varepsilon=10^{-4}\), \(100\times50\)
 & \(0.577\pm0.144\) & \(\mathbf{0.536\pm0.014}\) &
   \(\mathbf{0.595\pm0.010}\)\\
 & target \(k\)-means\(^{\ast}\)
 & \(\mathbf{0.590\pm0.016}\) & \(0.453\pm0.024\) & \(0.555\pm0.023\)\\
 & SCOT\(^{\dagger}\)
 & \(0.470\pm0.211\) & \(0.516\pm0.076\) & \(0.565\pm0.058\)\\
\bottomrule
\end{tabular}
\caption{PBMC multiome atlas based cell type transfer. Entries are means and sample standard
deviations over five frozen subsampling splits from one donor. The method rows
come from the frozen sets of runs identified in the protocol paragraph; the table
compares alignment quality and does not match wall-clock budgets.
\(^{\ast}\)Target \(k\)-means uses only the ATAC embedding; an optimal
one-to-one assignment maps its clusters to class names before macro-F1 is
computed.
\(^{\dagger}\)SCOT aligns \(900\) atlas RNA cells with \(480\) ATAC cells in the
prototype to cell context, followed by source class aggregation.}
\label{tab:pbmc-main-results}
\end{table}

\begin{figure}[H]
\centering
\includegraphics[width=.98\linewidth]
 {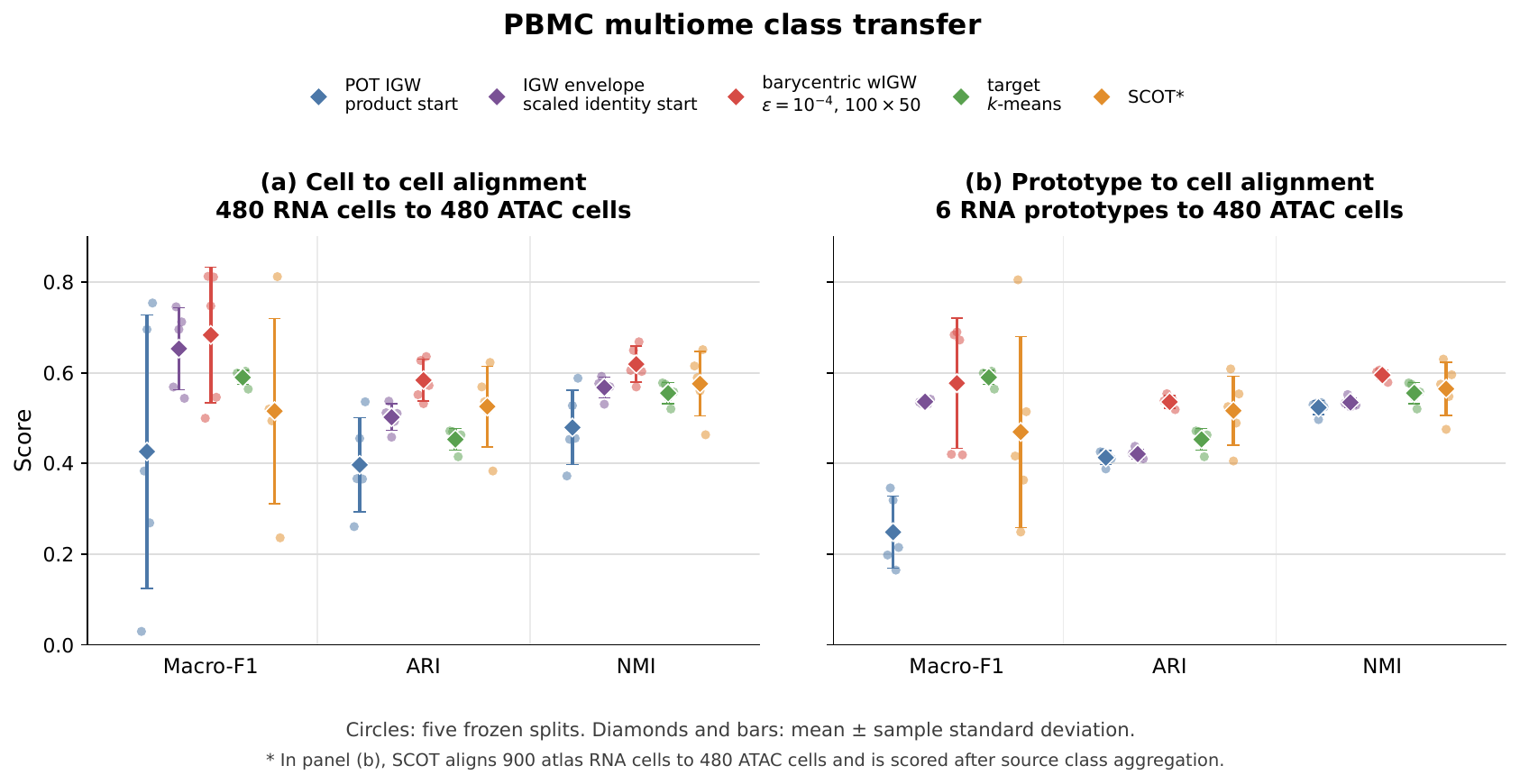}
\caption{PBMC cell type transfer by split. Circles show the five frozen subsamples;
diamonds show the means, and error bars show the sample standard deviation. In the
prototype to cell panel, SCOT aligns the \(900\) atlas RNA cells before class
aggregation, so it is a contextual comparison rather than a direct six-prototype
baseline. The wIGW label identifies the fixed
\(\varepsilon=10^{-4}\), \(100\times50\) run used in the table.}
\label{fig:pbmc-main-comparison}
\end{figure}

\paragraph{Secondary validation with a fixed protocol.}
After fixing \(\varepsilon=10^{-4}\) and the \(100\times50\) budget, we ran the
same protocol on five fresh subsampling seeds from the same donor. Relative to the
IGW envelope, the wIGW mean is higher by \(0.008\) in macro-F1, \(0.088\) in ARI,
and \(0.063\) in NMI for cell to cell transfer. For prototype to cell transfer,
the corresponding differences are \(-0.054\), \(0.132\), and \(0.081\).
ARI and NMI are higher for wIGW on all five new splits in both tasks, whereas
macro-F1 is variable. Table~\ref{tab:pbmc-postselection-validation} gives the full
secondary comparison. This is a post-selection check of the fixed protocol, not an
independent donor validation.

\paragraph{Ridge sensitivity.}
We repeat the wIGW runs for
\[
 \varepsilon\in
 \{10^{-6},10^{-5},10^{-4},10^{-3},10^{-2},10^{-1},1\},
\]
holding the splits, initialization, \(100\times50\) iteration budget, and
balancing rule fixed. We retain the existing value in the small ridge regime,
\(\varepsilon=10^{-4}\), as the fixed reference in this grid. The run and report
code receive that value explicitly and do not search the grid using metrics derived
from target labels. Across the five splits,
\(2\lambda_{\max}(S_\mu)\) ranges from \(0.771\) to \(0.814\) for cell sources
and from \(0.685\) to \(0.722\) for prototype sources. Thus only
\(\varepsilon=1\) among the tested values satisfies the strict strong ridge
condition of Theorem~\ref{thm:convergence} for every split.

Figure~\ref{fig:pbmc-epsilon-sensitivity} summarizes the sensitivity results.
The metrics are stable from \(10^{-6}\) through \(10^{-2}\). They decline at
\(10^{-1}\), and the alignment metrics at \(1\) differ substantially from the
small ridge regime. The spectral inequality certifies a sufficient
convex--concave regime for the outer problem. The empirical metrics separately assess alignment
quality. The full convergence theorem also imposes step size and inner oracle
conditions, which these fixed iteration runs do not certify.

\begin{figure}[H]
\centering
\includegraphics[width=.99\linewidth]
 {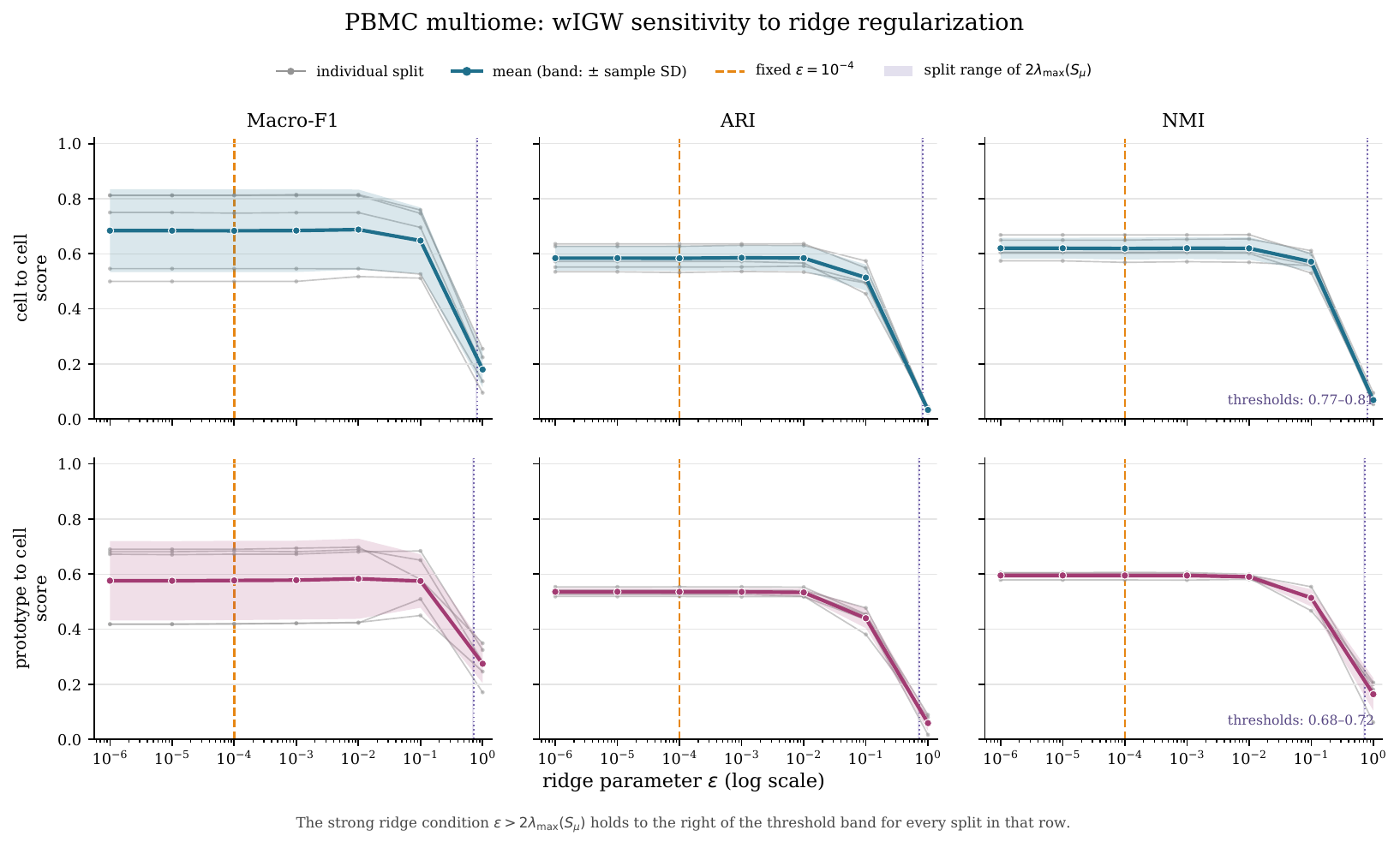}
\caption{Sensitivity to the ridge parameter. Rows correspond to the cell to cell
and prototype to cell tasks; columns report macro-F1, ARI, and NMI. Thin curves show
the five splits, thick curves show the mean, and shaded regions show the sample
standard deviation. The orange line marks \(\varepsilon=10^{-4}\); the purple
interval gives the split range of \(2\lambda_{\max}(S_\mu)\).}
\label{fig:pbmc-epsilon-sensitivity}
\end{figure}

Across all \(70\) runs, the largest coordinatewise marginal residual is
\(1.11\cdot10^{-16}\), the largest \(L^1\) change induced by final balancing and
repair is
\(2.24\cdot10^{-15}\), and the largest relative decrease in the running minimum
of the ridge objective over the final ten outer iterations is
\(8.30\cdot10^{-4}\).

\paragraph{Initialization and basis sensitivity.}
We performed a post hoc audit at the same reported iteration budget using the scaled
identity and two deterministic starts based on SVD frames, with selection by each
method's own objective. The scaled identity was selected on all five cell to cell splits
for both wIGW and the IGW envelope, on all five prototype to cell splits for the IGW
envelope, and on four of five prototype to cell splits for wIGW. The
prototype to cell wIGW results selected by objective across five splits have macro-F1 \(0.631\pm0.119\), ARI
\(0.535\pm0.014\), and NMI \(0.593\pm0.010\); the main table retains the original
results from the fixed start.

We also applied independent sign flips and orthogonal changes of basis to the two
modalities. Reusing a scaled identity in the transformed coordinates can reach
a different local solution; for example, the mean cell to cell wIGW macro-F1 changes
by \(-0.452\) after the sign flip and by \(-0.351\) after the orthogonal change
of basis.
The IGW envelope shows the same qualitative dependence. Covariant controls preserve
the plan exactly and change the objective by at most \(2.02\cdot10^{-16}\). Thus the
discrepancy has the expected invariance, while the finite local algorithms remain
initialization and basis sensitive. Appendix
Figure~\ref{fig:pbmc-basis-audit} gives the complete audit, and
Table~\ref{tab:pbmc-runtime} reports runtimes by method.

\paragraph{Limitations.}
The five splits come from one donor and measure subsampling variability rather
than uncertainty across donors. The frozen subsamples partly overlap, so their
sample standard deviations describe split variability rather than independent
replicate uncertainty. The annotation is derived from RNA and restricted to six
classes, and the source prototypes use atlas labels by construction. The benchmark is
class balanced, so its marginals do not represent natural PBMC abundance. All
SCOT runs emitted a Sinkhorn convergence warning with the fixed regularization
parameter;
their largest coordinatewise marginal residual was \(2.52\cdot10^{-5}\). The
comparison concerns alignment quality; the solver budgets differ by method and are
not matched by wall-clock time. The local solutions from the fixed start are also sensitive
to the coordinate representation, as the post hoc audit above shows. The transport
methods use local
optimization with finite iteration budgets, and the fixed
\(\varepsilon=10^{-4}\) lies below the strong ridge condition in the convergence
theorem. The barycentric model is appropriate when conditional mean geometry
retains the biological structure relevant to the alignment.

\FloatBarrier
\section{Conclusion}\label{sec:conclusion}

We introduced a framework based on conditional laws in which the disintegration
\(x\mapsto\pi_x\) represents the target side of a coupling. For the barycentric
inner-product relation, the conditional mean carries the compared geometry and a
martingale kernel supplies the variation needed to recover the prescribed target
marginal. The exact convex-order projection identifies this discrepancy with IGW over
laws below the target in convex order; under second moment assumptions, minimizers
exist and martingale gluing realizes an optimal coupling.

Moment duality yields a potential dual under compact support and, after positive ridge
regularization, an $A$--$B$ envelope whose weak OT objective is strongly convex
in its barycentric mean map. Under the spectral ridge, step-size, and oracle
hypotheses, the projected outer iteration has an explicit contraction bound under
oracle error. The point cloud and
graph feature studies isolate the refinement mechanism. In atlas based PBMC
prototype to cell transfer, wIGW exceeds the mean ARI and NMI of the IGW envelope
initialized by a scaled identity by \(0.115\) and \(0.060\), respectively; both
metrics are higher on all five splits.
Five additional splits from the same donor preserve the ARI and NMI advantage but not a
uniform macro-F1 advantage.
On the original five splits used for the ridge sweep, the metrics remain stable from \(10^{-6}\)
through \(10^{-2}\), an empirical range
below the threshold for the sufficient strong ridge condition used in the convergence theorem.

The construction is suited to settings in which conditional means carry the relevant
geometry and conditional variation represents target refinement. Convex order makes
the comparison directional: a mean preserving spread of the target cannot increase
the discrepancy. For fixed $(A,B)$, the ridge weak OT block has a unique barycentric
mean map, although its optimal coupling and martingale realization need not be unique.
Selecting among these realizations requires an additional criterion on conditional
variability. This limitation motivates richer relations between conditional laws,
including the maximal covariance relation \(D_{\rm MCov}\), the Wasserstein relation
\(D_{W_2}\), and the MMD relation \(D_k^{\rm MMD}\), which retain information beyond
the conditional mean. Their structural, dual, and computational properties remain to
be studied. The same coarse to fine viewpoint also suggests one-to-many comparisons
of quantum measurements and instruments, with their physical post-processing
constraints built into the formulation. Infinite-dimensional kernels and nonlinear
barycenters will require additional operator or barycenter control.

\Needspace{8\baselineskip}
\paragraph{Acknowledgments.}
The author is grateful to Alfred Galichon, Pierre Jacob, and Antoine Jacquet for
organizing the workshop \emph{Optimal Transport: Theory and Applications} at the
Institut d'\'Etudes Scientifiques de Carg\`ese in April 2024, and to Nathael Gozlan
and Beatrice Acciaio, whose talks there introduced the author to weak optimal
transport. The author also thanks Soumik Pal, Young-Heon Kim, Anna Korba, and
Brendan Pass for organizing the workshop \emph{Wasserstein Gradient Flows in Math
and Machine Learning} at the Banff International Research Station in June--July
2025. Discussions at these meetings helped motivate and shape the perspective
developed in this work.

OpenAI ChatGPT was used during manuscript preparation for editorial review, language
revision, and assistance with proof, literature, and code review. The author retains
responsibility for the manuscript's content and conclusions.

\let\proof\appendixproof
\let\endproof\endappendixproof
\appendix

\section{Analytic and variational foundations}
\label{app:detailed-proofs}\label{app:variational-proofs}

We establish here the functional analytic tools and variational results used for
ordinary IGW and barycentric weak OT. Intermediate estimates specify the topology and
integrability class used in each argument. Algebraic expansions that follow directly
from the stated identities are abbreviated.

\subsection{Weak compactness and the direct method}
\label{app:weak-compactness-direct-method}

We begin with the functional analytic facts used in the existence arguments. This
also fixes the precise meaning of ``weakly compact'' throughout the paper.
Our conventions for the weak topology follow
\cite[Sections~3.2--3.3 and Chapter~5]{BrezisFA}.

\begin{definition}[Weak convergence and weak compactness]
\label{def:app-weak-compactness}
Let $E$ be a real Banach space with continuous dual $E^*$. A sequence
$h_n\in E$ converges \emph{weakly} to $h\in E$, written
$h_n\rightharpoonup h$, if
\[
 \ell(h_n)\longrightarrow\ell(h)
 \qquad\text{for every }\ell\in E^*.
\]
A set $K\subset E$ is \emph{weakly compact} if it is compact for
$\sigma(E,E^*)$, the coarsest topology for which every $\ell\in E^*$ is
continuous. By the Eberlein--\v Smulian theorem
\cite[Chapter~3, Problem~10]{BrezisFA}, weak compactness is equivalent to the
following sequential property: every sequence in $K$ has a weakly convergent
subsequence whose limit belongs to $K$. Norm compactness is a strictly stronger
property.

When $E=H$ is a Hilbert space, the Riesz representation theorem identifies
$H^*$ with $H$, so weak convergence is tested by the inner products
$h\mapsto\ip{h}{g}_H$. In particular, for
$H=L^2(\mu;\R^{d_y})$ the tests are
$m_n\rightharpoonup m$ precisely when
\[
 \int g(x)^\top m_n(x)d\mu(x)
 \longrightarrow
 \int g(x)^\top m(x)d\mu(x)
 \qquad\text{for every }g\in L^2(\mu;\R^{d_y}).
\]
A functional $F:E\to(-\infty,+\infty]$ is \emph{weakly lower semicontinuous}
if every sublevel set $\{h:F(h)\le r\}$ is weakly closed.  In particular, if
$h_n\rightharpoonup h$, then $F(h)\le\liminf_nF(h_n)$.
\end{definition}

The compactness input in the following standard direct method statement is
Kakutani's reflexivity criterion \cite[Theorem~3.17]{BrezisFA}; compare the
minimization form in \cite[Corollaries~3.22--3.23]{BrezisFA}.

\begin{theorem}[Direct method in a reflexive space]
\label{thm:app-direct-method}
Let $E$ be a reflexive Banach space, let $K\subset E$ be nonempty, weakly closed,
and norm bounded, and let $F:K\to(-\infty,+\infty]$ be proper, bounded below,
and weakly lower semicontinuous.  Then $K$ is weakly compact and $F$ has a
minimizer on $K$.
\end{theorem}

\begin{proof}
Kakutani's criterion makes the closed unit ball weakly compact. Since $K$ is
norm bounded and weakly closed, it is therefore weakly compact. Choose a minimizing
sequence $(h_n)\subset K$.  Weak compactness gives a subsequence, not relabeled,
and $h^\star\in K$ such that $h_n\rightharpoonup h^\star$.  Weak lower
semicontinuity gives
\[
 F(h^\star)\le\liminf_nF(h_n)=\inf_{h\in K}F(h),
\]
so equality holds and $h^\star$ is a minimizer.
\end{proof}

\subsection{Ordinary IGW and barycentric weak OT}

\subsubsection{Proof of Theorem~\ref{thm:ordinary-igw-envelope}}\label{proof:2.1}

\begin{proof}
Fix $\pi\in\Pi(\mu,\nu)$ and let $(X,Y)$ and $(X',Y')$ be independent with
law $\pi$.  Independence of the two copies gives
\begin{align*}
 \E\ip{X}{X'}^2
 &=\sum_{r,s}\E[X_rX_s]\E[X'_rX'_s]=\norm{S_\mu}_F^2,\\
 \E\ip{Y}{Y'}^2&=\norm{S_\nu}_F^2,
\end{align*}
and
\[
 \E[\ip{X}{X'}\ip{Y}{Y'}]
 =\sum_{r,s}\E[X_rY_s]\E[X'_rY'_s]
 =\norm{M_\pi}_F^2.
\]
All terms are finite by Cauchy--Schwarz and the second moment assumptions.
Expanding the square therefore yields
\begin{equation}\label{eq:app-igw-moment}
 J_{\rm IGW}(\pi)
 =\norm{S_\mu}_F^2+\norm{S_\nu}_F^2-2\norm{M_\pi}_F^2.
\end{equation}

For any matrix $M$, completing the square gives
\[
 2\norm A_F^2-4\ip A M_F
 =2\norm{A-M}_F^2-2\norm M_F^2.
\]
Hence its minimum over $A$ is $-2\norm M_F^2$, with unique minimizer $A=M$.
Substitution into \eqref{eq:app-igw-moment} gives
\begin{align*}
 \inf_{\pi\in\Pi(\mu,\nu)}\inf_A
 \{2\norm A_F^2-4\ip A{M_\pi}_F\}
 &=\inf_A\left\{2\norm A_F^2+
   \inf_{\pi\in\Pi(\mu,\nu)}\int c_A\,d\pi\right\}.
\end{align*}
Because iterated infima commute,
$\inf_\pi\inf_A=\inf_A\inf_\pi$. This proves
\eqref{eq:ordinary-igw-envelope}.

For every coupling,
\[
 \norm{M_\pi}_F
 \le\int\norm x\norm y\,d\pi(x,y)
 \le\sqrt{M_2(\mu)M_2(\nu)}.
\]
Thus, for each fixed $\pi$, the Fenchel minimizer $A=M_\pi$ lies in
$\mathcal A_{\mu,\nu}$, so restricting $A$ to this set leaves the minimum
unchanged.

We next prove existence of a minimizer.  The set $\Pi(\mu,\nu)$ is tight because its two
marginals are fixed, hence relatively compact by Prokhorov's theorem
\cite[Theorem~5.1]{Billingsley}; it is weakly closed and therefore weakly compact.
Suppose
$\pi_n\Rightarrow\pi$.  For every coordinate product $x_ry_s$, choose a continuous
cutoff supported on $\{\norm x+\norm y\le2R\}$ and equal to one on the radius-$R$
set.  The cutoff integral converges weakly.  The discarded tail is uniform in $n$:
\begin{align*}
 \int\norm x\norm y\,\mathbf 1_{\{\norm x>R/2\}}d\pi_n
 &\le
 \left(\int\norm x^2\mathbf 1_{\{\norm x>R/2\}}d\mu\right)^{1/2}
 M_2(\nu)^{1/2},
\end{align*}
and the analogous target tail estimate also tends to zero.  Therefore
$M_{\pi_n}\to M_\pi$ entrywise and in Frobenius norm.  Equation
\eqref{eq:app-igw-moment} makes the IGW objective weakly continuous on the compact
coupling set, so there exists $\pi^\star\in\Pi(\mu,\nu)$ minimizing it.  The same
argument, together with compactness of $\mathcal A_{\mu,\nu}$, gives a jointly
minimizing pair $(A^\star,\pi^\star)$.

Finally, for fixed $A$,
\[
 |c_A(x,y)|\le4\norm A_F\norm x\norm y
 \le2\norm A_F(\norm x^2+\norm y^2).
\]
Thus $c_A$ is continuous and dominated in absolute value by a sum of integrable
marginal functions.  Kantorovich duality for such costs
\cite[Theorem~5.10]{VillaniOT} gives
\eqref{eq:ordinary-igw-kantorovich-dual} with pointwise Borel representatives.
If $u=-\psi$, feasibility of $(\varphi,\psi)$ is equivalent to
\[
 \varphi(x)\le u(y)-4y^\top A^\top x\quad\text{for every }y,
\]
and hence to $\varphi\le Q_Au$.  The infimum defining $Q_Au$ is universally
measurable for Borel $u$ by the analytic projection theorem
\cite[Proposition~7.47]{BertsekasShreve}, so it is measurable on the completed
source space.

For $(\varphi,\psi)\in\mathcal K_A$, choose $y_0$ such that $u(y_0)\in\R$.  Then
\[
 \varphi(x)\le Q_Au(x)\le u(y_0)-4y_0^\top A^\top x.
\]
The lower and upper bounds are integrable under $\mu$, and therefore
$Q_Au\in L^1(\mu)$.  Replacing $\varphi$ by $Q_Au$ preserves feasibility and cannot
decrease the dual objective.  Conversely, whenever $u\in L^1(\nu)$ and
$Q_Au\in L^1(\mu)$, the pair $(Q_Au,-u)$ belongs to $\mathcal K_A$.  These two
constructions prove \eqref{eq:ordinary-igw-Q-dual} with the potential class stated in
the theorem.
\end{proof}

\subsubsection{Proof of Theorem~\ref{thm:wot-blueprint}}\label{proof:2.4}

\begin{proof}
We first identify the feasible barycentric maps.  By the disintegration theorem
\cite[Theorem~3.4]{Kallenberg}, if
$\pi(dx,dy)=\mu(dx)\pi_x(dy)$ and $m_\pi(x)=\int y\,d\pi_x(y)$, then for every
integrable convex $u$, conditional Jensen gives
\[
 \int u(m_\pi(x))d\mu(x)
 \le\int\!\int u(y)d\pi_x(y)d\mu(x)=\int u(y)d\nu(y).
\]
Hence $(m_\pi)_\#\mu\cx\nu$.  Conversely, let $m_\#\mu\cx\nu$.
Theorem~\ref{thm:strassen} supplies a martingale kernel $K(z,dy)$ from
$m_\#\mu$ to $\nu$.  Define
\[
 \pi(dx,dy)=\mu(dx)K(m(x),dy).
\]
For bounded measurable $g$,
\[
 \int g(y)d\pi=\int\!\int g(y)K(z,dy)d(m_\#\mu)(z)=\int g(y)d\nu(y),
\]
and the first marginal is $\mu$ by construction.  The martingale identity gives
$\int yK(m(x),dy)=m(x)$ for $\mu$-almost every $x$.  Thus $m_\pi=m$, proving
the coupling/map equality.

The noncompact weak OT theorem
\cite[Theorems~2.9 and~3.1]{BBP} applies to the cost considered below.  It states
that if
$\nu\in\cP_2(\R^{d_y})$ and
$\widetilde C:\R^{d_x}\times\cP_2(\R^{d_y})\to\R\cup\{+\infty\}$ is jointly
lower semicontinuous for the Euclidean--$W_2$ product topology, bounded below, and
convex in its probability argument, then the weak primal admits a minimizer and
its value equals
\[
 \sup_{\psi\in\Phi_{b,2}}
 \left\{\int R_{\widetilde C}\psi\,d\mu-\int\psi\,d\nu\right\},
 \qquad
 R_{\widetilde C}\psi(x)
 :=\inf_{p\in\cP_2}\left\{\widetilde C(x,p)+\int\psi\,dp\right\},
\]
where $\Phi_{b,2}$ consists of continuous potentials bounded below by a constant
and above by a quadratic function.

For the present cost, set
\[
 C(x,p)=c(x,b(p)),\qquad b(p)=\int y\,dp(y).
\]
The barycenter map is continuous for $W_2$: if $p_n\to p$ in $W_2$, then an
optimal coupling and Cauchy--Schwarz give
$\norm{b(p_n)-b(p)}\le W_2(p_n,p)$.  Consequently $C$ is jointly continuous.
Because $b$ is affine and $c(x,\cdot)$ is convex, $p\mapsto C(x,p)$ is convex.
Let $C_0$ be the constant in the lower growth bound and put
$h(x)=C_0(1+\norm x^2)$.  Then
\[
 (C+h)(x,p)\ge\alpha\norm{b(p)}^2\ge0.
\]
Thus $C+h$ is jointly continuous, globally bounded below, and convex in $p$, while
$\nu\in\cP_2$ and $h\in L^1(\mu)$.  The cited weak OT theorem applies to
$C+h$.  Since $R_{C+h}\psi(x)=h(x)+R_C\psi(x)$, subtracting $\int h\,d\mu$ from
the primal and dual formulas yields existence of a primal minimizer and the probability-valued
dual formula for $C$.

It remains to reduce that dual to convex functions of the barycenter.  Given a
target potential $\psi$ in the weak dual, define
\[
 u(z)=\inf\left\{\int\psi(y)dp(y):p\in\cP_2(\R^{d_y}),\ b(p)=z\right\}.
\]
The competitor $p=\delta_z$ gives $u(z)\le\psi(z)$.  If $p_0,p_1$ have
barycenters $z_0,z_1$, then $(1-t)p_0+tp_1$ has barycenter
$(1-t)z_0+tz_1$; taking near minimizers proves convexity of $u$.  The growth
conditions in the weak dual make $u$ finite and lower bounded.  A finite convex
function on Euclidean space is continuous
\cite[Corollary~10.1.1]{Rockafellar}, and the Dirac competitor gives its
quadratic upper growth.  Partitioning the transform by $z=b(p)$ gives
\begin{align*}
 R_C\psi(x)
 &=\inf_{p}\left\{c(x,b(p))+\int\psi\,dp\right\}\\
 &=\inf_z\{c(x,z)+u(z)\}=Q_cu(x).
\end{align*}
For the replacement potential $u$, conditional Jensen and the Dirac competitor
give, for every $z$,
\[
 \inf_{p:\,b(p)=z}\int u(y)\,dp(y)=u(z).
\]
Consequently $R_Cu=Q_cu=R_C\psi$.  Since $u\le\psi$, the target term satisfies
$-\int u\,d\nu\ge-\int\psi\,d\nu$.  Replacing $\psi$ by $u$ therefore keeps the
transform fixed and can only increase the dual objective.  Thus convex barycentric
potentials yield the full dual value.

The dual may be indexed by all $u\in\mathcal U_2^{\rm cvx}$.  Every such $u$ has
an affine supporting lower bound
\cite[Corollary~12.1.2]{Rockafellar}.  Combining it with the coercive
quadratic lower bound on $c$ makes $Q_cu(x)>-\infty$, while evaluation at $z=0$
and the upper growth bounds give $|Q_cu(x)|\le C'(1+\norm x^2)$.  Moreover, for
every feasible $m$,
\[
 Q_cu(x)\le c(x,m(x))+u(m(x)),\qquad
 \int u(m)d\mu\le\int u\,d\nu.
\]
These inequalities give weak duality throughout $\mathcal U_2^{\rm cvx}$.  The
subclass constructed above already yields the primal value, so the supremum over
$\mathcal U_2^{\rm cvx}$ has the same value.

The feasible map set is convex: pointwise convex combinations remain below $\nu$
in convex order by Jensen.  If $c(x,\cdot)$ is $\lambda$-strongly convex and
$m_0,m_1$ are distinct minimizers, their midpoint is feasible and
\[
 \int c\left(x,\frac{m_0+m_1}{2}\right)d\mu
 \le\frac12\int c(x,m_0)d\mu+\frac12\int c(x,m_1)d\mu
 -\frac{\lambda}{8}\norm{m_0-m_1}_{L^2}^2,
\]
contradicting minimality.  Hence the optimal barycentric map is unique.
\end{proof}

\section{Barycentric formulations, moments, and projection}
\label{app:barycentric-proofs}

We prove here the coupling and map formulations, the moment representation, the zero
set, the convex order projection, and existence under finite second moments.

\subsection{Conditional laws and barycentric primal formulations}

\subsubsection{Proof of Proposition~\ref{prop:gw-specialization}}\label{proof:3.2}

\begin{proof}
By the disintegration theorem on standard Borel spaces
\cite[Theorem~3.4]{Kallenberg}, fix $\pi\in\Pi(\mu,\nu)$ and disintegrate
$\pi(dx,dy)=\mu(dx)\pi_x(dy)$.  The product law disintegrates as
\[
 (\pi\otimes\pi)(dx,dy,dx',dy')
 =\mu(dx)\mu(dx')\pi_x(dy)\pi_{x'}(dy').
\]
If the integrand is nonnegative, Tonelli's theorem applies; if it is signed and
integrable, the same rearrangement follows from Fubini's theorem.  Therefore
\begin{align*}
 &\iint\mathcal L(c_{\mathcal X}(x,x'),c_{\mathcal Y}(y,y'))
 d\pi(x,y)d\pi(x',y')\\
 &=\iint\left[\iint
 \mathcal L(c_{\mathcal X}(x,x'),c_{\mathcal Y}(y,y'))
 d\pi_x(y)d\pi_{x'}(y')\right]d\mu(x)d\mu(x').
\end{align*}
The bracket equals
$\mathfrak C_{\rm GW}(c_{\mathcal X}(x,x'),\pi_x,\pi_{x'})$.  The equality
holds for every coupling, so infimizing over $\pi$ proves the proposition.  Inner
products and squared loss give the stated IGW specialization.
\end{proof}

\subsubsection{Proof of Lemma~\ref{lem:feasible-set}}\label{proof:4.1}

\begin{proof}
Set $H=L^2(\mu;\R^{d_y})$.

\emph{Nonemptiness.}
Let $\bar y=\int y\,d\nu(y)$.  Jensen's inequality gives
$\delta_{\bar y}\cx\nu$, so the constant map $m\equiv\bar y$ belongs to
$\mathcal C_\nu$.

\emph{Convexity.}
If $m_0,m_1\in\mathcal C_\nu$ and $t\in[0,1]$, then every finite convex
function $u$ of at most linear growth satisfies
\[
 \int u((1-t)m_0+tm_1)\,d\mu
 \le(1-t)\int u(m_0)\,d\mu+t\int u(m_1)\,d\mu
 \le\int u\,d\nu.
\]
Hence $(1-t)m_0+tm_1\in\mathcal C_\nu$.

\emph{Norm boundedness.}
Although Definition~\ref{def:convex-order} uses convex tests of at most linear
growth, the quadratic test follows by monotone approximation.  For $R>0$, set
\[
 \phi_R(z):=
 \begin{cases}
  \norm z^2,&\norm z\le R,\\
  2R\norm z-R^2,&\norm z>R.
 \end{cases}
\]
Each $\phi_R$ is finite, convex, and of at most linear growth, and
$\phi_R(z)\uparrow\norm z^2$ as $R\uparrow\infty$.  Since
$m_\#\mu\cx\nu$, Definition~\ref{def:convex-order} and monotone convergence give
\[
 \norm{m}_{L^2(\mu)}^2
 =\int\norm{z}^2\,d(m_\#\mu)(z)
 \le\int\norm{y}^2\,d\nu(y)=M_2(\nu).
\]
Thus $\mathcal C_\nu$ is norm bounded.

\emph{Weak closedness.}
Suppose $m_n\in\mathcal C_\nu$ and $m_n\rightharpoonup m$ in $H$.  Fix a
finite continuous convex function $u$ of at most linear growth and define
\[
 I_u(q):=\int u(q(x))\,d\mu(x),\qquad q\in H.
\]
This functional is finite and convex, and it is lower semicontinuous in norm. To see
this, let $q_n\to q$ in $H$, take a subsequence realizing the lower limit,
and then extract a further subsequence converging pointwise $\mu$-almost
everywhere \cite[Theorem~4.9]{BrezisFA}.  Choose an affine function $\ell\le u$.
Since $u-\ell\ge0$,
continuity and Fatou's lemma give
\[
 \int (u-\ell)(q)\,d\mu
 \le\liminf_n\int (u-\ell)(q_n)\,d\mu.
\]
Strong $L^2$ convergence implies strong $L^1$ convergence, so
$\int\ell(q_n)\,d\mu\to\int\ell(q)\,d\mu$. Adding these relations proves that
$I_u$ is lower semicontinuous in norm. A convex functional with this property on a
Banach space is weakly lower semicontinuous
\cite[Corollary~3.9]{BrezisFA}.  Consequently,
\[
 \int u(m)\,d\mu
 \le\liminf_n\int u(m_n)\,d\mu
 \le\int u\,d\nu.
\]
Definition~\ref{def:convex-order} now gives $m_\#\mu\cx\nu$, so
$m\in\mathcal C_\nu$.  Thus $\mathcal C_\nu$ is weakly closed.

\emph{Weak compactness.}
The space $H$ is Hilbert and therefore reflexive.  The set $\mathcal C_\nu$ is
nonempty, weakly closed, and norm bounded, so
Theorem~\ref{thm:app-direct-method} implies that it is weakly compact.
\end{proof}

\subsubsection{Proof of Proposition~\ref{prop:map}}\label{proof:4.4}

\begin{proof}
Let $\pi\in\Pi(\mu,\nu)$.  Conditional Jensen gives, for every integrable convex
$u$,
\[
 \int u(m_\pi(x))d\mu(x)
 \le\int\!\int u(y)d\pi_x(y)d\mu(x)=\int u(y)d\nu(y),
\]
so $(m_\pi)_\#\mu\cx\nu$.  The coupling objective equals the map objective
evaluated at $m_\pi$.

Conversely, let $m\in\mathcal C_\nu$ and put $\eta=m_\#\mu$.
Theorem~\ref{thm:strassen} gives a measurable martingale kernel $K(z,dy)$ from
$\eta$ to $\nu$.
Set
\[
 \pi(dx,dy)=\mu(dx)K(m(x),dy).
\]
Its first marginal is $\mu$.  For bounded measurable $g$,
\[
 \int g(y)d\pi=\int\!\int g(y)K(z,dy)d\eta(z)=\int g(y)d\nu(y),
\]
so its second marginal is $\nu$.  Finally,
$\int yK(m(x),dy)=m(x)$ almost surely, and therefore $m_\pi=m$.  Every feasible
map is thus the conditional mean of a feasible coupling, and both infima coincide.
The moment estimate from Lemma~\ref{lem:feasible-set} ensures finiteness of all
terms.
\end{proof}

\subsection{Moments, zero structure, and projection}

\subsubsection{Proof of Proposition~\ref{prop:moment}}\label{proof:5.1}

\begin{proof}
Let $X,X'$ be independent with law $\mu$.  Expanding the square gives three
expectations.  Independence yields
\[
 \E\ip{X}{X'}^2=\norm{S_\mu}_F^2,\qquad
 \E\ip{m(X)}{m(X')}^2=\norm{S_m}_F^2,
\]
and, coordinate by coordinate,
\[
 \E[\ip{X}{X'}\ip{m(X)}{m(X')}]
 =\sum_{r,s}(\E[X_rm_s(X)])^2=\norm{M_m}_F^2.
\]
Substitution proves the moment identity.  Taking the infimum over the feasible map
set from Proposition~\ref{prop:map} gives the second formula.
\end{proof}

\subsubsection{Proof of Corollary~\ref{cor:zero}}\label{proof:5.2}

\begin{proof}
Theorem~\ref{thm:existence} gives an optimal map $m^\star\in\mathcal C_\nu$.
If the wIGW value is zero, then
\[
 \iint(\ip{x}{x'}-\ip{m^\star(x)}{m^\star(x')})^2d\mu(x)d\mu(x')=0.
\]
The integrand is measurable and nonnegative, hence it vanishes
$\mu\otimes\mu$-almost everywhere.  Conversely, any feasible map satisfying the
Gram identity has zero objective; nonnegativity then forces the infimum to be zero.
\end{proof}

\subsubsection{Proof of Corollary~\ref{cor:noisy-isometry}}\label{proof:5.3}

\begin{proof}
Take $m(x)=Tx$.  Since $T^\top T=I$,
$\ip{m(x)}{m(x')}=x^\top T^\top Tx'=\ip{x}{x'}$.  If $T_\#\mu\cx\nu$,
Corollary~\ref{cor:zero} therefore gives zero wIGW cost.

Under the noise model, conditional centering gives $\E[Y\mid X]=TX$.  For every
finite convex test $u$ of at most linear growth, conditional Jensen yields
\[
 \int u\,d(T_\#\mu)=\E u(TX)
 =\E u(\E[Y\mid X])\le\E u(Y)=\int u\,d\nu.
\]
Thus $T_\#\mu\cx\nu$, and the first part applies.  The argument uses only the
conditional centering condition $\E[\xi\mid X]=0$ on the residual law.
\end{proof}

\subsubsection{Proof of Theorem~\ref{thm:projection}}\label{proof:6.1}

\begin{proof}
Denote the projection value by
$P=\inf_{\eta\cx\nu}\IGW^2(\mu,\eta)$.  If $m\in\mathcal C_\nu$ and
$\eta=m_\#\mu$, then $(\mathrm{id},m)_\#\mu\in\Pi(\mu,\eta)$ and its IGW
cost equals the wIGW map cost $\mathcal J_\mu(m)$.  Consequently
\[
 P\le\IGW^2(\mu,m_\#\mu)\le \mathcal J_\mu(m).
\]
Infimizing over $m$ gives $P\le\wIGW^2(\mu,\nu)$.

For the reverse inequality, fix $\eta\cx\nu$ and
$\gamma(dx,dz)=\mu(dx)\gamma_x(dz)\in\Pi(\mu,\eta)$.  Define
$m_\gamma(x)=\int z\,d\gamma_x(z)$.  Conditional Jensen gives
$(m_\gamma)_\#\mu\cx\eta$, and transitivity gives
$(m_\gamma)_\#\mu\cx\nu$.

Let $(X,Z),(X',Z')$ be independent with law $\gamma$.  Conditional on
$(X,X')=(x,x')$, the variables $Z$ and $Z'$ have the product law
$\gamma_x\otimes\gamma_{x'}$.  Therefore Fubini and bilinearity give
\[
 \E[\ip{Z}{Z'}\mid X=x,X'=x']
 =\ip{m_\gamma(x)}{m_\gamma(x')}.
\]
For fixed $(X,X')$, the function
$b\mapsto(\ip{X}{X'}-b)^2$ is convex.  Conditional Jensen yields
\begin{align*}
 &(\ip{X}{X'}-\ip{m_\gamma(X)}{m_\gamma(X')})^2\\
 &\qquad\le
 \E[(\ip{X}{X'}-\ip{Z}{Z'})^2\mid X,X'].
\end{align*}
Second moments suffice for integrability: independence turns each squared inner
product into the Frobenius norm of a second moment matrix, and conditional Jensen
controls $m_\gamma$.  Integrating gives
\[
 \wIGW^2(\mu,\nu)\le \mathcal J_\mu(m_\gamma)\le J_{\rm IGW}(\gamma).
\]
Infimizing over $\gamma\in\Pi(\mu,\eta)$ and then over $\eta\cx\nu$ gives
$\wIGW^2(\mu,\nu)\le P$, proving equality.
\end{proof}

\subsubsection{Proofs of Corollaries~\ref{cor:target-refinement}
and~\ref{cor:igw-comparison}}
\label{proof:6.2}\label{proof:6.3}

\begin{proof}
If $\nu\cx\widetilde\nu$, transitivity gives
\[
 \{\eta:\eta\cx\nu\}\subseteq
 \{\eta:\eta\cx\widetilde\nu\}.
\]
The projection formula then implies
$\wIGW(\mu,\widetilde\nu)\le\wIGW(\mu,\nu)$.  For the ordinary IGW
comparison, use the admissible projection competitor $\eta=\nu$.
\end{proof}

\subsubsection{Proof of Theorem~\ref{thm:existence}}\label{proof:6.4}

\begin{proof}
We apply Theorem~\ref{thm:app-direct-method} with
\[
 H=L^2(\mu;\R^{d_y}),\qquad K=\mathcal C_\nu,
\]
and with the map objective from Proposition~\ref{prop:moment},
\[
 \mathcal J_\mu(m)
 =\norm{S_\mu}_F^2+\norm{S_m}_F^2-2\norm{M_m}_F^2.
\]

\emph{Feasible set and finiteness.}
Lemma~\ref{lem:feasible-set} verifies that $K$ is nonempty and weakly compact.
For every $m\in K$,
\[
 \norm{S_m}_F\le\Tr(S_m)=\norm{m}_{L^2}^2\le M_2(\nu),
 \qquad
 \norm{M_m}_F\le\sqrt{M_2(\mu)}\norm{m}_{L^2}.
\]
Thus $\mathcal J_\mu$ is finite on $K$.  It is bounded below by zero because
Proposition~\ref{prop:moment} identifies it with the integral of a square.

\emph{Weak continuity of the cross moment.}
Define
\[
 T:H\to\R^{d_x\times d_y},\qquad
 Tm=M_m=\int x m(x)^\top\,d\mu(x).
\]
For every $G\in\R^{d_x\times d_y}$, Cauchy--Schwarz gives
\[
 \left|\ip{G}{Tm}_F\right|
 =\left|\int\ip{G^\top x}{m(x)}\,d\mu(x)\right|
 \le\norm{G}_F\sqrt{M_2(\mu)}\norm{m}_{L^2(\mu)}.
\]
Thus $T$ is bounded and linear.  If $m_n\rightharpoonup m$ in $H$, every
coordinate of $Tm_n$ converges to the corresponding coordinate of $Tm$; because
the range is finite dimensional, weak and norm convergence coincide there
\cite[Proposition~3.6 and Theorem~3.10]{BrezisFA}, so
$Tm_n\to Tm$ in Frobenius norm.  Hence
$m\mapsto-2\norm{M_m}_F^2$ is weakly continuous.

\emph{Weak lower semicontinuity of the target moment.}
Since $S_m\succeq0$, the matrix Fenchel identity gives
\begin{equation}\label{eq:app-Sm-lsc}
 \norm{S_m}_F^2
 =\sup_{B\in\mathbb S_+^{d_y}}
 \left\{2\ip{B}{S_m}_F-\norm{B}_F^2\right\}
 =\sup_{B\in\mathbb S_+^{d_y}}
 \left\{2\int m(x)^\top Bm(x)\,d\mu(x)-\norm{B}_F^2\right\}.
\end{equation}
For fixed $B\succeq0$, the functional
$m\mapsto\int m^\top Bm\,d\mu=\norm{B^{1/2}m}_{L^2(\mu)}^2$ is weakly
lower semicontinuous \cite[Proposition~3.5(iii)]{BrezisFA}.  A pointwise supremum
of weakly lower semicontinuous functionals is weakly lower semicontinuous.
Consequently, \eqref{eq:app-Sm-lsc} shows that
$m\mapsto\norm{S_m}_F^2$ is weakly lower semicontinuous.  Adding the constant
$\norm{S_\mu}_F^2$ and the weakly continuous cross-moment term proves that
$\mathcal J_\mu$ is weakly lower semicontinuous on $K$.

\emph{Existence of a map minimizer.}
Theorem~\ref{thm:app-direct-method} therefore yields
$m^\star\in\mathcal C_\nu$ such that
$\mathcal J_\mu(m^\star)=\wIGW^2(\mu,\nu)$.

\emph{Projection minimizer and coupling realization.}
Let
\[
 P:=\inf_{\eta\cx\nu}\IGW^2(\mu,\eta),\qquad
 W:=\wIGW^2(\mu,\nu),\qquad
 \eta^\star:=(m^\star)_\#\mu.
\]
Theorem~\ref{thm:projection} gives $P=W$, while the deterministic coupling
$(\mathrm{id},m^\star)_\#\mu$ gives the explicit sandwich
\[
 P\le\IGW^2(\mu,\eta^\star)
 \le \mathcal J_\mu(m^\star)=W=P.
\]
Every inequality is an equality.  Hence $\eta^\star$ minimizes the projection
problem and $(\mathrm{id},m^\star)_\#\mu$ is an inner IGW optimizer.
Furthermore,
$M_2(\eta^\star)=\norm{m^\star}_{L^2}^2\le M_2(\nu)$, so
$\eta^\star\in\cP_2$.  Since $\eta^\star\cx\nu$,
Theorem~\ref{thm:strassen} gives a Borel
martingale kernel $K^\star$ from $\eta^\star$ to $\nu$.  The coupling
\[
 \pi^\star(dx,dy)=\mu(dx)K^\star(m^\star(x),dy)
\]
has conditional mean $m^\star(x)$, and hence has the same optimal wIGW objective.
Thus the optimal map $m^\star$ is realizable by an optimal coupling.
Since the objective depends only on this conditional mean, every martingale kernel
from $\eta^\star$ to $\nu$ yields an optimal wIGW coupling.  In particular, the wIGW
coupling primal admits a minimizer.

\emph{Converse reconstruction.}
Conversely, let $\eta^\star\cx\nu$ be a minimizer of the projection problem.  The quadratic
convex test gives $M_2(\eta^\star)\le M_2(\nu)$, so
$\eta^\star\in\cP_2$.  Theorem~\ref{thm:ordinary-igw-envelope} ensures that an
inner optimizer exists; let $\gamma^\star\in\Pi(\mu,\eta^\star)$ be any one.
The reverse half of the projection proof constructs
$m^\star(x)=\E_{\gamma^\star}[Z\mid X=x]$ and gives
\[
 W\le \mathcal J_\mu(m^\star)
 \le J_{\rm IGW}(\gamma^\star)=P=W.
\]
Thus $m^\star$ is a wIGW optimizer.  Any martingale kernel from
$(m^\star)_\#\mu$ to $\nu$ yields an optimal wIGW coupling by the same gluing
argument.
\end{proof}

\section{Duality and ridge formulations}
\label{app:duality-proofs}

We derive here the compact support and noncompact ridge duals and prove the curvature
condition that permits exchange of the two outer optimization variables.

\subsection{Compact and ridge duality}

\subsubsection{Proof of Theorem~\ref{thm:compact-dual}}\label{proof:7.1}

\begin{proof}
We first reduce all maps to the compact convex set $K=\operatorname{conv}(K_Y)$.
If $m_\#\mu\cx\nu$, then the convex function
$z\mapsto\operatorname{dist}(z,K)$ satisfies
\[
 0\le\int\operatorname{dist}(m(x),K)d\mu(x)
 \le\int\operatorname{dist}(y,K)d\nu(y)=0.
\]
Thus $m(x)\in K$ almost surely.  The set
\[
 \cM=\{m\in L^2(\mu;\R^{d_y}):m(x)\in K\ \mu\text{-a.e.}\}
\]
is convex and bounded by $R_Y$ in $L^2$.  It is norm closed: a strongly
convergent sequence has an almost everywhere convergent subsequence, whose limit
remains in the closed set $K$. A convex set closed in norm is weakly closed
\cite[Theorem~3.7]{BrezisFA}, so
reflexivity makes $\cM$ weakly compact.

For $m\in\cM$, convex order is encoded by
\begin{equation}\label{eq:app-cx-indicator}
 \sup_{u\in\cU_K}\left\{\int u(m)d\mu-\int u\,d\nu\right\}
 =\begin{cases}
 0,&m_\#\mu\cx\nu,\\
 +\infty,&\text{otherwise}.
 \end{cases}
\end{equation}
If convex order holds, Theorem~\ref{thm:strassen} supplies a martingale coupling
$(Z,Y)$ from $m_\#\mu$ to $\nu$. Both variables take values in $K$, so conditional
Jensen applies to every $u\in\cU_K$, even when $u$ is defined only on $K$:
\(\E u(Z)\le\E u(Y)\). Thus every term is nonpositive, and $u=0$ gives zero.
If convex order fails, restrict a globally defined separating convex test to $K$;
the restriction is a member $u_0\in\cU_K$ with a positive term. Replacing $u_0$
by $t u_0$ and letting $t\to\infty$ gives $+\infty$.

Insert \eqref{eq:app-cx-indicator} and the two matrix Fenchel identities into the
moment primal.  The variables $A$ and $m$ are both minimized, so their infima can
be grouped in either order.  For fixed $A$ we obtain
\[
 \inf_{m\in\cM}\sup_{(B,u)\in\mathcal B\times\cU_K}
 \mathcal L_A(m,B,u),
\]
where
\begin{align*}
 \mathcal L_A(m,B,u)
 =&\ 2\norm A_F^2-\norm B_F^2-\int u\,d\nu\\
 &+\int\{u(m(x))+2m(x)^\top Bm(x)-4m(x)^\top A^\top x\}d\mu(x).
\end{align*}

For Sion's theorem, give $\cM$ the weak $L^2$ topology and
$\mathcal B\times\cU_K$ the product of the finite-dimensional Frobenius topology
and the uniform topology of $C(K)$.  The first set is compact convex and the second
is convex.  For fixed $(B,u)$, the integrand is convex in $m(x)$.  If
$m_n\to m$ strongly in $L^2$, then $m_n\to m$ in probability; uniform continuity
and boundedness of $u$ on $K$ imply
$\int u(m_n)d\mu\to\int u(m)d\mu$.  The positive quadratic is norm continuous
and convex. Thus the whole $m$-section is convex and lower semicontinuous in norm,
hence weakly lower semicontinuous \cite[Corollary~3.9]{BrezisFA}. For fixed $m$,
the section is continuous and concave in $B$ and continuous linear in $u$. Sion's
theorem \cite{Sion}
therefore gives
\[
 \inf_{m\in\cM}\sup_{B,u}\mathcal L_A(m,B,u)
 =\sup_{B,u}\inf_{m\in\cM}\mathcal L_A(m,B,u).
\]

For fixed $(A,B,u)$, the function
\[
 (x,z)\mapsto u(z)+2z^\top Bz-4z^\top A^\top x
\]
is measurable in $x$ and continuous in $z\in K$.  The measurable minimum theorem
\cite[Theorem~18.19]{AliprantisBorder} provides a measurable minimizing selector,
so
\begin{align*}
 \inf_{m\in\cM}\int[\cdots]d\mu
 &=\int\min_{z\in K}
 \{u(z)+2z^\top Bz-4z^\top A^\top x\}d\mu(x)\\
 &=\int Q_{A,B}^Ku(x)d\mu(x).
\end{align*}
This yields the compact potential dual in Theorem~\ref{thm:compact-dual}.

Finally, the Fenchel optimizers are $A=M_m$ and $B=S_m$.  Since
$\norm x\le R_X$ and $\norm{m(x)}\le R_Y$,
\[
 \norm{M_m}_F\le R_XR_Y,\qquad
 0\preceq S_m\preceq R_Y^2I.
\]
Thus the Fenchel optimizers associated with each fixed $m$ belong to
$\mathcal A\times\mathcal B$.
\end{proof}

\subsubsection{Proof of Lemma~\ref{lem:unregularized-map-matrix}}
\label{proof:8.1}

\begin{proof}
Proposition~\ref{prop:moment} and
\eqref{eq:admissible-map-set} give
\[
 \wIGW^2(\mu,\nu)-\norm{S_\mu}_F^2
 =\inf_{m\in\mathcal C_\nu}
   \{\norm{S_m}_F^2-2\norm{M_m}_F^2\}.
\]
For every \(m\in\mathcal C_\nu\), Cauchy--Schwarz and
Lemma~\ref{lem:feasible-set} yield
\[
 \norm{M_m}_F\le\sqrt{M_2(\mu)M_2(\nu)},\qquad
 \norm{S_m}_F\le\Tr(S_m)=\norm m_{L^2(\mu)}^2\le M_2(\nu).
\]
Hence, for this fixed $m$, the Fenchel optimizers \(A=M_m\) and \(B=S_m\)
lie in \(\mathcal A_2\) and \(\mathcal B_2\). Restricting the Fenchel
identities to these domains does not change their values. For every fixed
\(m\in\mathcal C_\nu\),
\begin{align*}
 -2\norm{M_m}_F^2
 &=\inf_{A\in\mathcal A_2}
   \{2\norm A_F^2-4\ip{A}{M_m}_F\},\\
 \norm{S_m}_F^2
 &=\sup_{B\in\mathcal B_2}
   \{2\ip{B}{S_m}_F-\norm B_F^2\}.
\end{align*}
The variables \(A\) and \(B\) occur in separate terms, so the two identities
combine as
\begin{align*}
 \norm{S_m}_F^2-2\norm{M_m}_F^2
 &=\inf_{A\in\mathcal A_2}\sup_{B\in\mathcal B_2}
 \{2\norm A_F^2-\norm B_F^2
   -4\ip{A}{M_m}_F+2\ip{B}{S_m}_F\}\\
 &=\inf_{A\in\mathcal A_2}\sup_{B\in\mathcal B_2}
   \mathcal H_0(m,A,B).
\end{align*}
Indeed, the moment definitions give
\[
 -4\ip{A}{M_m}_F+2\ip{B}{S_m}_F
 =\int\!\left[-4m(x)^\top A^\top x+2m(x)^\top Bm(x)\right]d\mu(x)
 =\int c_{A,B}^0(x,m(x))d\mu(x).
\]
Taking the infimum over \(m\in\mathcal C_\nu\) proves the first equality in
\eqref{eq:unregularized-map-matrix-form}. Finally,
\[
 \inf_{m\in\mathcal C_\nu}\inf_{A\in\mathcal A_2}
 =\inf_{(m,A)\in\mathcal C_\nu\times\mathcal A_2}
 =\inf_{A\in\mathcal A_2}\inf_{m\in\mathcal C_\nu},
\]
which proves the second equality.
\end{proof}

\subsubsection{Proof of Proposition~\ref{prop:wot}}\label{proof:8.2}

\begin{proof}
Apply Theorem~\ref{thm:wot-blueprint} to
\[
 c_{A,B}^\varepsilon(x,z)
 =z^\top(2B+\varepsilon I)z-4z^\top A^\top x.
\]
The cost is finite and jointly continuous.  Its Hessian in $z$ is
\[
 \nabla_{zz}^2c_{A,B}^\varepsilon=4B+2\varepsilon I
 \succeq2\varepsilon I,
\]
so it is uniformly strongly convex.  Since $B\succeq0$, Young's inequality gives
\begin{align*}
 c_{A,B}^\varepsilon(x,z)
 &\ge\varepsilon\norm z^2-4\norm A_F\norm x\norm z\\
 &\ge\frac\varepsilon2\norm z^2
 -\frac{8\norm A_F^2}{\varepsilon}\norm x^2.
\end{align*}
For the upper bound,
\begin{align*}
 |c_{A,B}^\varepsilon(x,z)|
 &\le(2\norm B_F+\varepsilon)\norm z^2
 +4\norm A_F\norm x\norm z\\
 &\le C_{A,B,\varepsilon}(1+\norm x^2+\norm z^2)
\end{align*}
for a finite constant.  The assumptions
$\mu,\nu\in\cP_2$ provide the required marginal moments.

Theorem~\ref{thm:wot-blueprint} now gives existence of a primal minimizer, coupling/map equality,
and strong duality.  Its transform is
\[
 Q_{c_{A,B}^\varepsilon}u(x)
 =\inf_z\{u(z)+z^\top C_{B,\varepsilon}z-4z^\top A^\top x\}
 =Q_{A,B}^{(\varepsilon)}u(x),
\]
and its potential class is $\cU_2$.  The uniform Hessian lower bound and
convexity of $\mathcal C_\nu$ imply uniqueness: if two distinct minimizers existed,
their midpoint would decrease the objective by at least
$\varepsilon\norm{m_0-m_1}_{L^2}^2/4$.
The conditional mean of any minimizing coupling is feasible for the map problem
and has the same objective value.  It must therefore equal the unique map
$m_{A,B}$ in $L^2(\mu;\R^{d_y})$.

It remains to verify the asserted convexity in the coupling.  For
$\pi_0,\pi_1\in\Pi(\mu,\nu)$ and $t\in[0,1]$, disintegration with respect to the
common first marginal is affine, so
\[
 m_{(1-t)\pi_0+t\pi_1}=(1-t)m_{\pi_0}+tm_{\pi_1}
 \qquad\mu\text{-almost everywhere}.
\]
Since $z\mapsto c_{A,B}^\varepsilon(x,z)$ has Hessian at least
$2\varepsilon I$, integration of its convexity inequality proves convexity of
the coupling objective.  The same inequality is strong in the difference of
the two mean maps.  It is not generally strict in the couplings themselves,
because the affine map $\pi\mapsto m_\pi$ need not be injective.

For the parameter estimate, fix any feasible $m\in\mathcal C_\nu$.  Directly
comparing the two costs gives
\begin{align*}
 &\left|\int
   \bigl(c_{A,B}^\varepsilon-c_{A',B'}^\varepsilon\bigr)(x,m(x))
   d\mu(x)\right|\\
 &\quad\le
 2\norm{B-B'}_F\int\norm{m(x)}^2d\mu(x)
 +4\norm{A-A'}_F\int\norm{x}\norm{m(x)}d\mu(x)\\
 &\quad\le
 2M_2(\nu)\norm{B-B'}_F
 +4\sqrt{M_2(\mu)M_2(\nu)}\norm{A-A'}_F.
\end{align*}
The last line follows from Lemma~\ref{lem:feasible-set} and
Cauchy--Schwarz.  Evaluate the $(A,B)$ problem at a minimizer for $(A',B')$
to obtain one direction of \eqref{eq:ridge-oracle-lipschitz}; interchange the
two parameter pairs to obtain the reverse direction.
\end{proof}

\subsubsection{Proof of Theorem~\ref{thm:ridge-dual}}\label{proof:8.3}

\begin{proof}
The ridge primal and Proposition~\ref{prop:moment} give
\begin{equation}\label{eq:app-ridge-moment}
 \wIGWeps^2(\mu,\nu)-\norm{S_\mu}_F^2
 =\inf_{m\in\mathcal C_\nu}
 \{\norm{S_m}_F^2-2\norm{M_m}_F^2
 +\varepsilon\norm m_{L^2}^2\}.
\end{equation}
We derive the matrix form before making any minimax exchange.  Fix
$m\in\mathcal C_\nu$.  The bounds in the proof of
Lemma~\ref{lem:unregularized-map-matrix} place $M_m$ and $S_m$ in
$\mathcal A_2$ and $\mathcal B_2$.  Completing the two Frobenius squares gives
\begin{align}
 \inf_{A\in\mathcal A_2}
 \{2\norm A_F^2-4\ip A{M_m}_F\}
 &=-2\norm{M_m}_F^2,
 & A_{\rm opt}&=M_m,
 \label{eq:app-ridge-A-fenchel}\\
 \sup_{B\in\mathcal B_2}
 \{2\ip B{S_m}_F-\norm B_F^2\}
 &=\norm{S_m}_F^2,
 & B_{\rm opt}&=S_m.
 \label{eq:app-ridge-B-fenchel}
\end{align}
The $A$ and $B$ variables occur in separate terms, and the ridge term is
independent of both.  Hence
\begin{align*}
 &\norm{S_m}_F^2-2\norm{M_m}_F^2
   +\varepsilon\norm m_{L^2}^2\\
 &\quad=\inf_{A\in\mathcal A_2}\sup_{B\in\mathcal B_2}
 \bigl\{2\norm A_F^2-\norm B_F^2
 -4\ip A{M_m}_F+2\ip B{S_m}_F
 +\varepsilon\norm m_{L^2}^2\bigr\}\\
 &\quad=\inf_{A\in\mathcal A_2}\sup_{B\in\mathcal B_2}
 \mathcal H_\varepsilon(m,A,B),
\end{align*}
where the last equality is precisely the second line of
\eqref{eq:ridge-map-matrix-functional}.  Inserting this identity into
\eqref{eq:app-ridge-moment} yields
\[
 \wIGWeps^2-\norm{S_\mu}_F^2
 =\inf_{m\in\mathcal C_\nu}\inf_{A\in\mathcal A_2}
   \sup_{B\in\mathcal B_2}\mathcal H_\varepsilon(m,A,B).
\]
The two infima range over the product
$\mathcal C_\nu\times\mathcal A_2$ and therefore commute.  Thus
\[
 \wIGWeps^2-\norm{S_\mu}_F^2
 =\inf_{A\in\mathcal A_2}\inf_{m\in\mathcal C_\nu}
   \sup_{B\in\mathcal B_2}\mathcal H_\varepsilon(m,A,B).
\]

We next exchange $m$ and $B$ while keeping $A$ fixed.  Equip
$\mathcal C_\nu$ with the weak $L^2$ topology and
$\mathcal B_2$ with its Euclidean topology.  The minimizing set
$\mathcal C_\nu$ is compact and convex in the weak topology by
Lemma~\ref{lem:feasible-set}.  The maximizing set
$\mathcal B_2=\{B\in\mathbb S_+^{d_y}:\norm B_F\le M_2(\nu)\}$ is compact and
convex in finite dimension.  For fixed $B\in\mathcal B_2$, positivity of $B$
gives
$C_{B,\varepsilon}=2B+\varepsilon I_{d_y}\succeq\varepsilon I_{d_y}$.
Consequently, the $m$-section is convex and norm continuous, hence weakly lower
semicontinuous \cite[Corollary~3.9]{BrezisFA}.  For fixed $m$, all dependence on
$B$ is through
$2\ip{B}{S_m}_F-\norm B_F^2$, a continuous concave function on
$\mathcal B_2$.  Sion's theorem \cite{Sion} therefore exchanges the minimizing
variable $m$ and the maximizing variable $B$, for every fixed $A$:
\begin{align*}
 \inf_{m\in\mathcal C_\nu}\sup_{B\in\mathcal B_2}
 \mathcal H_\varepsilon(m,A,B)
 &=\sup_{B\in\mathcal B_2}\inf_{m\in\mathcal C_\nu}
 \mathcal H_\varepsilon(m,A,B).
\end{align*}
This is the only minimax exchange used so far; the outer $A$ and $B$ variables
have not been exchanged.

For fixed $(A,B)$, take the terms independent of $m$ outside the inner
infimum.  Proposition~\ref{prop:wot} then gives
\begin{align*}
 \inf_{m\in\mathcal C_\nu}\mathcal H_\varepsilon(m,A,B)
 &=2\norm A_F^2-\norm B_F^2
   +\min_{m\in\mathcal C_\nu}
     \int c_{A,B}^\varepsilon(x,m(x))d\mu(x)\\
 &=2\norm A_F^2-\norm B_F^2
   +\mathsf W_{A,B}^\varepsilon(\mu,\nu)
 =F_\varepsilon(A,B).
\end{align*}
Taking $\inf_A\sup_B$ proves the third line of
\eqref{eq:ridge-map-matrix-form}.  Finally, the potential formula in
Proposition~\ref{prop:wot} gives, for each fixed $(A,B)$,
\[
 F_\varepsilon(A,B)
 =\sup_{u\in\cU_2}\mathcal D_\varepsilon(A,B,u).
\]
The two suprema range over the product $\mathcal B_2\times\cU_2$, so they may
be combined.  This proves the final line of
\eqref{eq:ridge-map-matrix-form}.
\end{proof}

\subsubsection{Proof of Theorem~\ref{thm:swap}}\label{proof:8.4}

\begin{proof}
Proposition~\ref{prop:wot} and the definitions of
$\mathcal H_\varepsilon$ and $F_\varepsilon$ give, for every
$(A,B)\in\mathcal A_2\times\mathcal B_2$,
\begin{equation}\label{eq:app-F-partial-minimum}
 F_\varepsilon(A,B)
 =\inf_{m\in\mathcal C_\nu}\mathcal H_\varepsilon(m,A,B).
\end{equation}
Let $T:L^2(\mu;\R^{d_y})\to\R^{d_x\times d_y}$ be $Tm=M_m$.  Its adjoint
satisfies $T^*H(x)=H^\top x$, and
\[
 \norm{T^*H}_{L^2}^2
 =\int x^\top HH^\top x\,d\mu(x)
 \le\lambda_X\norm H_F^2.
\]
Thus $T^*T\preceq\lambda_XI$.  For fixed $B$, the part of the joint objective
depending on $(A,m)$ obeys the identity
\begin{align*}
 &2\norm A_F^2-4\ip A{Tm}_F+\varepsilon\norm m_{L^2}^2
 +2\int m^\top Bm\,d\mu\\
 &=2\norm{A-Tm}_F^2
 +\ip m{(\varepsilon I-2T^*T)m}_{L^2}
 +2\int m^\top Bm\,d\mu.
\end{align*}
When $\varepsilon\ge2\lambda_X$, every term on the right is jointly convex in
$(A,m)$.  Because $\pi\mapsto m_\pi$ is affine for couplings with first
marginal $\mu$, the corresponding objective is also jointly convex in
$(A,\pi)$.  For fixed $(A,m)$, equivalently fixed $(A,\pi)$, its dependence on
$B$ is concave.  Thus, under the spectral condition, $(A,m)$ or $(A,\pi)$ is
the minimizing block and $B$ is the maximizing block of a convex--concave
saddle problem.  Partial minimization over the fixed convex set
$\mathcal C_\nu$ therefore makes $A\mapsto F_\varepsilon(A,B)$ convex.

For fixed $A$ and feasible $m$, the $B$-dependence is
$-\norm B_F^2+2\ip B{S_m}_F$.  The infimum over $m$ of its affine part
is concave, and adding $-\norm B_F^2$ preserves concavity.  In fact this term makes
the result $2$-strongly concave.

The estimate \eqref{eq:ridge-oracle-lipschitz} makes
$(A,B)\mapsto\mathsf W_{A,B}^\varepsilon$ continuous.  Adding the two outer
quadratic terms proves continuity of $F_\varepsilon$.
Consequently,
\[
 g_\varepsilon(A)=\max_{B\in\mathcal B_2}F_\varepsilon(A,B)
\]
is convex as a pointwise maximum of convex functions and is continuous by
compactness of $\mathcal B_2$ and continuity of $F_\varepsilon$.

The sets $\mathcal A_2$ and $\mathcal B_2$ are finite-dimensional compact convex
sets.  Sion's theorem \cite{Sion}, with $A$ as the compact minimizing variable,
now gives
\[
\inf_{A\in\mathcal A_2}\sup_{B\in\mathcal B_2}F_\varepsilon(A,B)
 =\sup_{B\in\mathcal B_2}\inf_{A\in\mathcal A_2}F_\varepsilon(A,B).
\]
\end{proof}

\section{Reconstruction and algorithmic guarantees}
\label{app:algorithm-proofs}

We prove compatibility and martingale reconstruction, derive the finite mirror step
and outer domains, and establish the convergence and oracle error bounds.

\subsection{Reconstruction and finite algorithms}

\subsubsection{Proof of Proposition~\ref{prop:reconstruction}}\label{proof:10.1}

\begin{proof}
For fixed $(A,B)$, Proposition~\ref{prop:wot} gives a unique optimal barycentric
map $m_{A,B}$. Adding the two optimality inequalities at nearby parameter pairs
shows that these maps converge strongly when $(A,B)$ converges. Thus the
optimizer-stability hypothesis of Danskin's theorem
(Theorem~\ref{thm:danskin-envelope}) holds. Let $H_A,H_B$ be matrix directions;
that theorem and direct
differentiation of the inner cost give
\begin{align*}
 D_A\mathsf W_{A,B}^\varepsilon[H_A]
 &=-4\ip{H_A}{M_{m_{A,B}}}_F,\\
 D_B\mathsf W_{A,B}^\varepsilon[H_B]
 &=2\ip{H_B}{S_{m_{A,B}}}_F.
\end{align*}
Adding the derivatives of the outer regularizers yields
\begin{equation}\label{eq:app-outer-gradients}
 \nabla_AF_\varepsilon(A,B)=4(A-M_{m_{A,B}}),\qquad
 \nabla_BF_\varepsilon(A,B)=2(S_{m_{A,B}}-B).
\end{equation}

The estimate \eqref{eq:ridge-oracle-lipschitz}, together with the outer
quadratic terms in $F_\varepsilon$, proves that $F_\varepsilon$ is continuous.
Compactness of the two outer balls therefore gives an optimizer of the
nested problem.

For fixed $A$, $F_\varepsilon(A,\cdot)$ is $2$-strongly concave and hence has a
unique maximizer.  At $B^\star$, the first-order inequality for constrained
maximization is
\[
 \ip{\nabla_BF_\varepsilon(A^\star,B^\star)}{B-B^\star}_F\le0
 \qquad(B\in\mathcal B_2).
\]
The feasible choice $B=S_{m^\star}$ and \eqref{eq:app-outer-gradients} give
\[
 2\norm{S_{m^\star}-B^\star}_F^2\le0,
\]
so $B^\star=S_{m^\star}$, including on the boundary of the ball.

Let $g(A)=\max_{B\in\mathcal B_2}F_\varepsilon(A,B)$.  Uniqueness of the
maximizer and Theorem~\ref{thm:danskin-envelope} make $g$ differentiable.  At its constrained
minimum $A^\star$,
\[
 \ip{4(A^\star-M_{m^\star})}{A-A^\star}_F\ge0
 \qquad(A\in\mathcal A_2).
\]
Choosing the feasible point $A=M_{m^\star}$ gives
$-4\norm{A^\star-M_{m^\star}}_F^2\ge0$, and therefore
$A^\star=M_{m^\star}$.

Substituting compatibility into the envelope gives
\begin{align*}
 &2\norm{M_{m^\star}}_F^2-\norm{S_{m^\star}}_F^2
 +\int c_{M_{m^\star},S_{m^\star}}^\varepsilon
       (x,m^\star(x))d\mu(x)\\
 &\qquad=\norm{S_{m^\star}}_F^2-2\norm{M_{m^\star}}_F^2
 +\varepsilon\norm{m^\star}_{L^2}^2.
\end{align*}
Theorem~\ref{thm:ridge-dual} identifies this with the nonconstant part of the
optimal ridge primal, so $m^\star$ is an optimizer of the ridge wIGW primal.

Let $M=m^\star(X)$ under an optimal inner coupling $\pi^\star$.  A regular
conditional law of $Y$ given $M$ exists on Euclidean Borel spaces
\cite[Theorems~3.4 and~8.5]{Kallenberg}; define
\[
 \kappa^\star(dy\mid z):=\mathcal L(Y\in dy\mid M=z).
\]
Since $\E[Y\mid X]=m^\star(X)=M$, the tower property gives
\[
 \E[Y\mid M]=\E[\E[Y\mid X]\mid M]=\E[M\mid M]=M.
\]
Thus $\int y\,\kappa^\star(dy\mid z)=z$ for $(m^\star)_\#\mu$-almost every $z$,
so $\kappa^\star$ is a martingale kernel from $(m^\star)_\#\mu$ to $\nu$.  Gluing
this kernel after $m^\star$ produces a coupling with conditional mean $m^\star$,
and its ridge wIGW objective remains optimal.
\end{proof}

\subsubsection{Proof of Proposition~\ref{prop:finite-inner-mirror}}
\label{proof:11.1}\label{proof:finite-inner-mirror}

\begin{proof}
For $P\in\Pi(a,b)$,
\[
 m_i(P)=a_i^{-1}\sum_jP_{ij}y_j,\qquad
 \frac{\partial m_i}{\partial P_{ij}}=\frac{y_j}{a_i}.
\]
Only the $i$th barycentric row depends on $P_{ij}$.  Since
$C_{B,\varepsilon}$ is symmetric, the chain rule gives
\begin{align*}
 \frac{\partial}{\partial P_{ij}}
 a_i\{m_i^\top C_{B,\varepsilon}m_i-4x_i^\top A m_i\}
 &=(2C_{B,\varepsilon}m_i-4A^\top x_i)^\top y_j.
\end{align*}
This proves \eqref{eq:finite-oracle-gradient}. For $P,Q\in\Pi(a,b)$, set
$d_i=m_i(Q)-m_i(P)$. Since $P\mapsto m_i(P)$ is linear, direct expansion gives
\begin{align*}
 &f_{A,B}(Q)-f_{A,B}(P)-\ip{\nabla f_{A,B}(P)}{Q-P}_F\\
 &\qquad=\sum_i a_i d_i^\top C_{B,\varepsilon}d_i
 \ge\alpha_B\sum_i a_i\norm{d_i}^2.
\end{align*}
This is \eqref{eq:finite-mean-strong-convexity}; in particular, $f_{A,B}$ is
convex in $P$ and $2\alpha_B$-strongly convex in the induced barycentric vector.
If $P\ne Q$ but $m_i(P)=m_i(Q)$ for every $i$, the right side vanishes, so the
objective need not be strictly or strongly convex in the plan itself.

For the generalized relative entropy defined before
\eqref{eq:finite-mirror-step}, let \(P\in\Pi(a,b)\) be strictly positive and
set
\[
 G(P)=\nabla f_{A,B}(P),\qquad
 c(P)=\max\{\norm{G(P)}_\infty,10^{-12}\}.
\]
The ideal normalized mirror step solves
\[
 \argmin_{Q\in\Pi(a,b)}
 \left\{\ip{\nabla f_{A,B}(P)}Q_F+
 \frac{c(P)}{\gamma}\KL(Q\,\|\,P)\right\}.
\]
Without the marginal constraints, differentiating coordinatewise gives
\[
 \log(Q_{ij}/P_{ij})
 =-\frac{\gamma}{c(P)}(\nabla f_{A,B}(P))_{ij},
\]
and hence
\[
 \widetilde P
 =P\odot\exp\!\left(-\frac{\gamma}{c(P)}\nabla f_{A,B}(P)\right).
\]
The constrained minimizer is the KL projection of $\widetilde P$ onto
$\Pi(a,b)$.  Its Lagrange multipliers separate by rows and columns, so it has the
form $\operatorname{diag}(r)\widetilde P\operatorname{diag}(s)$; alternating
row and column scalings compute $r,s$ when run to convergence
\cite[Section~4.2]{PeyreCuturiOT}. Algorithm~\ref{alg:implemented-ab} instead
floors the kernel at \(10^{-300}\), stops after at most \(N_{\rm sk}\)
iterations, and repairs its remaining marginal error. It therefore implements
an inexact version of the ideal projection.
\end{proof}

\subsubsection{Proof of Proposition~\ref{prop:balls}}\label{proof:11.2}

\begin{proof}
For every row,
\[
 \norm{m_i(P)}^2
 =\norm{\E[Y\mid X=x_i]}^2
 \le\E[\norm Y^2\mid X=x_i]
\]
by conditional Jensen.  Multiplying by $a_i$ and summing gives
$\sum_i a_i\norm{m_i(P)}^2\le M_2(\nu)$.  Cauchy--Schwarz then yields
\begin{align*}
 \norm{M_P}_F
 &=\norm{\sum_i a_ix_im_i^\top}_F
 \le\sum_i a_i\norm{x_i}\norm{m_i}\\
 &\le\sqrt{M_2(\mu)M_2(\nu)}.
\end{align*}
Moreover, $S_P\succeq0$, so
\[
 \norm{S_P}_F\le\Tr(S_P)=\sum_i a_i\norm{m_i}^2\le M_2(\nu).
\]
By the compatibility hypothesis, $A=M_P$ and $B=S_P$, so these bounds
apply to every compatible tuple. Proposition~\ref{prop:reconstruction} shows, in
particular, that exact nested outer optimizers are compatible.
\end{proof}

\subsubsection{Proof of Theorem~\ref{thm:convergence}}\label{proof:11.3}

\begin{proof}
We separate curvature, sensitivity, and the projected iteration.

\emph{Curvature.}
Let $Tm=M_m$.  As in the proof of Theorem~\ref{thm:swap},
$\norm{T}^2\le\lambda_X$.  For fixed $B$, the second variation of the joint
$(A,m)$ objective in a direction $(H,h)$ is
\[
 4\norm H_F^2-8\ip H{Th}_F
 +2\varepsilon\norm h_{L^2}^2
 +4\int h(x)^\top Bh(x)d\mu(x).
\]
The last term is nonnegative.  Completing the square in $h$ and using
$\norm{T^*H}_{L^2}^2\le\lambda_X\norm H_F^2$ gives
\[
 4\norm H_F^2-8\ip{T^*H}{h}_{L^2}
 +2\varepsilon\norm h_{L^2}^2
 \ge\left(4-\frac{8\lambda_X}{\varepsilon}\right)\norm H_F^2.
\]
Partial minimization over $m\in\mathcal C_\nu$ preserves this strong convexity in
$A$.  For fixed $A$, the weak OT value is an infimum of affine functions of $B$;
it is concave, and $-\norm B_F^2$ makes $F_\varepsilon(A,\cdot)$
$2$-strongly concave. The saddle gradient operator is therefore strongly monotone
with parameter
\[
 \alpha=\min\left\{4-\frac{8\lambda_X}{\varepsilon},2\right\}>0.
\]
Set
\[
 \beta:=4-\frac{8\lambda_X}{\varepsilon},\qquad
 \Delta A:=A_2-A_1,\qquad \Delta B:=B_2-B_1,
\]
and write $F=F_\varepsilon$.  Strong convexity in $A$ and strong concavity in $B$
give
\begin{align*}
 F(A_2,B_1)
 &\ge F(A_1,B_1)
   +\ip{\nabla_AF(A_1,B_1)}{\Delta A}_F
   +\frac{\beta}{2}\norm{\Delta A}_F^2,\\
 F(A_1,B_2)
 &\ge F(A_2,B_2)
   -\ip{\nabla_AF(A_2,B_2)}{\Delta A}_F
   +\frac{\beta}{2}\norm{\Delta A}_F^2,\\
 F(A_1,B_2)
 &\le F(A_1,B_1)
   +\ip{\nabla_BF(A_1,B_1)}{\Delta B}_F
   -\norm{\Delta B}_F^2,\\
 F(A_2,B_1)
 &\le F(A_2,B_2)
   -\ip{\nabla_BF(A_2,B_2)}{\Delta B}_F
   -\norm{\Delta B}_F^2.
\end{align*}
Comparing the sum of the first two inequalities with the sum of the last two
cancels the four function values and yields
\begin{align*}
 &\ip{\nabla_AF(A_2,B_2)-\nabla_AF(A_1,B_1)}{\Delta A}_F\\
 &\qquad
 -\ip{\nabla_BF(A_2,B_2)-\nabla_BF(A_1,B_1)}{\Delta B}_F\\
 &\ge \beta\norm{\Delta A}_F^2+2\norm{\Delta B}_F^2\\
 &\ge \alpha\bigl(\norm{\Delta A}_F^2+\norm{\Delta B}_F^2\bigr).
\end{align*}
The left-hand side is $\ip{G(z_2)-G(z_1)}{z_2-z_1}$, which proves the asserted
strong monotonicity.

\emph{Sensitivity of the weak OT map.}
Let $m_i=m_{A_i,B_i}$, $d=m_1-m_2$, and
\[
 g_{A,B}(m)(x)=2C_{B,\varepsilon}m(x)-4A^\top x.
\]
Optimality on the convex feasible set gives
\[
 \ip{g_{A_1,B_1}(m_1)}{m_2-m_1}_{L^2}\ge0,\qquad
 \ip{g_{A_2,B_2}(m_2)}{m_1-m_2}_{L^2}\ge0.
\]
Adding, expanding $C_{B,\varepsilon}=2B+\varepsilon I$, and discarding the
nonnegative term $4\int d^\top B_1d\,d\mu$ gives
\begin{align*}
 2\varepsilon\norm d_{L^2}^2
 &\le4\ip{A_1-A_2}{Td}_F
 -4\int d(x)^\top(B_1-B_2)m_2(x)d\mu(x)\\
 &\le4\sqrt{\lambda_X}\norm{A_1-A_2}_F\norm d_{L^2}
 +4\sqrt{M_2(\nu)}\norm{B_1-B_2}_F\norm d_{L^2}.
\end{align*}
After division when $d\ne0$ (the zero case is trivial),
\begin{equation}\label{eq:app-map-sensitivity}
 \norm d_{L^2}
 \le\frac2\varepsilon\left(
 \sqrt{\lambda_X}\norm{A_1-A_2}_F+
 \sqrt{M_2(\nu)}\norm{B_1-B_2}_F\right).
\end{equation}

The moment differences satisfy
\[
 \norm{M_{m_1}-M_{m_2}}_F\le\sqrt{\lambda_X}\norm d_{L^2}.
\]
Also,
$m_1m_1^\top-m_2m_2^\top=d\,m_1^\top+m_2d^\top$, and therefore
\[
 \norm{S_{m_1}-S_{m_2}}_F
 \le(\norm{m_1}_{L^2}+\norm{m_2}_{L^2})\norm d_{L^2}
 \le2\sqrt{M_2(\nu)}\norm d_{L^2}.
\]
Using \eqref{eq:app-outer-gradients} and \eqref{eq:app-map-sensitivity}, the
vector of the two block gradient norms is bounded componentwise by
\[
 \left[
 \begin{pmatrix}4&0\\0&2\end{pmatrix}
 +\frac8\varepsilon
 \begin{pmatrix}\sqrt{\lambda_X}\\\sqrt{M_2(\nu)}\end{pmatrix}
 \begin{pmatrix}\sqrt{\lambda_X}&\sqrt{M_2(\nu)}\end{pmatrix}
 \right]
 \begin{pmatrix}\norm{A_1-A_2}_F\\\norm{B_1-B_2}_F\end{pmatrix}.
\]
Thus $L=\norm K_{\rm op}$ is a global Lipschitz constant.  The operator norm of
the diagonal term is $4$, and that of the rank-one term is
$8(\lambda_X+M_2(\nu))/\varepsilon$, giving the simpler bound in
Theorem~\ref{thm:convergence}.

\emph{Projected contraction.}
Strong monotonicity makes the solution of the variational inequality unique.  The Hilbert
projection characterization \cite[Theorem~3.16]{BauschkeCombettes} gives the
fixed-point identity
\[
 z^\star=\operatorname{Proj}_{\mathcal Z}(z^\star-\eta G(z^\star)).
\]
Nonexpansiveness of Euclidean projection
\cite[Proposition~4.16]{BauschkeCombettes} gives
\begin{align*}
 \norm{z_{k+1}-z^\star}
 &\le\norm{z_k-z^\star-\eta(G(z_k)-G(z^\star))-\eta e_k}\\
 &\le\norm{z_k-z^\star-\eta(G(z_k)-G(z^\star))}+\eta\norm{e_k}.
\end{align*}
Squaring the first norm, then using strong monotonicity and Lipschitz continuity,
gives
\[
 \norm{z_k-z^\star-\eta(G(z_k)-G(z^\star))}^2
 \le(1-2\eta\alpha+\eta^2L^2)\norm{z_k-z^\star}^2.
\]
Taking square roots proves the recurrence.  The condition
$0<\eta<2\alpha/L^2$ implies $1-2\eta\alpha+\eta^2L^2<1$, and hence
$q=\sqrt{1-2\eta\alpha+\eta^2L^2}<1$.
Writing \(d_k=\norm{z_k-z^\star}\) and iterating the recurrence gives
\[
 d_k\le q^k d_0+
 \eta\sum_{j=0}^{k-1}q^{k-1-j}\norm{e_j}.
\]
For exact oracles this reduces to \(d_k\le q^kd_0\). If
\(\sup_j\norm{e_j}\le\bar e\), summing the geometric series gives
\[
 d_k\le q^kd_0+\eta\frac{1-q^k}{1-q}\bar e,
 \qquad
 \limsup_{k\to\infty}d_k\le\frac{\eta\bar e}{1-q}.
\]
Finally, if \(\norm{e_j}\to0\), split the convolution sum at a fixed index
\(N\). Its finite initial part tends to zero with \(k\), while the tail is at
most \(\eta\sup_{j\ge N}\norm{e_j}/(1-q)\). Letting \(N\to\infty\) proves
\(d_k\to0\).
\end{proof}

\subsubsection{Proof of Proposition~\ref{prop:inner-gap-oracle-error}}
\label{proof:11.4}

\begin{proof}
Let $m=m_{A,B}$ and let $\widehat m$ be the barycentric map of an inner feasible
coupling. The proof uses only feasibility and the suboptimality of $\widehat P$
for the unregularized objective $f_{A,B}$; it therefore applies to a repaired KL
mirror iterate once that suboptimality is controlled. An optimizer of a different
entropy-regularized objective requires a separate bias bound before this argument
can be used. The Hessian of the inner objective in the barycentric vector is
$2C_{B,\varepsilon}$, whose
smallest eigenvalue is $2\alpha_B$. Strong convexity and the variational inequality
at the constrained minimizer imply that the objective suboptimality satisfies
\[
 \delta\ge\alpha_B\norm{\widehat m-m}_{L^2}^2,
 \qquad
 \norm{\widehat m-m}_{L^2}\le\sqrt{\delta/\alpha_B}.
\]
For the $A$ block, with $d=\widehat m-m$,
\[
 \norm{e_A}_F=4\norm{Td}_F
 \le4\sqrt{\lambda_X}\norm d_{L^2}.
\]
For the $B$ block, feasibility gives
$\norm{m}_{L^2},\norm{\widehat m}_{L^2}\le\sqrt{M_2(\nu)}$.  The identity
$\widehat m\widehat m^\top-mm^\top=d\,\widehat m^\top+md^\top$ yields
\[
 \norm{S_{\widehat m}-S_m}_F
 \le2\sqrt{M_2(\nu)}\norm d_{L^2},
\]
and hence
$\norm{e_B}_F=2\norm{S_{\widehat m}-S_m}_F
\le4\sqrt{M_2(\nu)}\norm d_{L^2}$.  Combining the two blocks in product
Frobenius norm gives
\[
 \norm e\le4\sqrt{\lambda_X+M_2(\nu)}\norm d_{L^2}
 \le4\sqrt{\frac{(\lambda_X+M_2(\nu))\delta}{\alpha_B}}.
\]
\end{proof}

\section{General geometry with conditional laws}\label{app:general-wgw-geometry}

We use Figure~\ref{fig:general-wgw-geometry} to isolate the general construction behind
Definition~\ref{def:generic-wgw}. A coupling turns each source point into a
probability-valued target node; a source relation and the two conditional nodes are
then passed to an arbitrary lifted cost.

\begin{figure}[H]
\centering
\resizebox{\textwidth}{!}{%
\begin{tikzpicture}[
  font=\small,
  >=Latex,
  panel/.style={rounded corners=3pt,draw=black!22,fill=black!1},
  ambient/.style={circle,fill=blue!20!white,draw=blue!48!black,
    line width=.55pt,inner sep=2.25pt},
  chosen/.style={circle,fill=blue!72!black,draw=white,
    line width=.4pt,inner sep=3pt},
  atom/.style={circle,fill=red!68!black,draw=white,
    line width=.3pt,inner sep=2pt},
  law/.style={ellipse,draw=red!62!black,fill=red!6,
    line width=.9pt,minimum width=3.55cm,minimum height=1.05cm},
  input/.style={circle,draw=black!40,fill=white,line width=.7pt,
    minimum size=.72cm},
  cost/.style={rounded corners=3pt,draw=green!45!black,fill=green!7,
    line width=1pt,minimum width=2.55cm,minimum height=1.05cm},
  stage/.style={->,line width=1pt,black!58},
  feed/.style={->,line width=.85pt,black!62}
]
  \draw[panel] (0,0) rectangle (4.05,4.55);
  \draw[panel] (5.35,0) rectangle (9.85,4.55);
  \draw[panel] (11.15,0) rectangle (15.25,4.55);

  \node[font=\bfseries,anchor=north] at (2.025,4.34) {source relation};
  \node[font=\bfseries,anchor=north] at (7.60,4.34) {conditional target laws};
  \node[font=\bfseries,anchor=north] at (13.20,4.34) {lifted comparison};

  \node[ambient] at (.72,1.10) {};
  \node[ambient] at (1.08,3.72) {};
  \node[ambient] at (2.20,.72) {};
  \node[ambient] at (3.48,3.48) {};
  \node[ambient] at (3.48,1.02) {};
  \node[chosen,label=left:$x$] (gx) at (1.12,2.78) {};
  \node[chosen,label=right:$x'$] (gxp) at (2.98,1.82) {};
  \draw[<->,line width=1.05pt,blue!70!black] (gx)--(gxp);
  \node[font=\footnotesize,fill=white,inner sep=2pt] at (2.02,3.28)
    {$a=c_{\mathcal X}(x,x')$};
  \node[blue!60!black] at (2.025,.34) {$x,x'\sim\mu$};

  \node[law] (lawx) at (7.60,2.98) {};
  \node[law] (lawxp) at (7.60,1.42) {};
  \node[red!65!black] at (7.60,3.67) {$\pi_x$};
  \node[red!65!black] at (7.60,2.11) {$\pi_{x'}$};
  \foreach \p/\s in {
    {6.62,2.98}/1.55,{7.00,2.80}/2.25,{7.43,3.10}/1.75,
    {7.88,2.82}/2.55,{8.34,3.08}/1.45,{8.67,2.92}/2.05}
    \node[atom,inner sep=\s pt] at (\p) {};
  \foreach \p/\s in {
    {6.61,1.44}/2.35,{7.03,1.25}/1.45,{7.43,1.53}/2.65,
    {7.90,1.29}/1.70,{8.31,1.56}/2.20,{8.67,1.37}/1.45}
    \node[atom,inner sep=\s pt] at (\p) {};
  \node[blue!60!black] at (7.60,.34)
    {$\pi(dx,dy)=\mu(dx)\pi_x(dy)$};

  \node[input] (ina) at (11.88,3.14) {$a$};
  \node[input,draw=red!55!black,text=red!65!black] (inrho) at (11.88,2.18)
    {$\pi_x$};
  \node[input,draw=red!55!black,text=red!65!black] (inrhop) at (11.88,1.22)
    {$\pi_{x'}$};
  \node[cost] (liftcost) at (13.78,2.18)
    {$\mathfrak C(a,\pi_x,\pi_{x'})$};
  \draw[feed] (ina.east)--([yshift=9pt]liftcost.west);
  \draw[feed] (inrho.east)--(liftcost.west);
  \draw[feed] (inrhop.east)--([yshift=-9pt]liftcost.west);
  \node[align=center,green!40!black] at (13.20,.48)
    {measurable conditional cost\\in $[0,+\infty]$};

  \draw[stage] (4.18,2.28)--(5.22,2.28)
    node[midway,above=3pt] {$\pi$};
  \draw[stage] (9.98,2.28)--(11.02,2.28)
    node[midway,above=3pt] {$\mathfrak C$};
\end{tikzpicture}}
\caption{Geometry of general weak GW with conditional laws. A coupling assigns to
each source point $x$ a target probability law $\pi_x$. For every source pair, the
source relation $a=c_{\mathcal X}(x,x')$ and the two induced laws are inputs to the
lifted cost $\mathfrak C(a,\pi_x,\pi_{x'})$. The objective averages this quantity under
$\mu\otimes\mu$ and optimizes the coupling. The construction uses measurable
relations on standard Borel spaces; linear, barycentric, and metric structures enter
through particular choices of the lifted cost.}
\label{fig:general-wgw-geometry}
\end{figure}
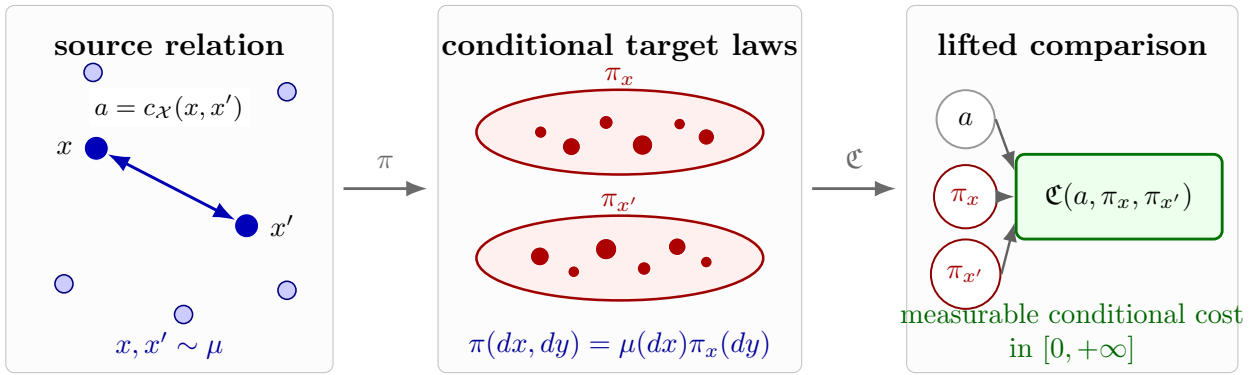

\begin{remark}[Further relations between conditional laws]
\label{rem:conditional-law-relations}
The aggregated construction permits several choices of $D$, each paired naturally
with a source relation.
\begin{enumerate}[label=(\roman*),leftmargin=2em]
\item If $\mathcal Y=\R^{d_y}$ and $\rho,\rho'\in\cP_1(\mathcal Y)$, let
\[
 D_{\rm bar}(\rho,\rho')
 :=\left\langle\int y\,d\rho(y),\int y'\,d\rho'(y')\right\rangle.
\]
Together with $c_{\mathcal X}(x,x')=\langle x,x'\rangle$ and
$\mathcal L(a,b)=(a-b)^2$, this is the barycentric wIGW specialization studied in
the main text.

\item On $\cP_2(\R^{d_y})$, the \emph{maximal covariance relation}
\[
 D_{\rm MCov}(\rho,\rho')
 :=\sup_{\omega\in\Pi(\rho,\rho')}
   \int\langle y,y'\rangle\,d\omega(y,y')
 =\frac12\left(
   \int\norm y^2d\rho(y)+\int\norm{y'}^2d\rho'(y')
   -W_2^2(\rho,\rho')\right)
\]
incorporates the conditional spreads through optimal pairwise alignment.
Regularization against a fixed reference law addresses the selection of martingale
gluings; see \cite{GuoNilssonWiesel2025} for the related martingale relaxation.

\item For Euclidean distance relations, one may take
\[
 c_{\mathcal X}(x,x')=\norm{x-x'},
 \qquad D_{W_2}(\rho,\rho')=W_2(\rho,\rho'),
\]
on $\cP_2(\mathcal Y)$ together with, for example,
$\mathcal L(a,b)=(a-b)^2$. The resulting loss compares a source distance with the
Wasserstein distance between two conditional target laws.

Both $D_{\rm MCov}$ and $D_{W_2}$ contain an inner optimal transport problem between
the conditional laws. They therefore produce an outer relational optimization with
nested conditional transports, an architecture related to nested distance
\cite{PflugPichler2012}, adapted or bicausal Wasserstein distance
\cite{BackhoffBartlBeiglbockEder2020}, and recursive entropic transport algorithms
\cite{QuTran2021}. Definition~\ref{def:generic-wgw} averages relations between two
rows induced by a single coupling. Nested and adapted transport encode a filtration
through recursive or causal coupling constraints. An equivalence between these
formulations would require additional assumptions.

\item Let $k(y,y')=\langle\psi(y),\psi(y')\rangle_{\mathcal H}$ be a measurable
positive definite kernel for which the Bochner mean embedding
$m_k(\rho):=\int\psi(y)d\rho(y)$ exists. Two associated relations are
\[
 \begin{aligned}
 D_k^{\rm ip}(\rho,\rho')
   &:=\langle m_k(\rho),m_k(\rho')\rangle_{\mathcal H}
     =\iint k(y,y')d\rho(y)d\rho'(y'),\\
 D_k^{\rm MMD}(\rho,\rho')
   &:=\norm{m_k(\rho)-m_k(\rho')}_{\mathcal H}.
 \end{aligned}
\]
The first pairs with a source kernel or feature inner product, and the second with a
source feature distance. MMD is a pseudometric in general and a metric for
characteristic kernels; see \cite{MuandetEtAl,ParkMuandet}.
\end{enumerate}

These functionals are measurable on the indicated moment classes, or on classes with
the stated RKHS integrability, under their standard topologies.
Definition~\ref{def:generic-wgw} may be restricted to those classes or completed by
measurable extensions. The projection theorem, moment reduction, and $A$--$B$
duality in the main text are established for the finite-dimensional bilinear
relation $D_{\rm bar}$. Corresponding structural and computational results for
$D_{\rm MCov}$, $D_{W_2}$, $D_k^{\rm ip}$, and $D_k^{\rm MMD}$ require separate
analysis and are left for future work.
\end{remark}

\section{Reference theorems used}\label{app:reference-theorems}

For completeness, we state the precise forms of Strassen's martingale theorem,
Danskin's envelope theorem, and Sion's minimax theorem used in our arguments. The
proof of Sion invokes
the standard finite form of Fan's KKM lemma \cite{FanKKM}. Kellerer's theorem gives
the corresponding martingale existence result for suitable time-indexed
one-dimensional families increasing in convex order \cite{Kellerer}; this paper uses
Strassen's two-time theorem.

\Needspace{12\baselineskip}
\subsection{Strassen's martingale characterization}\label{subsec:strassen}

\begin{theorem}[Strassen \cite{Strassen}]
Let $\eta$ and $\nu$ be Borel probability measures on $\R^d$ with finite first
moments. The following are equivalent.
\begin{enumerate}[label=(\roman*),leftmargin=2em]
\item For every finite continuous convex function $u:\R^d\to\R$ of at most linear
growth,
\[
 \int u\,d\eta\le\int u\,d\nu.
\]
Equivalently, the same inequality holds for every convex $u$ for which both
integrals are well defined.
\item There exists a Borel probability measure
$\kappa\in\Pi(\eta,\nu)$ on $\R^d\times\R^d$ such that, for the coordinate pair
$(Z,Y)$,
\[
 \E_\kappa[Y\mid Z]=Z\qquad\kappa\text{-almost surely}.
\]
\item There exists a Borel probability kernel $z\mapsto\kappa_z$ on $\R^d$ such
that
\[
 \int\kappa_z\,d\eta(z)=\nu,
 \qquad
 \int y\,d\kappa_z(y)=z
 \quad\text{for }\eta\text{-almost every }z.
\]
\end{enumerate}
In (ii)--(iii), finite first moments make the conditional barycenter well
defined. Condition (i) automatically forces equality of the barycenters of $\eta$
and $\nu$, because affine functions and their negatives are convex.
\end{theorem}

\begin{proof}
The equivalence of (ii) and (iii) is disintegration on Euclidean Borel spaces
\cite[Theorem~3.4]{Kallenberg}:
\[
 \kappa(dz,dy)=\eta(dz)\kappa_z(dy),
 \qquad \E[Y\mid Z=z]=\int y\,d\kappa_z(y).
\]
Thus the conditional martingale identity is the barycenter identity in
(iii).

Assume (ii).  For every integrable convex $u$, conditional Jensen gives
\[
 u(Z)=u(\E[Y\mid Z])\le\E[u(Y)\mid Z].
\]
Integrating proves $\int u\,d\eta\le\int u\,d\nu$, hence (i).

For (i)$\Rightarrow$(ii), we give the separation argument.  Equip
$\Pi(\eta,\nu)$ with weak convergence together with convergence of first moments.
It is compact because the fixed marginals give tightness and uniform integrability
of $\norm z+\norm y$; the compactness input is Prokhorov's theorem
\cite[Theorem~5.1]{Billingsley}.  For every bounded continuous vector field
$h:\R^d\to\R^d$, define
\[
 \mathcal F_h(\kappa)=\int h(z)^\top(y-z)d\kappa(z,y).
\]
This functional is continuous: it is weakly continuous after compact truncation,
and its tails are controlled by boundedness of $h$ and the fixed first moments.  A
coupling is a martingale if and only if $\mathcal F_h(\kappa)=0$ for every such
$h$; bounded continuous tests determine conditional expectation.

Suppose no martingale coupling exists.  The image of the compact convex set
$\Pi(\eta,\nu)$ under the family $(\mathcal F_h)_h$ is compact and convex in the
product topology and does not contain the origin.  Strict separation in this locally
convex product space follows from the Hahn--Banach separation theorem
\cite[Chapter~1]{BrezisFA} and uses finitely many coordinates.  Combining their vector fields
produces one bounded continuous $h$ and some $\epsilon>0$ such that
\begin{equation}\label{eq:appendix-strassen-separation}
 \inf_{\kappa\in\Pi(\eta,\nu)}
 \int h(z)^\top(y-z)d\kappa(z,y)\ge\epsilon.
\end{equation}

Apply Kantorovich duality \cite[Theorem~5.10]{VillaniOT} to the continuous cost with
at most linear growth
$c_h(z,y)=h(z)^\top(y-z)$.  By
\eqref{eq:appendix-strassen-separation}, there are integrable dual functions $a,b$
such that
\[
 a(z)+b(y)\le h(z)^\top(y-z),\qquad
 \int a\,d\eta+\int b\,d\nu>0.
\]
Define
\[
 u(y)=\sup_z\{a(z)+h(z)^\top z-h(z)^\top y\}.
\]
It is a supremum of affine functions and therefore convex.  Boundedness of $h$
bounds all slopes; the dual constraint evaluated at one fixed target point bounds the
intercepts from above.  Thus $u$ is finite and has at most linear growth.  Evaluating
at $z=y$ gives $a(y)\le u(y)$, while the dual constraint for every $z$ gives
\[
 b(y)\le\inf_z\{h(z)^\top y-h(z)^\top z-a(z)\}=-u(y).
\]
Convex order now implies the contradiction
\[
 0<\int a\,d\eta+\int b\,d\nu
 \le\int u\,d\eta-\int u\,d\nu\le0.
\]
Therefore a martingale coupling exists.
\end{proof}

Every application of this theorem uses Euclidean Borel spaces and measures in
$\cP_1$ or $\cP_2$; these settings satisfy the stated measurable and topological
hypotheses.

\subsection{Danskin's envelope theorem}\label{subsec:danskin}

\begin{theorem}[Danskin \cite{Danskin1966}]\label{thm:danskin-envelope}
Let $\Theta\subset\R^q$ be open, let $\mathcal X$ be a metric space, and let
$\Phi:\Theta\times\mathcal X\to\R$ be differentiable in its first variable.
For either choice $\operatorname{ext}\in\{\min,\max\}$, define
\[
 V(\theta)=\operatorname{ext}_{x\in\mathcal X}\Phi(\theta,x),
 \qquad
 \mathcal X^\star(\theta)
 =\arg\operatorname{ext}_{x\in\mathcal X}\Phi(\theta,x).
\]
Fix $\theta_0\in\Theta$. Assume that the extremum is attained for $\theta$ near
$\theta_0$ and that the optimizers are locally stable: whenever
$\theta_n\to\theta_0$ and $x_n\in\mathcal X^\star(\theta_n)$, a subsequence
converges to some $x^\star\in\mathcal X^\star(\theta_0)$. Assume also joint
gradient continuity along such optimizer sequences: for every additional
$\vartheta_n\to\theta_0$,
$\nabla_\theta\Phi(\vartheta_n,x_n)\to
\nabla_\theta\Phi(\theta_0,x^\star)$ along that subsequence. These hypotheses
hold, in particular, when $\mathcal X$ is compact and
$\Phi,\nabla_\theta\Phi$ are jointly continuous.

Then the directional derivative of $V$ at $\theta_0$ is
\[
 V'(\theta_0;h)=
 \begin{cases}
  \displaystyle\min_{x\in\mathcal X^\star(\theta_0)}
  \ip{\nabla_\theta\Phi(\theta_0,x)}{h},&\operatorname{ext}=\min,\\[7pt]
  \displaystyle\max_{x\in\mathcal X^\star(\theta_0)}
  \ip{\nabla_\theta\Phi(\theta_0,x)}{h},&\operatorname{ext}=\max.
 \end{cases}
\]
If all active optimizers give the same outer gradient $p$, then $V$ is
differentiable at $\theta_0$ and $\nabla V(\theta_0)=p$. In particular, this
holds when the optimizer is unique.
\end{theorem}

\begin{proof}
We give the maximization argument; minimization follows by applying it to
$-\Phi$. For any $x\in\mathcal X^\star(\theta_0)$, comparison with the same $x$
at $\theta_0+th$ gives
\[
 \liminf_{t\downarrow0}
 \frac{V(\theta_0+th)-V(\theta_0)}{t}
 \ge\ip{\nabla_\theta\Phi(\theta_0,x)}h.
\]
For the reverse inequality, choose
$x_t\in\mathcal X^\star(\theta_0+th)$. Optimizer stability supplies a
subsequence converging to an active $x^\star$ at $\theta_0$. Since
$V(\theta_0)\ge\Phi(\theta_0,x_t)$, the mean-value formula and gradient
continuity give the matching upper bound
$\ip{\nabla_\theta\Phi(\theta_0,x^\star)}h$. Maximizing over the active set
proves the formula. If the active gradients coincide, the same two comparisons
for arbitrary increments give a uniform $o(\norm h)$ remainder and hence
differentiability.
\end{proof}

For Proposition~\ref{prop:reconstruction}, the ridge makes the optimal mean map
unique and the two optimality inequalities give its strong stability under
changes of $(A,B)$. For the finite oracle, the transport polytope is compact and
all minimizing plans have the same barycentric vector, hence the same outer
gradient. For the outer maximum over $B$, the maximizing matrix is unique. These
facts verify the hypotheses above in every use in Sections~\ref{sec:reconstruction}
and~\ref{sec:algorithm}.

\Needspace{8\baselineskip}
\subsection{Sion's minimax theorem}\label{subsec:sion}

For a real-valued function $f$ on a convex set, quasi-convexity means that every
sublevel set is convex; quasi-concavity means that every superlevel set is convex.

\begin{theorem}[Sion \cite{Sion}]
Let $E$ and $F$ be real Hausdorff topological vector spaces. Let
$X\subset E$ be a nonempty compact convex set and let $Y\subset F$ be a nonempty
convex set. Suppose $f:X\times Y\to\R$ satisfies:
\begin{enumerate}[label=(\roman*),leftmargin=2em]
\item for every $y\in Y$, the section $x\mapsto f(x,y)$ is lower semicontinuous
and quasi-convex on $X$;
\item for every $x\in X$, the section $y\mapsto f(x,y)$ is upper semicontinuous
and quasi-concave on $Y$.
\end{enumerate}
Then
\[
 \min_{x\in X}\sup_{y\in Y}f(x,y)
 =\sup_{y\in Y}\min_{x\in X}f(x,y).
\]
Only the minimizing set $X$ is required to be compact. In our applications the
two minimax values are finite; when both sets are compact and the corresponding
sections are semicontinuous, the outer minimum and maximum also have optimizers.
\end{theorem}

\begin{proof}
Set
\[
 \underline v=\sup_{y\in Y}\min_{x\in X}f(x,y),\qquad
 \overline v=\min_{x\in X}\sup_{y\in Y}f(x,y).
\]
For every $(x,y)$, $\min_{x'}f(x',y)\le f(x,y)\le\sup_{y'}f(x,y')$; taking the
outer extrema gives $\underline v\le\overline v$.

Fix $r>\underline v$ and, for $y\in Y$, define
$F_y=\{x\in X:f(x,y)\le r\}$.  Each $F_y$ is closed by lower semicontinuity and
convex by quasi-convexity.  We use the following finite intersection consequence of
Fan's KKM lemma \cite{FanKKM}: if $y_1,\ldots,y_N\in Y$ and
$\min_x f(x,y)<r$ for every $y$ in their convex hull, then
$\bigcap_{i=1}^NF_{y_i}\ne\varnothing$.  Indeed, quasi-concavity gives, for
$y=\sum_i\lambda_i y_i$,
\[
 f(x,y)\ge\min_{i:\lambda_i>0}f(x,y_i).
\]
If the intersection were empty, the KKM covering alternative applied to the closed
convex sublevel sets would produce an active convex combination $y$ for which every
$x$ has $f(x,y)>r$, contradicting $\min_xf(x,y)<r$.

Because $r>\underline v$, the premise holds for every $y\in Y$, hence for every
finite convex hull.  The family $(F_y)_{y\in Y}$ has the finite intersection
property.  Compactness of $X$ provides
$x_r\in\bigcap_{y\in Y}F_y$.  Consequently
$\sup_yf(x_r,y)\le r$, and therefore $\overline v\le r$.  Letting
$r\downarrow\underline v$ gives $\overline v\le\underline v$ and proves equality.
Finally, $x\mapsto\sup_yf(x,y)$ is lower semicontinuous as a supremum of
lower semicontinuous functions, so it has a minimizer on compact $X$.
\end{proof}

For Theorem~\ref{thm:compact-dual}, $E=L^2(\mu;\R^{d_y})$ carries its weak topology,
$X=\cM$, $F=\mathbb S^{d_y}\times C(K)$ carries the product of the Frobenius and
uniform topologies, and $Y=\mathcal B\times\cU_K$; the matrix $A$ is fixed before
Sion is applied. For Theorem~\ref{thm:ridge-dual}, $X=\mathcal C_\nu$ has the weak
$L^2$ topology and $Y=\mathcal B_2$ has the Frobenius topology. For
Theorem~\ref{thm:swap}, $X=\mathcal A_2$ and $Y=\mathcal B_2$ are finite-dimensional
compact convex sets. The proofs of those theorems verify the required sectionwise
properties in these specific topologies.

\section{Experimental configurations and evaluation metrics}
\label{app:experimental-configurations}

We specify here the reported synthetic quantities and the numerical settings
used in the synthetic and PBMC experiments. The symbols \(T_{\rm out}\),
\(T_{\rm in}\), and \(N_{\rm sk}\) denote the numbers of outer iterations,
inner mirror iterations per outer iteration, and Sinkhorn iterations allowed
per inner mirror step. In Table~\ref{tab:run-ab-config},
\(\eta=\eta_A=\eta_B\).

{\scriptsize
\renewcommand{\arraystretch}{1.08}
\begin{longtable}{@{}>{\raggedright\arraybackslash}p{.17\textwidth}
>{\raggedright\arraybackslash}p{.25\textwidth}
>{\raggedright\arraybackslash}p{.20\textwidth}
>{\raggedright\arraybackslash}p{.29\textwidth}@{}}
\caption{Definitions of the quantities reported in the synthetic experiments.
The plan \(P_{\rm cert}\) is the constructed martingale witness,
\((\widehat A,\widehat B,\widehat P)\) is the returned weak solver state, and
\(P_{\rm POT}\) and \(P_{\rm env}\) are returned by the matched start ordinary
solvers. Objective values and residuals are recomputed from the indicated plan
or state; compatibility and marginal residuals are numerical checks rather
than additional objectives.}
\label{tab:synthetic-reported-quantities}\\
\toprule
reported label & mathematical definition & evaluated plan or state &
interpretation\\
\midrule
\endfirsthead
\multicolumn{4}{@{}l}{\footnotesize\itshape Table~\thetable\ continued}\\
\toprule
reported label & mathematical definition & evaluated plan or state &
interpretation\\
\midrule
\endhead
\multicolumn{4}{@{}l}{\emph{Objectives and algebraic checks}}\\
clean rotation IGW &
\(J_{\rm IGW}\) with \(Y=Z\) and \(P=\operatorname{diag}(a)\) &
indexed coupling between the source and clean rotated skeleton &
sanity check that the rotation preserves the source Gram matrix\\
weak certificate &
\(J_{\rm w}(P_{\rm cert})\) &
explicit martingale coupling between each parent and its children &
floating point evaluation of the algebraic zero wIGW witness\\
weak solved primal &
\(J_{\rm w}(\widehat P)\) &
returned balanced weak coupling &
unregularized weak objective attained by the finite run\\
ridge solved value &
\(J_{\rm w,\varepsilon}(\widehat P)\) &
the same returned coupling &
ridge objective used to retain the best outer state, recomputed after repair\\
ordinary refinement plan &
\(J_{\rm IGW}(P_{\rm cert})\) &
the weak certificate coupling &
same plan ordinary comparison; not an ordinary optimality certificate\\
ordinary POT &
\(J_{\rm IGW}(P_{\rm POT})\) &
returned balanced coupling from the finite POT solve &
finite local ordinary IGW comparator\\
ordinary \(A\)/OT envelope &
\(J_{\rm IGW}(P_{\rm env})\) &
returned balanced coupling from the matched start envelope solve &
ordinary envelope formulation and implementation cross check\\
\midrule
\multicolumn{4}{@{}l}{\emph{Reconstruction checks}}\\
certificate pushforward error &
\(\max_{i,r}|m_i(P_{\rm cert})_r-z_{ir}|\) &
weak certificate coupling &
coordinatewise recovery error of the known clean skeleton\\
solved wIGW pushforward error (shapes); pushforward RMS relative to certificate
(graphs) &
\(\max_{i,r}|\widehat m_{ir}-z_{ir}|\) for shapes;
\(\bigl((nd_y)^{-1}\sum_i\)\newline
\(\norm{\widehat m_i-m_i(P_{\rm cert})}^2\)\newline
\(\bigr)^{1/2}\) for graphs &
returned balanced weak coupling &
maximum coordinate error for shapes and coordinatewise RMS error for graphs\\
\midrule
\multicolumn{4}{@{}l}{\emph{Numerical checks on the returned weak state}}\\
\(A\)-compatibility &
\(\norm{\widehat A-M_{\widehat m}}_F\) &
returned \((\widehat A,\widehat B,\widehat P)\) &
consistency of the retained cross moment variable\\
\(B\)-compatibility &
\(\norm{\widehat B-S_{\widehat m}}_F\) &
returned \((\widehat A,\widehat B,\widehat P)\) &
consistency of the retained second moment variable\\
row marginal residual &
\(\norm{\widehat P\mathbf 1-a}_\infty\) &
returned balanced weak coupling &
source marginal feasibility\\
column marginal residual &
\(\norm{\widehat P^\top\mathbf 1-b}_\infty\) &
returned balanced weak coupling &
target marginal feasibility\\
maximum marginal residual &
\(\max\{\norm{\widehat P\mathbf 1-a}_\infty,\)\newline
\(\norm{\widehat P^\top\mathbf 1-b}_\infty\}\) &
returned balanced weak coupling &
aggregate feasibility value reported in the shape results\\
\bottomrule
\end{longtable}
}

\begin{table}[H]
\centering
\footnotesize
\begin{tabularx}{\textwidth}{@{}>{\raggedright\arraybackslash}p{.27\textwidth}X@{}}
\toprule
parameter & value used in every synthetic weak solver run\\
\midrule
ridge & \(\varepsilon=10^{-4}\)\\
initial values &
\(A_{\rm init}=I_{d_x,d_y}\), \(B_{\rm init}=0.1I_{d_y}\), and
\(P_{\rm init}=ab^\top\). The code copies \(A_{\rm init}\), projects
\(B_{\rm init}\) onto the positive semidefinite cone, and begins the outer ball
projections after the first inner solve.\\
mirror step &
\(\gamma=0.5\), with gradient divisor
\(\max\{\norm{G}_\infty,10^{-12}\}\).\\
outer domains &
\(R_A=\sqrt{M_2(\mu)M_2(\nu)}\) and \(R_B=M_2(\nu)\), with Euclidean
projection onto the corresponding Frobenius balls and the positive semidefinite
constraint for \(B\).\\
inner warm start and selection &
The coupling retained by one outer iteration initializes the next inner solve.
Within an inner solve, the warm start and all mirror iterates are compared using
\(f_{A,B}\), and the lowest value is retained.\\
inner Sinkhorn step &
Tolerance \(10^{-12}\), checked every ten iterations; kernel floor
\(10^{-300}\); iteration cap \(N_{\rm sk}\) from
Table~\ref{tab:run-ab-config}. Reaching this cap is allowed inside the inexact
oracle.\\
outer retention &
The tuple \((A,B,P)\) before the outer update with the smallest direct ridge primal
\(J_{\rm w,\varepsilon}(P)\) is retained.\\
pre-update monitoring &
At each outer iteration, the implementation records the direct ridge primal, the
fixed-plan Lagrangian, the \(A\)- and \(B\)-compatibility residuals, and the two
marginal residuals before updating the outer matrices. These records form the
diagnostic history and do not alter the iterates or the return rule.\\
final balancing &
Tolerance \(10^{-12}\), checked every ten iterations, with at most
\(20\,000\) iterations. After scaling, the repair downscales overfull rows and
columns. For the remaining nonnegative deficits
\(r=a-P\mathbf 1\), \(c=b-P^\top\mathbf 1\), and their common mass
\(\delta=\mathbf 1^\top r=\mathbf 1^\top c>0\), it adds
\(rc^\top/\delta\). The objective and diagnostics are recomputed using the
repaired coupling, and the \(L^1\) change from the retained pre-balance
coupling to the returned coupling is recorded.\\
restarts &
One deterministic start and no random restarts. Every sweep level starts a new
solver run from the common initial values above.\\
\bottomrule
\end{tabularx}
\caption{Common configuration of the synthetic \(A\)--\(B\) runs.}
\label{tab:common-ab-config}
\end{table}

\begin{table}[H]
\centering
\footnotesize
\renewcommand{\arraystretch}{1.08}
\begin{tabularx}{\textwidth}{@{}>{\raggedright\arraybackslash}Xcc@{\hspace{.8em}}c@{\hspace{.8em}}cccc@{}}
\toprule
run & \(n\) & \(q\) & data seeds & \(T_{\rm out}\) & \(T_{\rm in}\) &
\(N_{\rm sk}\) & \(\eta\)\\
\midrule
symmetric cat & 500 & 2 & 4, 7 & 30 & 60 & 60 & 0.20\\
homoscedastic cat sweep & 500 & 2 & 4, 30 & 24 & 50 & 60 & 0.20\\
heteroscedastic cat sweep & 500 & 2 & 4, 60 & 24 & 50 & 60 & 0.20\\
Gaussian cat & 250 & 8 & 11, 91 & 35 & 70 & 70 & 0.20\\
graph instance & 10 & 10 & 4 & 45 & 80 & 300 & 0.18\\
\bottomrule
\end{tabularx}
\caption{Parameters that vary among runs of Algorithm~\ref{alg:implemented-ab}.
Here \(n\) is the number of source points
and \(q\) is the number of target children per source point. Cat entries list the
source sampling seed followed by the refinement seed. The graph source is
deterministic, so its entry gives only the refinement seed.}
\label{tab:run-ab-config}
\end{table}

\begin{table}[H]
\centering
\small
\begin{tabular}{@{}lrr@{}}
\toprule
run & maximum iterations & relative tolerance\\
\midrule
symmetric cat & 80 & \(10^{-7}\)\\
homoscedastic cat sweep & 80 & \(10^{-7}\)\\
heteroscedastic cat sweep & 80 & \(10^{-7}\)\\
Gaussian cat & 80 & \(10^{-7}\)\\
graph instance & 100 & \(10^{-8}\)\\
graph sweep & 80 & \(10^{-8}\)\\
\bottomrule
\end{tabular}
\caption{Synthetic ordinary IGW solver parameters. POT receives the Gram
matrices \(XX^\top\) and \(YY^\top\), squared loss, symmetric mode, no Armijo
line search, and $G_0=P_{\rm cert}$. The $A$/OT envelope receives
$P_0=P_{\rm cert}$ and $A_0=X^\top P_{\rm cert}Y$ and uses
\texttt{ot.emd} for each linear OT block. Both methods use the tabulated cap and
relative tolerance, absolute tolerance $10^{-9}$, no restarts, and final
balancing tolerance $10^{-12}$ with at most $20\,000$ iterations. The envelope
also uses compatibility tolerance $10^{-9}$, OT tie tolerance $10^{-12}$, and
at most $100\,000$ EMD iterations; its saved diagnostics include the iteration
count, both marginal residuals, and the final Frank--Wolfe gap. The graph sweep
reuses the seed 11 refinement offsets at all eight noise levels and rescales
them with the noise amplitude. Every reported value is recomputed from the
returned balanced plan.}
\label{tab:synthetic-ordinary-config}
\end{table}

{\footnotesize
\begin{longtable}{@{}>{\raggedright\arraybackslash}p{.24\textwidth}
>{\raggedright\arraybackslash}p{.70\textwidth}@{}}
\caption{PBMC data preparation and frozen split construction.}
\label{tab:pbmc-data-config}\\
\toprule
component & PBMC configuration\\
\midrule
\endfirsthead
\multicolumn{2}{@{}l}{\footnotesize\itshape Table~\thetable\ continued}\\
\toprule
component & PBMC configuration\\
\midrule
\endhead
dataset &
10x Genomics PBMC from one healthy donor, paired RNA and ATAC,
processed by Cell Ranger ARC 1.0.0 \cite{TenXPBMC}.\\
quality control &
\(500\le\) detected genes \(\le6000\); at least \(1000\) gene expression
counts; mitochondrial fraction at most \(0.15\); at least \(3000\) ATAC
fragments; fraction of reads in peaks (FRIP) at least \(0.15\). The filter
retains \(8212\) of \(12016\) cells. The available barcode metadata do not
include transcription start site (TSS) enrichment, nucleosome signal, or
blacklist fraction, so these criteria are not evaluated.\\
annotation &
CellTypist 1.7.1 with the pinned \texttt{Immune\_All\_Low.pkl} model,
followed by a fixed mapping to six classes; \(7760\) cells enter the eligible
pool.\\
subsamples &
Seeds \(17,23,31,47,59\). Each split has a \(900\)-cell atlas
(\(150\) per class) and a disjoint \(480\)-cell evaluation set
(\(80\) per class).\\
RNA transform &
Normalize each library to \(10^4\), apply \(\log(1+x)\), select \(2000\)
highly variable genes on the atlas, and fit an eight-component PCA on the
atlas.\\
ATAC transform &
Retain peaks present in at least ten atlas cells; fit TF--IDF and latent
semantic indexing with nine components on the atlas; discard the first
component.\\
centering and scaling &
Subtract the modality specific atlas mean and divide by its root mean square
radius. Apply each fitted transformation unchanged to the evaluation cells.\\
use of labels and pair identities &
Source atlas labels construct the six RNA prototypes, and source evaluation
labels group transported mass by class. Target labels determine the balanced
subsamples and score class transfer. Physical RNA/ATAC pair identities are
withheld from the solvers and used for retrieval scoring.\\
\bottomrule
\end{longtable}
}

\begin{table}[H]
\centering
\footnotesize
\begin{tabularx}{\textwidth}{@{}>{\raggedright\arraybackslash}p{.24\textwidth}X@{}}
\toprule
component & PBMC solver configuration\\
\midrule
support sizes &
Cell to cell: \(480\) RNA and \(480\) ATAC cells. Prototype to cell:
\(6\) RNA prototypes and \(480\) ATAC cells. All empirical masses are uniform.\\
POT IGW &
Product coupling start; at most \(100\) iterations; relative tolerance
\(10^{-9}\), absolute tolerance \(10^{-11}\), squared loss on \(XX^\top\)
and \(YY^\top\).\\
IGW envelope &
Scaled rectangular identity \(A\); at most \(100\) outer updates; exact
linear OT oracle with at most \(100\,000\) iterations; relative tolerance
\(10^{-9}\), absolute tolerance \(10^{-11}\).\\
barycentric wIGW &
\(\varepsilon=10^{-4}\), \(T_{\rm out}=100\), \(T_{\rm in}=50\),
\(\eta_A=\eta_B=0.12\), and inner mirror step \(\gamma=0.5\).
The Sinkhorn cap \(N_{\rm sk}\) is \(80\) for cell to cell and \(160\) for
prototype to cell.\\
wIGW initialization &
\(P_{\rm init}=ab^\top\); \(A_{\rm init}=I_{d_x,d_y}\) scaled to \(R_A\);
\(B_{\rm init}\) is the projection of \(0.1I_{d_y}\) onto
\(\mathcal B_2\). One deterministic start and no selection based on labels.\\
wIGW feasibility &
Marginal scaling tolerance \(10^{-10}\), followed by the exact nonnegative
marginal repair. The maximum change induced by final balancing and repair over the
\(70\)-run ridge sweep is
\(2.24\cdot10^{-15}\) in \(L^1\).\\
target \(k\)-means &
\(k=6\), \(20\) initializations, at most \(300\) iterations. An optimal
one-to-one linear assignment (Hungarian matching) maximizing total overlap maps
clusters to classes before macro-F1 is computed.\\
SCOT &
Pinned commit \texttt{14649be6e14017dcfe7ba619091b33d1df55f6a9};
\(k=50\), correlation graph, \(\epsilon_{\rm SCOT}=10^{-3}\), uniform
marginals, and no normalization. For the prototype to cell comparison, SCOT aligns
the \(900\) atlas RNA cells with the \(480\) ATAC cells and then aggregates
source rows by class.\\
ridge sensitivity &
\(\varepsilon\in\{10^{-6},10^{-5},10^{-4},10^{-3},10^{-2},10^{-1},1\}\);
all other wIGW settings remain fixed.\\
\bottomrule
\end{tabularx}
\caption{PBMC baselines, wIGW solver, and ridge sensitivity configuration.}
\label{tab:pbmc-solver-config}
\end{table}

\subsection{PBMC runtime and local solver robustness}

\begin{table}[H]
\centering
\footnotesize
\renewcommand{\arraystretch}{1.08}
\begin{tabular}{@{}lccc@{}}
\toprule
method & cell to cell time (s) & prototype to cell time (s) &
maximum marginal residual\\
\midrule
POT IGW, product start
 & \(0.152\pm0.079\) & \(0.0046\pm0.0009\) & \(2.22\cdot10^{-16}\)\\
IGW envelope, scaled identity
 & \(0.088\pm0.026\) & \(0.0014\pm0.0001\) & \(3.33\cdot10^{-16}\)\\
barycentric wIGW, \(\varepsilon=10^{-4}\)
 & \(46.314\pm1.354\) & \(4.340\pm0.075\) & \(5.55\cdot10^{-17}\)\\
target \(k\)-means
 & \(0.015\pm0.009\) & \(0.0147\pm0.0009\) & ---\\
SCOT
 & \(1.292\pm0.512\) & \(3.054\pm0.824\) & \(2.52\cdot10^{-5}\)\\
\bottomrule
\end{tabular}
\caption{Method runtime and feasibility in the PBMC comparison. Times are means
and sample standard deviations over the five frozen splits and measure the method
call after preprocessing. The method specific budgets are not matched by
wall-clock time; this table documents observed cost rather than a hardware-normalized
efficiency comparison.}
\label{tab:pbmc-runtime}
\end{table}

\begin{figure}[H]
\centering
\includegraphics[width=.86\linewidth]
 {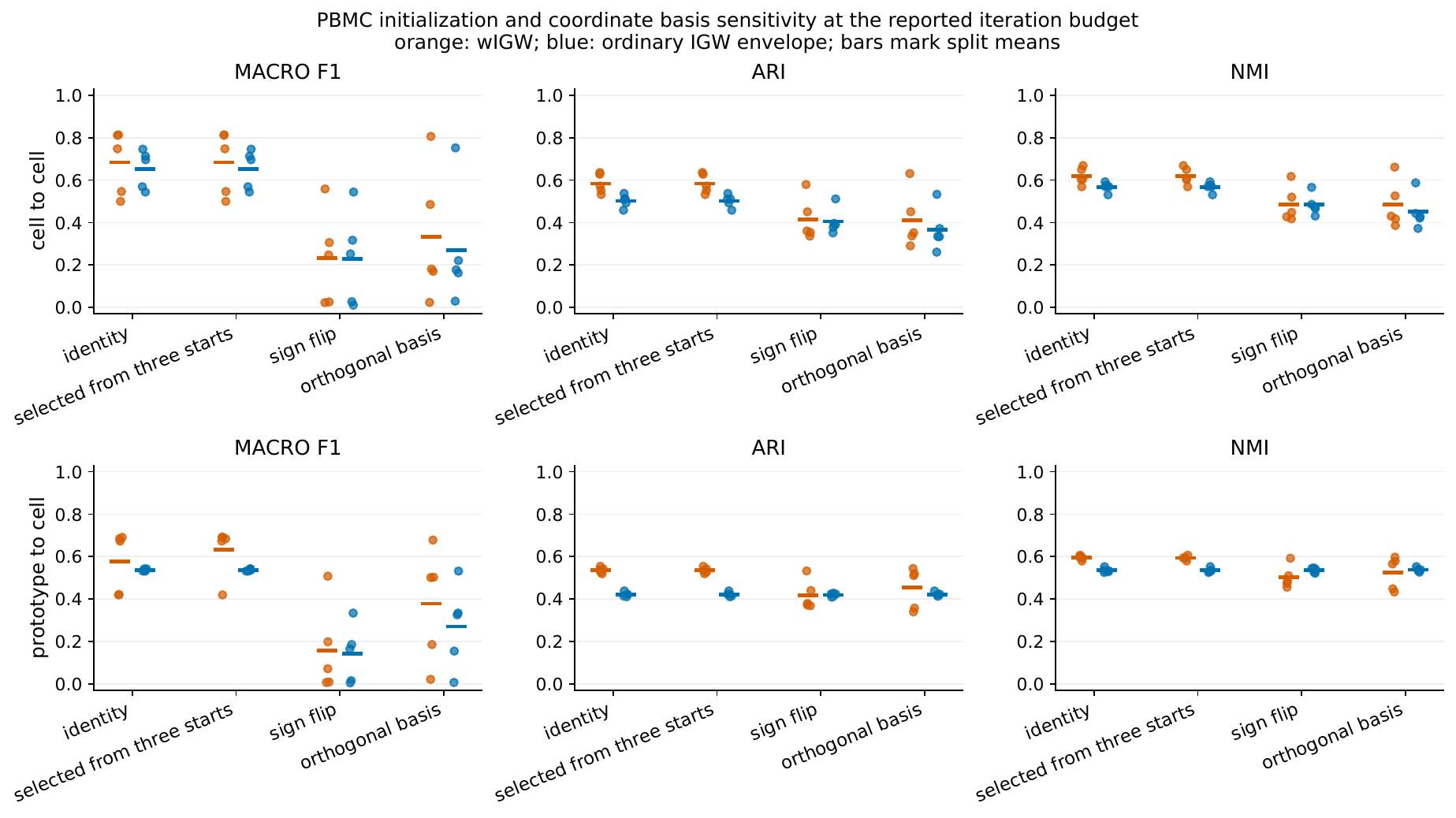}
\caption{Post hoc PBMC audit of initialization and basis sensitivity at the reported
iteration budget. Rows show the cell to cell and prototype to cell tasks, and columns show
macro-F1, ARI, and NMI. Orange denotes wIGW and blue denotes the IGW envelope;
points are frozen splits and horizontal bars are means. ``Selected from three starts'' uses
selection by each method's own objective among the scaled identity and two deterministic
starts based on SVD frames. The sign flips and orthogonal changes of basis use independent
coordinate changes in the two modalities followed by a newly constructed scaled identity start.}
\label{fig:pbmc-basis-audit}
\end{figure}

\begin{table}[H]
\centering
\footnotesize
\renewcommand{\arraystretch}{1.08}
\begin{tabular}{@{}llccc@{}}
\toprule
task & method & macro-F1 & ARI & NMI\\
\midrule
cell to cell
 & POT IGW, product start
 & \(0.368\pm0.267\) & \(0.389\pm0.115\) & \(0.476\pm0.092\)\\
 & IGW envelope, scaled identity
 & \(0.569\pm0.067\) & \(0.499\pm0.026\) & \(0.568\pm0.024\)\\
 & barycentric wIGW, \(\varepsilon=10^{-4}\)
 & \(0.577\pm0.111\) & \(\mathbf{0.587\pm0.029}\) &
   \(\mathbf{0.631\pm0.031}\)\\
 & target \(k\)-means
 & \(\mathbf{0.580\pm0.018}\) & \(0.461\pm0.017\) & \(0.563\pm0.019\)\\
\addlinespace[.2em]
prototype to cell
 & POT IGW, product start
 & \(0.239\pm0.065\) & \(0.412\pm0.019\) & \(0.526\pm0.022\)\\
 & IGW envelope, scaled identity
 & \(0.553\pm0.023\) & \(0.424\pm0.012\) & \(0.539\pm0.019\)\\
 & barycentric wIGW, \(\varepsilon=10^{-4}\)
 & \(0.500\pm0.123\) & \(\mathbf{0.556\pm0.022}\) &
   \(\mathbf{0.620\pm0.020}\)\\
 & target \(k\)-means
 & \(\mathbf{0.580\pm0.018}\) & \(0.461\pm0.017\) & \(0.563\pm0.019\)\\
\bottomrule
\end{tabular}
\par
\caption{Secondary validation with a fixed protocol on five fresh subsampling seeds
from the same donor, numbered 71, 73, 79, 83, and 89. The ridge value, iteration
budgets, preprocessing, and
method specific initializations---including the scaled identity for wIGW and the
IGW envelope---were fixed before these splits were run. Entries are means and
sample standard deviations. This audit does not provide independent donor validation.}
\label{tab:pbmc-postselection-validation}
\end{table}

The post hoc robustness audit uses the same \(100\times50\) wIGW budget and
the same \(100\)-iteration budget for the IGW envelope as the main comparison. It
selects among the scaled identity and two deterministic starts based on SVD frames
using only each method's own objective. Independent sign flips and orthogonal changes
of basis reconstruct the scaled
identity in each transformed coordinate system. Labels and pair identities do not
enter any run or selection. The covariant controls keep the canonical plan and
transform the returned matrices; their maximum plan difference is zero and their
maximum objective difference is \(2.02\cdot10^{-16}\).
\par\Needspace{4\baselineskip}\noindent
Among the wIGW states selected by objective, the maximum compatibility residuals for
\(A\) and \(B\) are
\(1.42\cdot10^{-4}\) and \(2.07\cdot10^{-4}\), respectively; the maximum projected
residual of the returned state is \(5.90\cdot10^{-4}\), and the maximum change in the running best over the
final ten iterations is \(8.29\cdot10^{-4}\). Because Algorithm~\ref{alg:implemented-ab}
returns a best tuple retained before the update, these values are numerical diagnostics
rather than a saddle point or convergence certificate.

\FloatBarrier

\subsection{PBMC evaluation metrics}\label{subsec:pbmc-evaluation-metrics}

Let \(z_i^X\in\{1,\ldots,C\}\) and \(z_j^Y\in\{1,\ldots,C\}\), with
\(C=6\), denote the source and target cell type labels. For a transport plan
\(P\), the class mass delivered to target cell \(j\) and its predicted label are
\[
 q_{cj}:=\sum_{i:z_i^X=c}P_{ij},
 \qquad
 \widehat z_j\in\arg\max_{1\le c\le C}q_{cj}.
\]
An exact tie is resolved by the fixed class order. Target labels evaluate
\(\widehat z\) and do not enter the transport optimization.

Let
\(n_{c\ell}:=\#\{j:z_j^Y=c,\widehat z_j=\ell\}\),
\(a_c:=\sum_\ell n_{c\ell}\),
\(b_\ell:=\sum_c n_{c\ell}\), and
\(n:=\sum_{c,\ell}n_{c\ell}\). The macro-F1 score is
\[
 \operatorname{macroF1}
 :=\frac1C\sum_{c=1}^C
 \frac{2n_{cc}}
 {2n_{cc}+\sum_{r\ne c}n_{rc}+\sum_{\ell\ne c}n_{c\ell}},
\]
with a zero contribution when the denominator vanishes. To define ARI, put
\[
 S:=\sum_{c,\ell}\binom{n_{c\ell}}2,
 \qquad A:=\sum_c\binom{a_c}2,
 \qquad B:=\sum_\ell\binom{b_\ell}2,
 \qquad T:=\binom n2.
\]
Then
\[
 \operatorname{ARI}
 :=\frac{S-AB/T}{\tfrac12(A+B)-AB/T}.
\]
For \(p_{c\ell}:=n_{c\ell}/n\), \(p_c:=a_c/n\), and
\(\widehat p_\ell:=b_\ell/n\), define
\[
 I:=\sum_{c,\ell:p_{c\ell}>0}p_{c\ell}
       \log\frac{p_{c\ell}}{p_c\widehat p_\ell},
 \quad
 H:=-\sum_c p_c\log p_c,
 \quad
 \widehat H:=-\sum_\ell\widehat p_\ell\log\widehat p_\ell.
\]
The reported normalized mutual information uses arithmetic normalization,
\[
 \operatorname{NMI}:=\frac{2I}{H+\widehat H}.
\]
The logarithm base does not affect this ratio.

For the target only \(k\)-means reference, let \(g_j\) be the cluster of target
cell \(j\) and let
\(h_{rc}:=\#\{j:g_j=r,z_j^Y=c\}\). After clustering, we choose
\[
 \sigma^\star\in\arg\max_{\sigma\in\mathfrak S_C}
 \sum_{r=1}^C h_{r,\sigma(r)}
\]
and set \(\widehat z_j=\sigma^\star(g_j)\). This optimal assignment, commonly
called Hungarian matching, uses target labels only to name clusters for
macro-F1. When the maximizer is not unique, we use the optimum returned by
SciPy's \texttt{linear\_sum\_assignment} applied to \(-h\). The matching does
not change the clustering, ARI, or NMI.

\paragraph{Paired cell retrieval.}
Paired cell retrieval is evaluated only in the cell to cell task. For target
cell \(j\), let \(r_j\) be the rank of its physically paired RNA cell when
source cells are ordered by decreasing \(P_{ij}\). The implementation uses the
fixed source order to break ties. We report
\[
 \operatorname{top}\text{-}k
 :=\frac1n\sum_{j=1}^n\mathbf 1\{r_j\le k\},
 \quad k\in\{1,5\},
 \qquad
 \operatorname{MRR}:=\frac1n\sum_{j=1}^n\frac1{r_j}.
\]
Pair identities are withheld during optimization. Retrieval is undefined for
target \(k\)-means and for the prototype to cell task. Each metric is
computed separately on the five frozen splits. Table~\ref{tab:pbmc-pair-retrieval}
reports the empirical means and sample standard deviations; its row for random ranking
gives the exact expectation for \(n=480\).

\begin{table}[H]
\centering
\footnotesize
\renewcommand{\arraystretch}{1.08}
\begin{tabular}{@{}lccc@{}}
\toprule
method & top-1 & top-5 & mean reciprocal rank\\
\midrule
random ranking (expectation)
 & \(0.0021\) & \(0.0104\) & \(0.0141\)\\
\addlinespace[.2em]
POT IGW, product start
 & \(0.026\pm0.025\) & \(0.033\pm0.024\) & \(0.037\pm0.025\)\\
IGW envelope, scaled identity
 & \(0.056\pm0.012\) & \(0.064\pm0.011\) & \(0.067\pm0.012\)\\
barycentric wIGW, \(\varepsilon=10^{-4}\)
 & \(\mathbf{0.068\pm0.017}\) & \(\mathbf{0.203\pm0.038}\) &
   \(\mathbf{0.146\pm0.027}\)\\
SCOT
 & \(0.027\pm0.015\) & \(0.106\pm0.068\) & \(0.084\pm0.044\)\\
\bottomrule
\end{tabular}
\caption{Paired cell retrieval in the cell to cell PBMC task. For each ATAC
cell, the \(480\) RNA cells are ranked by transported mass. Empirical entries
are means and sample standard deviations over the five frozen splits. Under a
uniformly random ranking, the exact expectations are \(1/480\) for top-1,
\(5/480\) for top-5, and \(H_{480}/480\) for mean reciprocal rank, where
\(H_{480}=\sum_{r=1}^{480}r^{-1}\). Target \(k\)-means has no cross-modal
coupling, and the prototype to cell task has no unique paired source cell, so
retrieval is not defined in those cases.}
\label{tab:pbmc-pair-retrieval}
\end{table}


\small
\begin{thebibliography}{99}

\bibitem{Memoli} F. M\'emoli. Gromov--Wasserstein distances and the metric approach to
object matching. \emph{Foundations of Computational Mathematics}, 11:417--487, 2011.

\bibitem{PeyreCuturiSolomon} G. Peyr\'e, M. Cuturi, and J. Solomon.
Gromov--Wasserstein averaging of kernel and distance matrices. In \emph{ICML}, 2016.

\bibitem{ZhangGoldfeldMrouehSriperumbudur} Z. Zhang, Z. Goldfeld, Y. Mroueh, and
B. K. Sriperumbudur. Gromov--Wasserstein distances: entropic regularization,
duality and sample complexity. \emph{Annals of Statistics}, 52(4):1616--1645,
2024.

\bibitem{OreshkovCalsamiglia2009} O. Oreshkov and J. Calsamiglia.
Distinguishability measures between ensembles of quantum states.
\emph{Physical Review A}, 79:032336, 2009.
doi:10.1103/PhysRevA.79.032336.

\bibitem{LeppajarviSedlak2021} L. Lepp\"aj\"arvi and M. Sedl\'ak.
Post-processing of quantum instruments. \emph{Physical Review A},
103:022615, 2021. doi:10.1103/PhysRevA.103.022615.

\bibitem{GRST} N. Gozlan, C. Roberto, P.-M. Samson, and P. Tetali. Kantorovich
duality for general transport costs and applications. \emph{Journal of Functional
Analysis}, 273(11):3327--3405, 2017.

\bibitem{BBP} J. Backhoff-Veraguas, M. Beiglb\"ock, and G. Pammer. Existence,
duality, and cyclical monotonicity for weak transport costs. \emph{Calculus of
Variations and Partial Differential Equations}, 58:203, 2019.

\bibitem{Strassen} V. Strassen. The existence of probability measures with given
marginals. \emph{Annals of Mathematical Statistics}, 36:423--439, 1965.

\bibitem{GozlanJuillet} N. Gozlan and N. Juillet. On a mixture of Brenier and Strassen
theorems. \emph{Proceedings of the London Mathematical Society},
120(3):434--463, 2020. doi:10.1112/plms.12302.

\bibitem{DomingoEnrichSchiffMroueh2023} C. Domingo-Enrich, Y. Schiff, and Y. Mroueh.
Learning with stochastic orders. In \emph{International Conference on Learning
Representations (ICLR)}, 2023. arXiv:2205.13684.
\url{https://arxiv.org/abs/2205.13684}.

\bibitem{TenXPBMC} 10x Genomics. \emph{PBMC from a Healthy Donor---No Cell
Sorting (10k)}. Epi Multiome dataset analyzed using Cell Ranger ARC 1.0.0,
2020.
\href{https://www.10xgenomics.com/datasets/pbmc-from-a-healthy-donor-no-cell-sorting-10-k-1-standard-1-0-0}
{10x Genomics data page},
accessed August 23, 2026.

\bibitem{IBMUSD} IBM. \emph{Unbalanced Sobolev Descent}. Source code
repository for the NeurIPS 2020 paper; the experiments use assets from its
\texttt{img/} directory. Archived GitHub repository,
\url{https://github.com/IBM/USD}, accessed August 19, 2026.

\bibitem{PatyChoneKramarz2022} F.-P. Paty, P. Chon\'e, and F. Kramarz. Algorithms
for weak optimal transport with an application to economics. arXiv:2205.09825, 2022.

\bibitem{RiouxGoldfeldKato2024} G. Rioux, Z. Goldfeld, and K. Kato. Entropic
Gromov--Wasserstein distances: Stability and algorithms.
\emph{Journal of Machine Learning Research}, 25(363):1--52, 2024.

\bibitem{DemetciEtAl2022} P. Demetci, R. Santorella, B. Sandstede,
W. S. Noble, and R. Singh. SCOT: Single-cell multi-omics alignment with optimal
transport. \emph{Journal of Computational Biology}, 29(1):3--18, 2022.
doi:10.1089/cmb.2021.0446.

\bibitem{ChungSongKimPark} J. Chung, E. Song, W. H. Kim, and G. Park.
Convex distance operator transport: A convex and geometry-preserving formulation.
In \emph{Proceedings of the 43rd International Conference on Machine Learning},
PMLR 306, 2026.

\bibitem{WangWangDing2026} H.-H. Wang, Y. Wang, and H. Ding.
MIRROR: Aligning semantic relations from language to image via
Gromov--Wasserstein. To appear in \emph{European Conference on Computer Vision},
2026. arXiv:2606.29462.

\bibitem{ChenLimMemoliWanWang} S. Chen, S. Lim, F. M\'emoli, Z. Wan, and Y. Wang.
Weisfeiler--Lehman meets Gromov--Wasserstein. In \emph{Proceedings of the 39th
International Conference on Machine Learning}, PMLR 162:3371--3416, 2022.

\bibitem{BauerMemoliNeedhamNishino} M. Bauer, F. M\'emoli, T. Needham, and
M. Nishino. The $Z$-Gromov--Wasserstein distance.
\emph{Journal of Machine Learning Research}, 26(291):1--57, 2025.

\bibitem{VincentCuazEtAl} C. Vincent-Cuaz, R. Flamary, M. Corneli, T. Vayer, and
N. Courty. Semi-relaxed Gromov--Wasserstein divergence and applications on graphs.
In \emph{International Conference on Learning Representations}, 2022.

\bibitem{BeierBeinertSteidl} F. Beier, R. Beinert, and G. Steidl. On a linear
Gromov--Wasserstein distance. \emph{IEEE Transactions on Image Processing},
31:7292--7305, 2022.

\bibitem{Kallenberg} O. Kallenberg. \emph{Foundations of Modern Probability}.
Third edition, Probability Theory and Stochastic Modelling 99, Springer, Cham, 2021.

\bibitem{ShakedShanthikumar2007} M. Shaked and J. G. Shanthikumar.
\emph{Stochastic Orders}. Springer Series in Statistics, Springer, New York, 2007.

\bibitem{Danskin1966} J. M. Danskin. The theory of max-min, with applications.
\emph{SIAM Journal on Applied Mathematics}, 14(4):641--664, 1966.
doi:10.1137/0114053.

\bibitem{POT} R. Flamary et al. POT: Python Optimal Transport.
\emph{Journal of Machine Learning Research}, 22(78):1--8, 2021.

\bibitem{DominguezCondeEtAl2022} C. Dom\'inguez Conde et al.
Cross-tissue immune cell analysis reveals tissue-specific features in humans.
\emph{Science}, 376(6594):eabl5197, 2022.
doi:10.1126/science.abl5197.

\bibitem{CaoGongHongWan2022} K. Cao, Q. Gong, Y. Hong, and L. Wan.
A unified computational framework for single-cell data integration with optimal
transport. \emph{Nature Communications}, 13:7419, 2022.
doi:10.1038/s41467-022-35094-8.

\bibitem{XuBegoliMcCord2022} Y. Xu, E. Begoli, and R. P. McCord.
sciCAN: Single-cell chromatin accessibility and gene expression data integration
via cycle-consistent adversarial network.
\emph{npj Systems Biology and Applications}, 8:33, 2022.
doi:10.1038/s41540-022-00245-6.

\bibitem{BrezisFA} H. Brezis. \emph{Functional Analysis, Sobolev Spaces and Partial
Differential Equations}. Universitext, Springer, New York, 2011.
doi:10.1007/978-0-387-70914-7.

\bibitem{Billingsley} P. Billingsley. \emph{Convergence of Probability Measures}.
Second edition, Wiley, New York, 1999.

\bibitem{VillaniOT} C. Villani. \emph{Optimal Transport: Old and New}.
Grundlehren der mathematischen Wissenschaften 338, Springer, Berlin, 2009.

\bibitem{BertsekasShreve} D. P. Bertsekas and S. E. Shreve.
\emph{Stochastic Optimal Control: The Discrete-Time Case}.
Mathematics in Science and Engineering 139, Academic Press, New York, 1978.

\bibitem{Rockafellar} R. T. Rockafellar. \emph{Convex Analysis}.
Princeton Mathematical Series 28, Princeton University Press, Princeton, 1970.

\bibitem{Sion} M. Sion. On general minimax theorems.
\emph{Pacific Journal of Mathematics}, 8(1):171--176, 1958.

\bibitem{AliprantisBorder} C. D. Aliprantis and K. C. Border.
\emph{Infinite Dimensional Analysis: A Hitchhiker's Guide}.
Third edition, Springer, Berlin, 2006.

\bibitem{PeyreCuturiOT} G. Peyr\'e and M. Cuturi. Computational optimal transport.
\emph{Foundations and Trends in Machine Learning}, 11(5--6):355--607, 2019.

\bibitem{BauschkeCombettes} H. H. Bauschke and P. L. Combettes.
\emph{Convex Analysis and Monotone Operator Theory in Hilbert Spaces}.
Second edition, CMS Books in Mathematics, Springer, Cham, 2017.

\bibitem{GuoNilssonWiesel2025} I. Guo, S. Nilsson, and J. Wiesel.
Dynamic characterization of barycentric optimal transport problems and their
martingale relaxation. arXiv:2511.21287, 2025.

\bibitem{PflugPichler2012} G. Ch. Pflug and A. Pichler.
A distance for multistage stochastic optimization models.
\emph{SIAM Journal on Optimization}, 22(1):1--23, 2012.

\bibitem{BackhoffBartlBeiglbockEder2020} J. Backhoff-Veraguas, D. Bartl,
M. Beiglb\"ock, and M. Eder. Adapted Wasserstein distances and stability in
mathematical finance. \emph{Finance and Stochastics}, 24:601--632, 2020.

\bibitem{QuTran2021} Z. Qu and B. Tran. Entropic regularization of the nested
distance. arXiv:2107.09864, 2021.

\bibitem{MuandetEtAl} K. Muandet, K. Fukumizu, B. Sriperumbudur, and
B. Sch\"olkopf. Kernel mean embedding of distributions: A review and beyond.
\emph{Foundations and Trends in Machine Learning}, 10(1--2):1--141, 2017.

\bibitem{ParkMuandet} J. Park and K. Muandet. A measure-theoretic approach to
kernel conditional mean embeddings. In \emph{Advances in Neural Information
Processing Systems}, 33, 2020.

\bibitem{FanKKM} K. Fan. A generalization of Tychonoff's fixed point theorem.
\emph{Mathematische Annalen}, 142:305--310, 1961.

\bibitem{Kellerer} H. G. Kellerer. Markov-Komposition und eine Anwendung auf
Martingale. \emph{Mathematische Annalen}, 198:99--122, 1972.

\end{thebibliography}
\end{document}